\documentclass[12pt, reqno]{amsart}
\usepackage{verbatim}
\usepackage{graphicx, epsfig, psfrag}
\usepackage{amsfonts,amsmath,amssymb,amsbsy,amsthm}
\usepackage{bm}
\usepackage{color}
\usepackage{float}
\usepackage{hyperref}
\usepackage{mathrsfs}
\usepackage{subfigure}
\usepackage{bbm}
\usepackage[normalem]{ulem}
\usepackage{enumitem}
\usepackage{placeins}

\newtheorem{assumption}{Assumption}

\numberwithin{equation}{section}

\usepackage{environments}

\usepackage[top=2.5cm,bottom=2.5cm,left=2cm,right=2cm]{geometry}
\renewcommand{\paragraph}[1]{\subsubsection{#1}}

\def\transpose{{\hspace{-1pt}\top}}

\def\R{\mathbb{R}}

\def\dx{\,{\rm d}x}

\def\Sad{\,{\rm S}^{\rm a}_{\rm d}}
\def\Sas{\,{\rm S}^{\rm a}_{\rm s}}
\def\Scd{\rm{S}^{\rm c}_{\rm d}}
\def\Scs{\rm{S}^{\rm c}_{\rm s}}
\def\<{\langle}
\def\>{\rangle}

\def\conv{{\rm conv}}

\def\mF{{\sf F}}

\def\mG{{\sf G}}

\def\a{{\rm a}}
\def\b{{\rm b}}
\def\c{{\rm c}}

\def\bqcf{{\rm bqcf}}
\newcommand{\weakto}{\rightharpoonup}

\def\nn{{\rm nn}}

\def\calL{\mathcal{L}}
\def\calR{\mathcal{R}}

\newcommand{\triple}{(\rho \alpha\beta)}
\newcommand{\tripleTau}{(\tau\gamma\delta)}

\newcommand{\omeRho}{\omega_{\rho}}
\newcommand{\barZeta}{\bar{\zeta}}

\newcommand{\ZZh}{Z_h}
\newcommand{\zzhalpha}{{z_h}_{\alpha}}
\newcommand{\zzhbeta}{{z_h}_{\beta}}
\newcommand{\la}{\langle}
\newcommand{\ra}{\rangle}

\definecolor{docol}{rgb}{0, 0, 0}
\definecolor{cocol}{rgb}{0,0, 0}
\definecolor{ascol}{rgb}{0, 0,0}
\newcommand{\helen}[1]{{\color{black}#1}}

\newcommand{\co}[1]{{\color{cocol} #1}}
\newcommand{\dao}[1]{{\color{docol} #1}}

\def\Xint#1{\mathchoice
{\XXint\displaystyle\textstyle{#1}}%
{\XXint\textstyle\scriptstyle{#1}}%
{\XXint\scriptstyle\scriptscriptstyle{#1}}%
{\XXint\scriptscriptstyle\scriptscriptstyle{#1}}%
\!\int}
\def\XXint#1#2#3{{\setbox0=\hbox{$#1{#2#3}{\int}$ }
\vcenter{\hbox{$#2#3$ }}\kern-.6\wd0}}

\def\dashint{\Xint-}
\title{Force-Based Atomistic/Continuum Blending for Multilattices}\thanks{DO was supported by the NSF PIRE Grant OISE-0967140 and NSF RTG program DMS-1344962.  XL was supported by the Simons Collaboration Grant with Award ID: 426935.  CO was supported by ERC Starting Grant 335120. }
\author{Derek Olson, Xingjie Li, Christoph Ortner, Brian Van Koten}

\begin{document}

\begin{abstract}
   We formulate the blended force-based quasicontinuum (BQCF) method for
   multilattices and develop rigorous error estimates in terms of the
   approximation parameters: atomistic region, blending region and continuum
   finite element mesh.  Balancing the approximation parameters yields a
   convergent atomistic/continuum multiscale method for multilattices with point
   defects, including a rigorous convergence rate in terms of the computational
   cost. The analysis is illustrated with numerical results for a Stone--Wales
   defect in graphene.
\end{abstract}

\maketitle

% We present an analysis of the blended force-based quasicontinuum (BQCF) method for point defects in multilattices and augment this analysis with numerical results for a Stone--Wales defect in graphene.  This analysis gives the first rigorous analysis for an atomistic-to-continuum method valid for defects in multilattices, which comprise those materials with more than one atom per unit cell. \commentco{probably want a more ``conceptual'' paragraph first - need to think about it}
% \end{abstract}

\section{Introduction}

A full twenty years has passed since the original proposal of the quasicontinuum method~\cite{ortiz1996quasicontinuum} captivated the materials science community with the potential to model material phenomena spanning vastly different length scales.  The quasicontinuum (QC) method was among the first of the so-called atomistic-to-continuum (AtC) coupling algorithms which sought to bridge the gap between length scales from the nano to macroscale.  A remarkable number of these AtC methods have been proposed since (see e.g.~\cite{tadmor2011,miller2009,acta.atc} for a thorough discussion of many of these), and recently a mathematical framework has begun to emerge to analyze and compare several of these methods for defects in crystalline materials comprised of a Bravais lattice.  Indeed, all three of the blended force-based quasicontinuum method (BQCF), blended energy-based quasicontinuum (BQCE), and blended ghost force correction (BGFC) methods have recently been analyzed in the context of a single defect in a two or three dimensional Bravais lattice~\cite{blended2014,OrtnerZhang2014bgfc} as has the optimization-based AtC approach of~\cite{olson2015}.  Analyses in two and three dimensional Bravais lattices also exist for the AtC method of~\cite{lu2013}, but this has not yet been extended to allow for defects.  Meanwhile, the methods~\cite{MakridakisMitsoudisRosakis2012,shapeev_2011,shapeev2012} have been shown to be consistent (or free of ghost forces) for pair potential interactions only.

In the present work, we resolve the long-standing challenge to develop a rigorous numerical analysis for AtC methods in the context of \textit{multilattices}, which allows for more than one atom to be present in the unit cell of the crystal. This description includes important materials such as hcp metals, diamond structures, and recently discovered 2D materials such as graphene and hexagonal boron nitride.

Concretely, we generalise the formulation and analysis of the blended force-based quasicontinuum (BQCF) method. Our main result is that, for a point defect in a homogeneous host crystal, the BQCF method for multilattices exhibits the same rate of convergence as in the Bravais lattice case. This is in sharp contrast with the blended energy-based quasicontinuum method for which a reduced convergence rate is expected in the multilattice setting \cite{OrtnerZhang2014bgfc}.

The present work represents the first analysis that has been undertaken that remains valid for an AtC method which permits defects in a two or three dimensional multilattice. Even analyses of AtC methods for defect-free multilattices remain extremely sparse:  the homogenized QC method~\cite{AbdulleLinShapeevII,AbdulleLinShapeev2012}, for example, only allows for dead load external forces while the cascading Cauchy--Born method was rigorously analyzed only in one-dimensional multilattices for phase transforming materials~\cite{dobson2007multilattice}.

As its name entails, the BQCF method is a force-based AtC method where a hybrid force operator is constructed instead of a hybrid energy functional~\cite{dobson_esaim,Shenoy:1999a,shilkrot2002coupled,lu2013,Bochev_08_MMS,Bochev_08_OUP}.  The primary advantage of force-based methods is that the forces can easily be defined in a way to avoid spurious interface effects (ghost forces); that is, the defect-free perfect crystal is a bona fide equilibrium configuration of the AtC force operator.  The cost of defining the BQCF method and other force-based methods to be free of ghost forces is that these force fields are no longer conservative, which creates significant challenges in their numerical analysis \cite{dobson2010sharp, lu2014}.  The blended force-based methods, originally studied in~\cite{li2012positive,Bochev_08_MMS,Bochev_08_OUP, lu2013}, seek to overcome this problem by a smooth blending between atomistic and continuum forces over a region called the blending, overlap, or handshake region.  Similar force-based blending methods have also been applied to coupling peridynamics with classical elasticity~\cite{seleson2013}.

An alternative to the force-based paradigm is the energy-based paradigm where a global, hybrid energy is defined which is some combination of atomistic and continuum energies.  This encompasses the original quasicontinuum method and many other offshoots and ancestors~\cite{ortiz1996quasicontinuum,xiao2004,abraham1998,E2006,shimokawa,datta2004,eidel2009,Bauman_08_CM}.  The peril of these methods is the aforementioned ghost forces, and it remains open to construct a general, ghost-force free, energy-based AtC method for Bravais lattices in two or three dimensions.  As such we do not concern ourselves with an energy-based AtC method for multilattices; however, see~\cite{OrtnerZhang2014bgfc,shapeevMulti} for promising directions.

\subsection{Outline}
We begin in Section~\ref{model} by formulating an atomistic model for a multilattice material describing a single point defect embedded in an infinite homogeneous crystal.  This is a canonical extension of the framework adopted for Bravais lattices in~\cite{olson2015,blended2014,bqcf13,Ehrlacher2013,OrtnerZhang2014bgfc}.

In Section~\ref{bqcf} we then formulate the BQCF method for this model and state our main results: (1) existence of a solution to the multilattice BQCF method and (2) a sharp error estimate. We also convert this error estimate to an estimate in terms of the computational complexity of the BQCF method in Section~\ref{num} which in particular allows us to balance approximation parameters to obtain a formulation optimised for the error / cost ratio. We present a numerical verification of these rates by testing the method on a Stone--Wales defect in graphene.  The complexity estimates obtained for the BQCF method for point defects in multilattices match those estimates in~\cite{blended2014} for Bravais lattices.

Finally, Section~\ref{analysis} covers the technical details needed to prove our main result, Theorem~\ref{main_thm}.  These technical details can be seen as generalizations of the results of Bravais lattices, and the primary new component is having to account for shifts between atoms in the same unit cell.

\subsection{Notation}
We introduce new notation throughout the paper required to carry out the analysis.  For the convenience of the reader, we have listed many of these in Appendix~\ref{sec:appnotation}.  Here, we briefly establish several basic conventions we make throughout.  We use $d$ and $n$ to denote the dimensions of the domain and range respectively, calligraphic fonts (e.g. $\mathcal{L}, \mathcal{M}$) to denote lattices, sans-serif fonts (e.g. $\mF, \mG$) for $n \times d$ matrices, the lower case Greek letters $\alpha, \beta, \gamma, \delta, \iota, \chi$ are used as subscripts denoting atomic species, and the lower case Greek letters $\rho, \tau, \sigma$ denote vectors (bond directions) between lattice sites.

The symbol $| \cdot |$ is used to denote the $\ell^2$ norm of a single vector in $\mathbb{R}^m$, while $\| \cdot\|$ is used to denote either an $\ell^p$ or $L^p$ norm over a specified set.  We use $\cdot$ for the dot product between two vectors, $\otimes$ as the tensor product, and $:$ as the inner product on tensors.

Derivatives of functions $f: \mathbb{R}^d \to \mathbb{R}^n$ are denoted by $\nabla f : \mathbb{R}^d \to \mathbb{R}^{d \times n}$ and higher order derivatives by $\nabla^j f$.  Given $F:X \to Y$ where $X$ and $Y$ are Banach spaces, we denote Fr\'echet or Gateaux derivatives by $\delta^j F$, $j$ indicating the order. We will most commonly interpret these derivatives as (multi-)linear forms and use them when $Y = \mathbb{R}$, in which case we will then write the Gateaux derivatives as
\begin{align*}
&\<\delta F(x), y\> , \quad x,y \in X\\
&\<\delta^2 F(x)z,y\>, \quad x,y,z \in X \quad \mbox{and so on for higher order derivatives.}
\end{align*}
We reserve $D$ for specific finite difference operators (defined in \eqref{finite_diff1} and \eqref{finite_diff2}), and use $B_R$ to denote the ball of radius $R$ about the origin.

We use the modified Vinogradov notation $x \lesssim y$ throughout the manuscript to mean there exists a positive constant $C$ such that $x \leq Cy$. Where appropriate, we clarify what the constant $C$ is allowed to depend on; in particular if there is any dependence on approximation parameters then it will always be made explicit.

\section{Atomistic Model}\label{model}

\subsection{Defect-free Multilattice}
% We first describe the defect-free multilattice and formulate the
% assumptions on the total potential energy of this multilattice.
We consider an infinite Bravais lattice, or simply a {\em lattice}, {$\mathcal{L}$, defined by
\[
\calL := \mF\mathbb{Z}^d, \quad \text{ for some } \mF \in \mathbb{R}^{d \times d}, \quad \det(\mF) = 1, \quad \mbox{and $d \in \{2,3\},$}
\]
where the requirement $\det(\mF) = 1$ is purely a notational convenience.  From a physical standpoint by taking symmetry into account, it can be shown that there are only 14 unique physical lattices in 3D and five in 2D (see e.g.~\cite{tadmor2011}); however, we consider the lattice to merely be a mathematical framework. A multilattice is then obtained by associating a basis of $S$ atoms to each lattice site, and this is also referred to as a crystal when the Bravais lattice is interpreted as one of the unique physical lattices.}

For each site $\xi \in \mathcal{L}$, these
$S$ atoms are located inside the unit cell of $\xi$ at positions $\xi +
p_\alpha^{\rm ref}$ for $p_\alpha^{\rm ref} \in \mathbb{R}^d$ and $\alpha = 0, \ldots, S-1$.  The
multilattice is then defined by
\[
\mathcal{M} := \bigcup_{\alpha = 0}^{S-1}\mathcal{L} + p_\alpha^{\rm ref}.
\]
We call each $\mathcal{L} + p_\alpha^{\rm ref}$ a sublattice; {{here the addition ``+'' means a translation of the lattice $\mathcal{L}$ by the
vector $p_{\alpha}^{\rm ref}$.}} Without loss of generality, we further assume $p_0^{\rm ref} = 0$ (one atom is always located at a lattice site).  Furthermore, we make the distinction between a lattice site, which we use to refer to a site in the Bravais lattice, $\mathcal{L}$, and an atom which is an element in the multilattice $\mathcal{M}$.

Two simple examples of multilattices are shown in Figure~\ref{fig:multilattice}
including the 2D hexagonal lattice (e.g., graphene) for which
\begin{equation}\label{graph_param}
\mathcal{L} = a_0\begin{pmatrix} \sqrt{3} &\sqrt{3}/2 \\ 0 &3/2\end{pmatrix}\mathbb{Z}^2, \quad S = 2, \quad p_0 = \begin{pmatrix} 0 \\ 0\end{pmatrix}, \quad p_1 = a_0\begin{pmatrix} \sqrt{3}/2 \\ 1/2\end{pmatrix}, \quad a_0 = \frac{\sqrt{2}}{3^{3/4}}.
\end{equation}
(The $a_0 = \frac{\sqrt{2}}{3^{3/4}}$ prefactor is due to the normalisation $\det(\mF) = 1$.)

\begin{figure}
\subfigure[2D graphene: the dashed circles indicate the interaction
   neighbourhoods of the highlighted atoms.]{
   \includegraphics[width=0.42\textwidth]{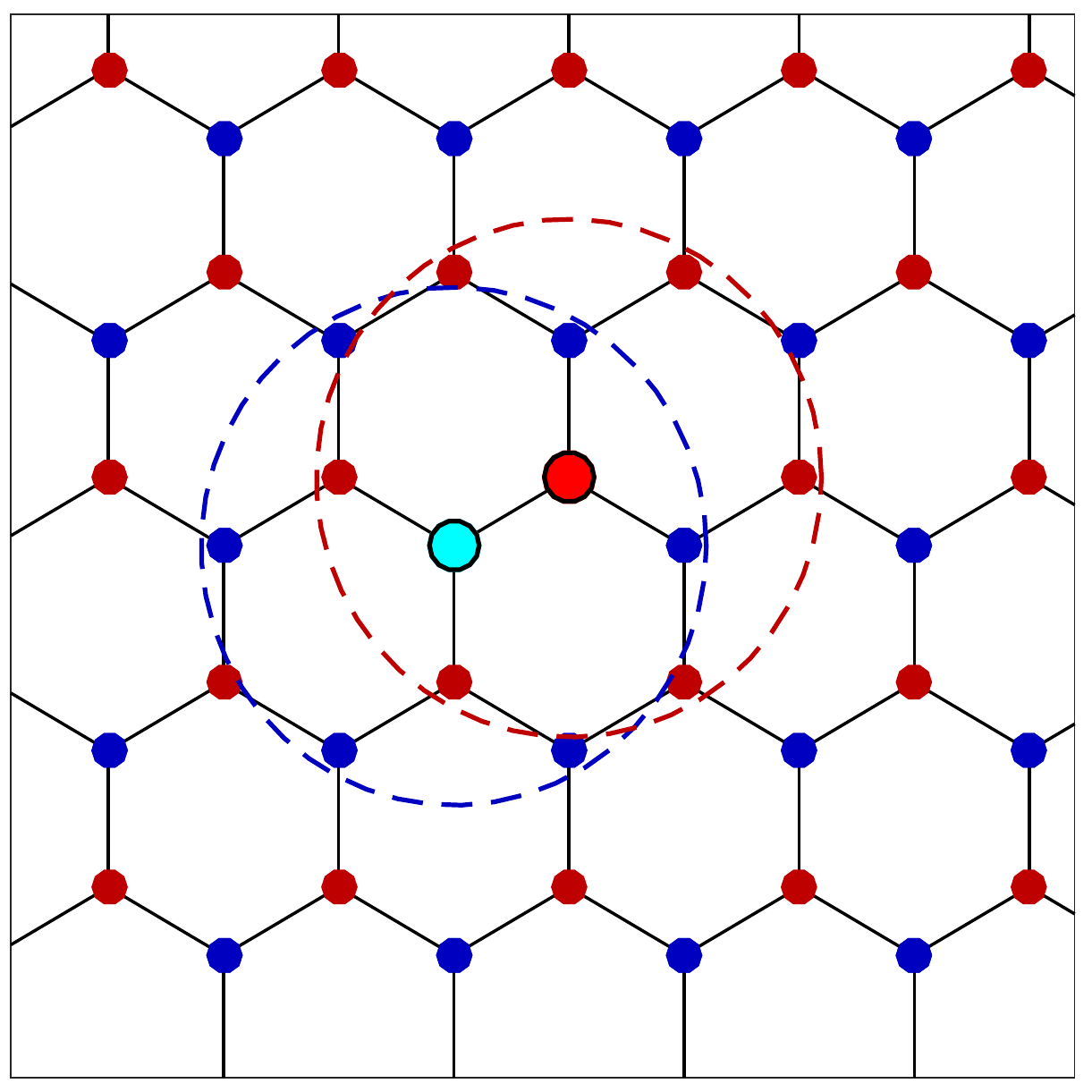}}
\qquad
\subfigure[3D rock salt: the interior cube represents a possible
      choice of unit cell.]{
   \includegraphics[width=0.42\textwidth]{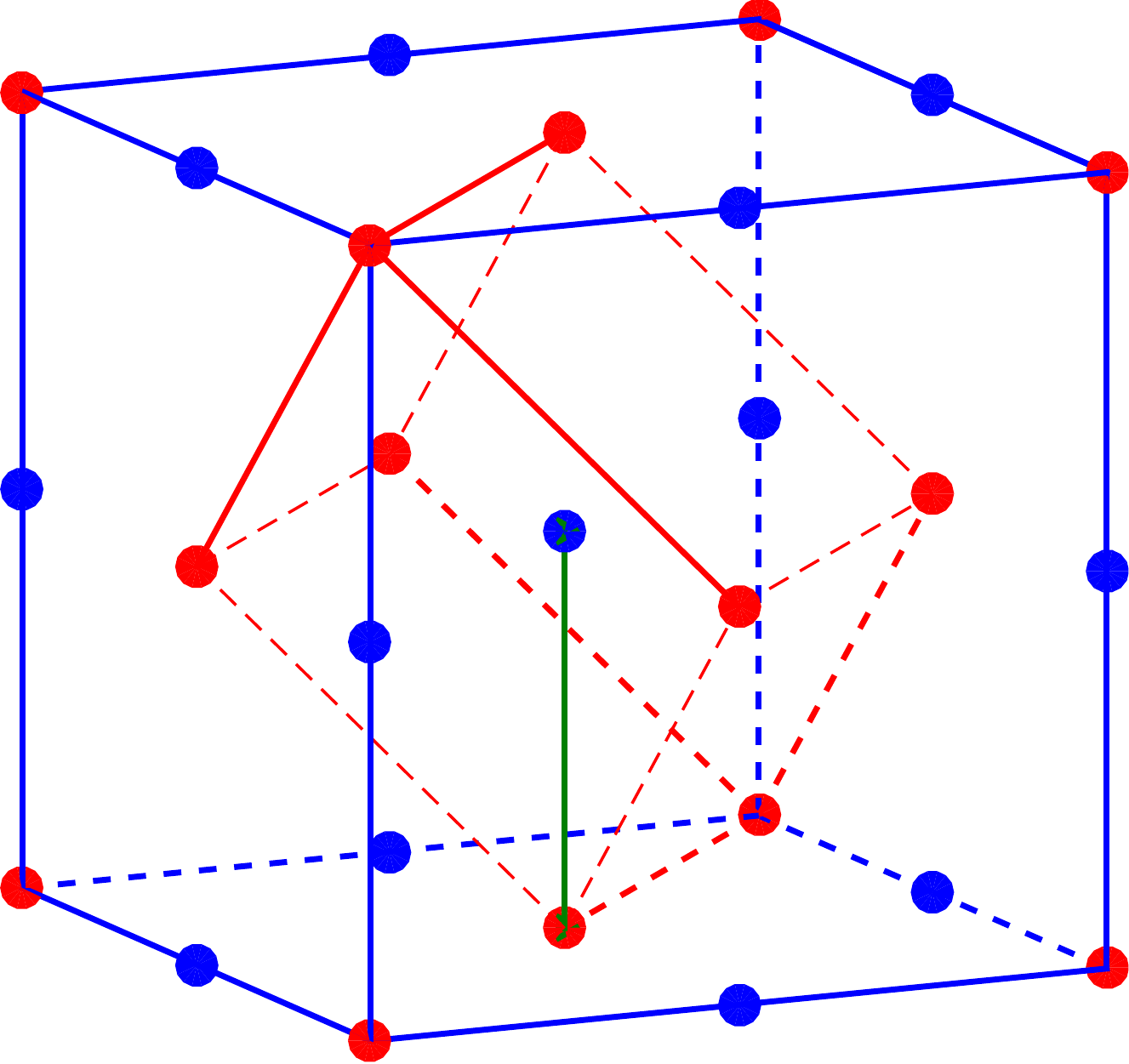}}
\caption{Examples of multilattice structures.} \label{fig:multilattice}
\end{figure}

For each species of atom, we define the deformation field $y_\alpha(\xi)$ as the deformation of the atom of species $\alpha$ at site $\xi$.  We note that $y_\alpha:\mathcal{L} \to \mathbb{R}^n$ where the range dimension $n \in \{2,3\}$ may be different than the domain dimension $d$ to allow, e.g., for out of plane displacements in $2D$. {However, we remark that our later assumptions on stability of the multilattice (Assumption~\ref{assumption2}) will place a restriction on the out of plane behavior; for example bending, or rippling, cannot currently be incorporated into the analysis.  We further discuss the issues involved in this in our concluding discussion, Section~\ref{discussion}.}  In the case of these out of plane displacements, we will use $\xi \in \mathbb{R}^2$ as both a vector in $\mathbb{R}^2$ and as the vector $\begin{pmatrix} \xi \\ 0 \end{pmatrix} \in \mathbb{R}^3$.  (We remark that though we will not consider dislocations, we could also consider $n = 1$ for an anti-plane screw dislocation model by fixing a second coordinate to be constant in this framework.)

The set of all sublattice deformations is denoted by $\bm{y}:= (y_{\alpha})_{\alpha=0}^{S-1}$ and displacements by $\bm{u}:= (u_{\alpha})_{\alpha=0}^{s-1}$. Equivalently we can describe the kinematics of a multilattice by a pair $(Y, \bm{p})$ where $Y : \mathcal{L} \to \mathbb{R}^n$ is a deformation field and $p_0, \dots, p_{S-1} : \mathcal{L} \to \mathbb{R}^n$ are shift fields. The two descriptions are related by
\begin{align*}
Y(\xi) =~ y_0(\xi), \quad p_\alpha(\xi) = y_\alpha(\xi) - y_0(\xi); \qquad \mbox{and} \qquad y_\alpha(\xi) =~ Y(\xi) + p_\alpha(\xi),
\end{align*}
and analogous expressions hold for displacements as well.

We now turn to a description of the energy. We will make the fundamental modeling
assumption that the total potential energy of the system can be written as a sum
of \textit{site potentials}---that is,
\begin{equation}\label{atDefEnergy}
   \hat{\mathcal{E}}^\a_{\rm hom}(\bm{y}) := \sum_{\xi \in \mathcal{L}} \hat{V}(D\bm{y}(\xi)),
\end{equation}
where the various new symbols introduced are specified in the following. \dao{We also note that this assumption is not restrictive as almost any reasonable classical potential such as an $n$-body potential, pair functional, or bond-order potential may be written in this form.} \dao{The main restriction is that long-range Coulomb interaction is excluded.}

We use $D\bm{y}(\xi)$ to denote the collection of finite differences (relative atom positions)
needed to compute the energy at site $\xi$. More precisely, we specify a
{\em finite} set of triples
\begin{equation*}
   \mathcal{R} \subset \mathcal{L} \times \{0, 1, \ldots, S-1\} \times \{0, 1, \ldots, S-1\}\setminus \bigcup_{\alpha = 0}^{S-1}\{ (0\alpha\alpha)\},
\end{equation*}
and use
\begin{equation}\label{finite_diff1}
   D_{\triple}\bm{y}(\xi) := y_\beta(\xi + \rho) - y_\alpha(\xi)
\end{equation}
to denote the relative positions of species $\beta$ at site $\xi+\rho$ and
species $\alpha$ at site $\xi.$ The collection of finite differences, or
finite difference {\dao{stencils}}, $D\bm{y}$, is then defined by
\begin{equation}\label{finite_diff2}
D\bm{y}(\xi) := \left(D_{\triple}\bm{y}(\xi)\right)_{\triple \in \mathcal{R}}.
\end{equation}
%
% given a site $\xi$, we assume a finite interaction range $$ with $|\rho| \leq r_{\rm cut}$.  Thus elements of $\mathcal{R}$ are triples, $\triple$, with $\rho$ denoting a lattice vector between sites and $\alpha$ and $\beta$ indicating the species of atoms involved.  Since the potential energy may only depend upon on distances between atoms, we define the finite difference notation
% \[
% D_{\triple}\bm{y}(\xi) := y_\beta(\xi + \rho) - y_\alpha(\xi).
% \]
In terms of $(Y,\bm{p})$, this this notation becomes
\[
D_{\triple}(Y,\bm{p}) := Y(\xi + \rho) - Y(\xi) + p_\beta(\xi+\rho) - p_\alpha(\xi)
\quad \text{and} \quad
 D(Y,\bm{p}) := \big( D_{\triple} (Y,\bm{p}))_{\triple \in \mathcal{R}}.
\]
For future reference we remark that we can write
\begin{equation*}
   D_{\triple} \bm{y} = D_\rho y_\beta(\xi) + p_\beta(\xi)-p_\alpha(\xi),
   %    \quad \text{and} \quad
   % D_{\triple}(Y,\bm{p})  = D_\rho Y(\xi) + D_{\rho} p_\beta(\xi) + p_\beta(\xi)-p_\alpha(\xi),
\end{equation*}
where $D_\rho f(\xi) := f(\xi + \rho) - f(\xi)$.  Moreover, we define the set of lattice vectors in $\mathcal{R}$ as
\[
\mathcal{R}_1 := \{ \rho \in \mathcal{L} : \exists \triple \in \mathcal{R}\}.
\]

The site potential is then a function $\hat{V} : (\R^n)^\mathcal{R} \to \R \cup \{+\infty \}$, where $+\infty$ allows for singularities in the potential (though we will later assume certain smoothness of the potential for convenience of the analysis).

Since the homogeneous reference configuration, $\bm{y}^{\rm ref}$, defined by
\begin{equation}\label{ref_config}
\bm{y}^{\rm ref}_\alpha(\xi) := \xi + p^{\rm ref}_\alpha,
\end{equation}
for constant $p_\alpha^{\rm ref} \in \mathbb{R}^n$
yields infinite energy, {\dao{(due to an infinite sum over constant values of the site potential in the reference configuration),}} we thus will consider an energy difference functional defined on displacements from the reference state instead of \eqref{atDefEnergy}. For a displacement
$\bm{u} \equiv (U, \bm{p})$ from the reference state $\bm{y}^{\rm ref}$ let
\begin{equation*}
   V(D\bm{u}(\xi)) = \hat{V}(D(\bm{y}^{\rm ref} + \bm{u})(\xi)),
\end{equation*}
and then the associated energy difference functional is defined by
\begin{equation}\label{atDispEnergy}
\mathcal{E}^\a_{\rm hom}(\bm{u}) := \sum_{\xi \in \mathcal{L}} V(D\bm{u}(\xi)) - V(0).
\end{equation}
where $V(0)$ is a constant which will not affect minimization or force computations, so for simplicity, we \dao{assume without loss of generality that $V(0)=0$}.  In Theorem~\ref{well_defined} \dao{below}, we recall a result \dao{of~\cite{olsonOrtner2016}} that characterizes for which displacements\dao{,} $\bm{u}$\dao{,}
$\mathcal{E}^\a_{\rm hom}(\bm{u})$ is well-defined.

A convenient notation for derivatives of $V$ is the following: if  $\triple, \tripleTau \in \mathcal{R}$ and $\bm{g} = (\bm{g}_{\triple})_{\triple \in \mathcal{R}} \in (\R^n)^{\mathcal{R}}$, we set
\begin{align*}
[V_{,\triple}(\bm{g})]_{i} :=~& \frac{\partial V(\bm{g})}{\partial \bm{g}_{\triple}^i }, \quad i = 1,\ldots, n, \\
V_{,\triple}(\bm{g}) :=~& \frac{\partial V(\bm{g})}{\partial \bm{g}_{\triple}}, \\
[V_{,\triple\tripleTau}(\bm{g})]_{ij} :=~& \frac{\partial^2 V(\bm{g})}{\partial \bm{g}_{\tripleTau}^j \partial \bm{g}_{\triple}^i},  \quad i,j = 1,\ldots, n, \\
V_{,\triple\tripleTau}(\bm{g}) :=~& \frac{\partial^2 V(\bm{g})}{\partial \bm{g}_{\tripleTau} \partial \bm{g}_{\triple}},
\end{align*}
and {\dao{note that this}} can be extended to derivatives of arbitrary order.  Furthermore, we adopt the convention that if $\triple \notin \mathcal{R}$, then $V_{,\triple} = 0$.

The following standing assumptions on the interaction range and site potentials are made.
\begin{assumption}\label{assumption1}
\hspace{2em}
\begin{enumerate}
\item[(V.1)] The interaction range, $\mathcal{R}$, satisfies
\begin{align*}
    & \mbox{\dao{For each $\alpha \in \{0,\ldots,S-1\}$, the set of vectors $\rho$ such that $(\rho\alpha\alpha) \in \mathcal{R}$ spans $\mathbb{R}^d$,}} \\
   & \mbox{and } (0\alpha\beta) \in \mathcal{R} \quad \text{for all $\alpha \neq \beta \in \{0,\ldots,S-1\}$ }. %\label{cond2}.
\end{align*}
%\item[V.2] There exists $R_{\rm def} > 0$ such that $V_\xi \equiv V$ for all $|\xi| \geq R_{\rm def}$.
\item[(V.2)] $V$ is four times continuously differentiable with uniformly bounded derivatives and satisfies $V(0) = 0$ (for simplicity of notation).  Since $V:(\mathbb{R}^n)^{\mathcal{R}} \to \mathbb{R}$, the statement that $V$ has uniformly bounded derivatives means there exists $M$ such that for any multi-index $\gamma$ with $|\gamma| \leq 4$, $|\partial_\gamma V| \leq M$.
\end{enumerate}
\end{assumption}

We remark that (V.1) may always be met by enlarging the interaction range, \dao{$\mathcal{R}$}. \co{On the other hand, (V.2) is made for simplicity of the analysis; it can be weakened to admit interatomic potentials with typical singularities under collisions of atoms, but this would introduce several additional technicalities in our analysis.}

Next, we specify the function space over which $\mathcal{E}_{\rm hom}^\a(\bm{u})$ is defined, which can be achieved in several equivalent ways. A convenient route is by first defining a continuous, piecewise linear interpolant of an atomistic displacement.  Let $\mathcal{T}_\a$ be a simplicial decomposition of $\mathcal{L}$ obtained as in~\cite{blended2014}: first let $\hat{T} := \conv \{0, e_1, e_2\}$ \dao{(where $\conv$ represents the convex hull of a set)} be the unit triangle in $2D$ and $\hat{T}_1, \ldots, \hat{T}_6$ six congruent tetrahedra in $3D$ that subdivide the unit cube \dao{(see Figure 1 in~\cite{blended2014} for an illustration in $3D$)} and then define
\begin{equation*}\label{eq:Ta}
\mathcal{T}_\a =
\begin{cases}
\{\xi + \mF \hat{T}, \xi -\mF\hat{T}: \xi \in \mathcal{L} \}, \quad \mbox{if $d = 2$,} \\
\{\xi + \dao{\mF}\hat{T}_i: \xi \in \mathcal{L}, i = 1, \ldots, 6 \}, \quad \mbox{if $d = 3$}.
\end{cases}
\end{equation*}
We will often refer to this as the atomistic triangulation or \textit{fully refined} triangulation.  As noted before, we may always enlarge the interaction range, $\mathcal{R}$, so we may assume without loss of generality that
 \begin{equation*} \label{assumption:mesh}
\dao{     \text{if } {\rm conv}\{ \xi, \xi+\rho \} \text{ is an edge of $\mathcal{T}_\a$,
     then there exist $\alpha, \beta$ such that } \triple \in \mathcal{R}.}
 \end{equation*}

Given a {\helen{discrete set of displacement values}} $u: \mathcal{L} \to \mathbb{R}^n$, we then denote the continuous, piecewise linear interpolant of $u$ with respect to $\mathcal{T}_\a$ by $Iu \equiv \bar{u}$. We will use both notations, \dao{$Iu$ and $\bar{u}$}, depending on which is notationally more convenient. Subsequently, we define the function space
\begin{equation*}\label{disSpace}
\begin{split}
\mathcal{U} :=~& \left\{ \bm{u} = (u_\alpha)_{\alpha = 0}^{S-1} : u_\alpha:\mathcal{L} \to \mathbb{R}^n, \|\bm{u}\|_\a < \infty \right\}, \, \mbox{where} \\
\|\bm{u}\|_\a^2 :=~& \sum_{\alpha = 0}^{S-1}\|\nabla Iu_\alpha\|_{L^2(\mathbb{R}^d)}^2 + \sum_{\alpha \neq \beta}\| Iu_\alpha - Iu_\beta\|_{L^2(\mathbb{R}^d)}^2.
\end{split}
\end{equation*}

Clearly, $\|\cdot\|_\a$ is not a norm on $\mathcal{U}$ since $\|\bm{u}\|_\a = 0$ only implies that each $u_\alpha(\xi)$  is a constant independent of $\alpha$.  However, $\|\cdot\|_\a$ is a semi-norm on $\mathcal{U}$ and hence a true norm on the quotient space
\begin{equation*}
   \bm{\mathcal{U}} := \mathcal{U}/\mathbb{R}^n
      := \big\{  \{(u_\alpha + C)_{\alpha = 0}^{S-1} : C \in \mathbb{R}^n \}
               \,:\, \bm{u} \in \mathcal{U} \big\}.
\end{equation*}
Since the atomistic energy is invariant with respect to addition by constants, it is exactly this quotient space which we utilize as our function space.  {\helen{We also note that $\bm{u}$ and $(U,\bm{p})$ are two equivalent descriptions for the displacements}, and an equivalent norm on this space which will be convenient in terms of the $(U,\bm{p})$ description is}
\begin{equation*}%\label{atom_norm}
\|(U,\bm{p})\|_\a := \|\nabla IU\|_{L^2(\mathbb{R}^d)}^2 + \sum_{\alpha = 1}^{S-1} \| Ip_\alpha\|_{L^2(\mathbb{R}^d)}^2.
\end{equation*}

A dense subspace of $\mathcal{U}$ that we will use as a test function space is $\bm{\mathcal{U}}_0$ where
\begin{equation*}\label{testSpace}
\begin{split}
\mathcal{U}_0 :=~& \left\{\bm{u} {\helen{\in \mathcal{U}}} : {\rm supp}(\nabla Iu_0), \,  \mbox{and} \, {\rm supp}(Iu_\alpha - Iu_0) \, \mbox{are compact}\right\}, \\
\bm{\mathcal{U}}_0 :=~& \mathcal{U}_0/\mathbb{R}^n.
\end{split}
\end{equation*}
As proven in~\cite{olsonOrtner2016}, this test space is dense in $\bm{\mathcal{U}}$.
\begin{lemma}\cite[Lemma A.1]{olsonOrtner2016}\label{lem:dense}
The quotient space $\bm{\mathcal{U}}_0$ is dense in $\bm{\mathcal{U}}$.
\end{lemma}

\subsection{Point Defect}
We
now introduce a framework to embed a point defect in a homogeneous
multilattice. This problem has been heavily used in analyzing and comparing
different AtC methods for simple lattices
in~\cite{olson2015,blended2014,acta.atc,OrtnerZhang2014bgfc} as it allows for a
range of non-trivial benchmark problems and serves as a first step in analyzing
more complicated scenarios such as interacting defects~\cite{hudson2015}.  {\helen{Point
defects can be}} thought of as zero-dimensional defects representing a
change to a single site in the lattice. Common examples include vacancies,
interstitials, substitutions, and in graphene, the Stone--Wales defect which we
use for our numerical verification.

Our first task is to define an analog of $\mathcal{E}^\a_{\rm hom}$ for point
defects, which is well-defined on the function space $\bm{\mathcal{U}}$.
We accomplish this through a site-dependent site potential, $V_\xi$, which must take
into account the defective structure of the lattice near the defect core, \dao{which we assume to be at or near the
origin}. We then write the atomistic potential energy as
\begin{equation}\label{defEnergy}
   \mathcal{E}^\a(\bm{u}) := \sum_{\xi \in \mathcal{L}}
               V_\xi(D\bm{u}(\xi)).
\end{equation}

As in Assumption~\ref{assumption1}, we require certain smoothness of the site-dependent site potential in addition to homogeneity outside of a defect core.
\begin{assumption}\label{assumptionSite}
\quad
\begin{enumerate}
\item[(V.3)] There exists $R_{\rm def} > 0$ such that $V_\xi \equiv V$ for all \dao{$|\xi| \geq R_{\rm def}$}.
\item[(V.4)] Each $V_\xi$ is four times continuously differentiable with uniformly bounded derivatives.
\end{enumerate}
\end{assumption}

We now recall from \cite[Theorem 2.2]{olsonOrtner2016} that $\mathcal{E}^\a$
and $\mathcal{E}^\a_{\rm hom}$ are well-defined on $\bm{\mathcal{U}}$; \dao{the main idea of the proof is that both are defined on displacements having compact support, and by density of $\bm{\mathcal{U}}_0$ in $\bm{\mathcal{U}}$, they may be uniquely extended by continuity to all of $\bm{\mathcal{U}}$}.
\begin{theorem}\cite[Lemma 3.3]{olsonOrtner2016}\label{well_defined}
Assume the reference configuration $\bm{y}^{\rm ref}$ with $y^{\rm ref}_\alpha(\xi) = \xi + p^{\rm ref}_\alpha$ is an equilibrium configuration of the defect free energy meaning that
\begin{equation}\label{ostrich1}
\sum_{\xi \in \mathcal{L}} \sum_{\triple \in \mathcal{R}} \hat{V}_{,\triple}(D\bm{y}^{\rm ref}(\xi)) \cdot D\bm{v}(\xi) = 0, \quad \forall \, \bm{v} \in \bm{\mathcal{U}}_0.
\end{equation}
Then $\mathcal{E}^\a_{\rm hom}(\bm{u})$ and $\mathcal{E}^\a(\bm{u})$ may be uniquely extended to continuous functions on $\bm{\mathcal{U}}$ which are ${\rm C}^3$ (three times continuously differentiable) on $\bm{\mathcal{U}}$.
\end{theorem}

% \medspace
\begin{remark}
The condition~\eqref{ostrich1} that the reference configuration be an equilibrium is equivalent to requiring the shifts are equilibrated within each cell.  See~\cite[Lemma 9]{olsonOrtner2016} for details. Such reference configurations are thus straightforward to generate numerically.
\end{remark}

\medskip

Since we will eventually be working with a finite domain on which there is no difference between the original functionals and their extensions, we make no distinction between an energy and its continuous extension.

We are now able to pose the defect equilibration problem which we wish to
approximate with the BQCF method, {\helen{that is, to find $\bm{u}^\infty \in\bm{\mathcal{U}} $ such that}}
\begin{equation}\label{def_problem}
\bm{u}^\infty \in \arg\min_{\bm{u} \in \bm{\mathcal{U}}} \mathcal{E}^\a(\bm{u}),
\end{equation}
where {\helen{$\arg\min$}} represents the set of local minima of a functional.

While Assumptions~\ref{assumption1} and~\ref{assumptionSite} can be readily
weakened in various ways, the next assumption concerning existence and stability
of a defect configuration minimizing $\mathcal{E}^\a$ is essential for our
analysis:
\begin{assumption}\label{assumption2} (Strong Stability)
   There exists a solution, $\bm{u}^\infty$, to~\eqref{def_problem} and a constant $\gamma_\a > 0$
   such that
\[
\<\delta^2\mathcal{E}^\a(\bm{u}^\infty)\bm{v},\bm{v}\> \geq \gamma_\a \|\bm{v}\|_\a^2
\qquad \forall \bm{v} \in \bm{\mathcal{U}}_0.
\]
\end{assumption}

\medskip

Proving Assumption~\ref{assumption2} turns out to be notoriously difficult;
indeed the only result of this kind we are aware of is for a special case of
a screw dislocation in a simple lattice~\cite[Remark 3.2]{hudson2015} under anti-plane
deformation. Nevertheless, we expect it to hold for {\em virtually all}
realistic defects and realistic interatomic potentials. We also mention
that it can be numerically checked \textit{a posteriori} once the defect
configuration has been computed.

% \co{***This assumption importantly allows us to establish decay estimates on both the displacements, $U^\infty = u_0^\infty$, and shifts, $\bm{p}^\infty$, and $k$th order finite differences defined by $D_{\rho_1\rho_2\cdots\rho_k} u(\xi) := D_{\rho_1}D_{\rho_2}\cdots D_{\rho_k} u(\xi)$ for nonzero lattice vectors $\rho_1, \rho_2,\ldots, \rho_k$.  We refer to~\cite{olsonOrtner2016} for a more thorough discussion and proof and~\cite{Ehrlacher2013} for a corresponding result for point defects in Bravais lattices.***}
% \commentco{I'd remove this paragraph; the assumption is box-standard stability,
% I don't think it needs extra motivation.}

A useful consequence of Assumption~\ref{assumption2} is the following
regularity result, which is proven in \cite{olsonOrtner2016} and
extends the analogous simple lattice result~\cite{Ehrlacher2013}. \co{These decay rates will be an essential component for converting the BQCF
error estimates in terms of solution regularity that are presented in Section~\ref{bqcf}
into complexity estimates that are numerically verified in Section~\ref{num}.}

\begin{theorem}\cite[Theorem 2.5]{olsonOrtner2016}\label{decay_thm}
For $\bm{\rho} = \rho_1 \dots \rho_k$, the defect solution $(U^\infty, \bm{p}^\infty)$ satisfies
\begin{equation}\label{decay_est}
\begin{split}
|D_{\bm{\rho}} U^\infty(\xi)| \lesssim~& (1 + |\xi|)^{1-d-k}, \quad \mbox{for $1 \leq k \leq 3$}, \\
|D_{\bm{\rho}} p_\alpha^\infty(\xi)| \lesssim~& (1+|\xi|)^{-d-k}, \quad \mbox{for $0 \leq k \leq 2$, and all $\alpha = 0, \ldots, S-1$.}
\end{split}
\end{equation}
The implied constant is allowed to depend on the interaction range through the maximum of $|\rho|$ for $\rho \in \mathcal{R}_1$, the site potential, and $\gamma_\a$.
\end{theorem}

\medskip

These decay rates will be an essential component for converting the BQCF
error estimates in terms of solution regularity that are presented in Section~\ref{bqcf}
into complexity estimates that are numerically verified in Section~\ref{num}.

Since we will compare discrete atomistic configurations
 with continuous finite element functions, it will be useful to reformulate
Theorem~\ref{decay_thm} in terms of gradients of smooth interpolants, which
we define in the next lemma (see~\cite{blended2014} for further details and the proof).

% To that end, For this task, we denote a smooth interpolant of a
% displacement $u$ (or shift $p$) by $\tilde{I}u$ or $\tilde{u}$ ($\tilde{I}p$ or
% $\tilde{p}$).  The definition of $\tilde{I}u$ is taken from~.

\begin{lemma}\label{smoothInterpolant}
Let $u:\mathcal{L} \to \mathbb{R}^n$, then there exists a unique function $\tilde{I}u:\mathbb{R}^d \to \mathbb{R}^n$ with  $\tilde{I}u \in {\rm C}^{2,1}(\mathbb{R}^d)$ such that
\begin{enumerate}
\item $\tilde{I}u$ is multiquintic in $\xi + \mF(0,1)^d$ for each $\xi \in \mathcal{L}$.
% \item $\tilde{I}u(\xi) = u(\xi)$ for all  $\xi \in \mathcal{L}$.
\item Given any multiindex $\gamma$ with $|\gamma| \leq 2$, the interpolant
   satisfies $\partial_\gamma \tilde{I}u(\xi) = D^{\nn}_\gamma u(\xi)$
   where $D^{\nn}_\gamma$ are nearest-neighbor finite difference operators,
\begin{align*}
D^{\nn,0}_i u(\xi) :=~& u(\xi), \\
D^{\nn,1}_i u(\xi) :=~& \frac{1}{2}(u(\xi + \mF e_i) - u(\xi - \mF e_i)) \quad (e_i \mbox{ is the $i$th standard basis vector}), \\
D^{\nn,2}_i u(\xi) :=~& u(\xi + \mF e_i) -2u(\xi) + u(\xi - \mF e_i), \\
D^{\nn}_\gamma u(\xi) :=~& D^{\nn,|\gamma_1|}_{1}\cdots D^{\nn,|\gamma_d|}_{d}u(\xi).
\end{align*}
\end{enumerate}
\end{lemma}

\medskip

We will apply $\tilde{I}$ to both displacements and shifts using the notation
\[
\tilde{I}(U, \bm{p}) = (\tilde{I}U, \tilde{I}\bm{p}) = (\tilde{U}, \tilde{\bm{p}}).
\]
Then, combining Theorem~\ref{decay_thm} and Lemma \eqref{smoothInterpolant}
yields the following result.

\begin{theorem}\label{decay_thm1}
The defect solution $(U^\infty, \bm{p}^\infty)$ satisfies
\begin{equation}\label{decay_est_cont}
\begin{split}
|\nabla^j \tilde{I}U^\infty(x)| \lesssim~& (1+|x|)^{1-d-j}, \quad \mbox{for $j = 1,2$}, \\
|\nabla^j \tilde{p}_\alpha^\infty(x)| \lesssim~& (1+|x|)^{-d-j}, \quad \mbox{for $j = 0,1,2$, and all $\alpha = 0, \ldots, S-1$,}
\end{split}
\end{equation}
where the implied constant is again allowed to depend on the interaction range, the site potential, and $\gamma_\a$.
\end{theorem}

\section{BQCF Method Formulation and Main Results}\label{bqcf}

Any AtC approximation of the defect problem~\eqref{def_problem}
must include the following ingredients: the atomistic and continuum
domains, a coarsened finite element mesh in the continuum region, a
specification of the continuum model, and finally and most importantly
a mechanism for coupling the atomistic and continuum components.

We define the atomistic and continuum domains for the multilattice BQCF method
by making similar choices as in the BQCF method for Bravais
lattices~\cite{blended2014}. We first give an intuitive description of the
domains involved, but will (re-)define them again below after introducing the
\textit{blending function}. Choose a computational domain $\Omega \subset \mathbb{R}^d$
to be a (large) polygonal domain containing the origin (the defect). Fix a
``defect core'' region $\Omega_{\rm core}$ such that, if $V_\xi \not\equiv V$,
then $\xi \in \Omega_{\rm core}$.  Then take $\Omega_\a$, the atomistic domain,
to be a polygonal domain with $\Omega_{\rm core} \subset \Omega_\a \subset
\Omega$, and set $\Omega_\c$, the continuum domain to be $\Omega_\c =
\Omega\setminus\Omega_{\rm core}$. In blending methods, the atomistic and
continuum domains overlap in a blending region $\Omega_{\rm b} = \Omega_\c \cap
\Omega_\a$ over which the atomistic and continuum forces will be blended.

Next, we define a finite element mesh $\mathcal{T}_h$ over $\Omega$ with
nodes $\mathcal{N}_h$. For now we only require that the finite element mesh
is fully refined over $\Omega_\a$, that is,
if $T \cap \Omega_\a \neq \emptyset$, then $T \in \mathcal{T}_h$ if and only if $T \in \mathcal{T}_\a$, but we will state further assumptions in
Section~\ref{sec:approx_params}.

The continuum model we adopt is the Cauchy--Born model~\cite{cauchy, born1954, ortiz1996quasicontinuum}, a nonlinear
hyperelastic model, which is amenable to AtC couplings due to the definition of the
strain energy density function in terms of the atomistic potential $V$,
\[
W_{\rm CB}(\mG, \bm{p}) := V\Big((\mG \rho + p_\beta - p_\alpha)_{\triple \in \mathcal{R}}\Big) \quad \mbox{for $\mG \in \mathbb{R}^{n \times d}$ and $\bm{p} \in (\mathbb{R}^n)^{S}$},
\]
without resorting to any constitutive laws. {\helen{We note that $G$ here is the deformation gradient of lattice sites in a unit cell while $\bm{p}$
are the displacements of shift vectors; in contrast with typical continuum treatments of multilattices, we maintain the shift vectors as degrees of freedom in the Cauchy--Born model and do not minimize them out.}}

For $W^{1,\infty}$ displacement fields, $U$, and $L^\infty$ shift fields, $\bm{p}$,
this leads to a Cauchy--Born energy functional, formally (for now) defined by
\begin{equation*}\label{cb_energy}
\mathcal{E}^\c(U, \bm{p}) := \int_{\mathbb{R}^d} W_{\rm CB}(\nabla U(x), \bm{p}(x))\, dx = \int_{\mathbb{R}^d} V\big(\nabla(U,\bm{p})\big)\, dx
\end{equation*}
where
\begin{equation*}\label{cont_grad}
\nabla(U,\bm{p}) := \big(\nabla_{\triple}(U,\bm{p})\big)_{\triple \in \mathcal{R}} := \big(\nabla_\rho U + p_\beta - p_\alpha\big)_{\triple \in \mathcal{R}}
\end{equation*}
is a continuum variant of the atomistic \dao{finite difference stencil
\begin{equation*}
D(U,\bm{p})(x) = \big(D_{\triple}(U,\bm{p})(x)\big)_{\triple \in \mathcal{R}} := \big(D_\rho U(x) + p_\beta(x+\rho) - p_\alpha(x)\big)_{\triple \in \mathcal{R}}.
\end{equation*}}

The admissible finite element space we consider will be $\mathcal{P}_1$ finite elements for both the displacements and the shifts subject to homogeneous boundary conditions.  However, we will again consider equivalence classes of finite element functions by taking a quotient space.  Thus, we define
\begin{equation*}\label{fin_spaces}
\begin{split}
\mathcal{U}_h :=~& \left\{u \in {\rm C}^0(\Omega) : u|_{T} \in \mathcal{P}_1(T), \quad \forall \, T \in \mathcal{T}_h\right\}, \\
\bm{\mathcal{U}}_h :=~& \mathcal{U}_h / \mathbb{R}^n, \\
\mathcal{U}_{h,0} :=~& \left\{u \in {\rm C}^0(\mathbb{R}^d): u|_{T} \in \mathcal{P}_1(T), \quad \forall \, T \in \mathcal{T}_h, u = 0 \mbox{ on } \mathbb{R}^d\setminus\Omega \right\}, \\
\bm{\mathcal{U}}_{h,0} :=~& \mathcal{U}_{h,0}/ \mathbb{R}^n, \\
\bm{\mathcal{P}}_{h,0}
:=~& \dao{\left\{\bm{p}=(p_0,\dots,p_{S-1}): p_0=0, \text{ and }p_1,\dots,p_{S-1} \in \big(\mathcal{U}_{h,0}\big)^{S-1}\right\}}.
\end{split}
\end{equation*}
These spaces are endowed with the norm
\begin{equation*}\label{eq:ml}
\|(U, \bm{p})\|_{\rm ml}^2 := \|\nabla U \|_{L^2(\mathbb{R}^d)}^2 + \sum_{\alpha = 0}^{S-1}\|p_\alpha\|^2_{L^2(\mathbb{R}^d)} = \|\nabla U \|_{L^2(\mathbb{R}^d)}^2 + \|\bm{p}\|^2_{L^2(\mathbb{R}^d)},
\end{equation*}
where $\|\bm{p}\|^2_{L^2(\mathbb{R}^d)} = \sum_{\alpha = 0}^{S-1}\|p_\alpha\|^2_{L^2(\mathbb{R}^d)}$ is used for brevity.  Along with the finite element space, we also introduce the standard piecewise linear finite element interpolant, $I_h$, defined as usual through $I_h u(\nu) = u(\nu)$ for $\nu \in \mathcal{N}_h$.

The BQCF method is defined by blending forces on each degree of freedom, \dao{$(\nu,\alpha) \in \mathcal{N}_h \times \{0,\ldots,S-1\}$}, where the forces are defined by a weighted average of atomistic and continuum forces:
\begin{equation}\label{f_bqcf}
\mathcal{F}^{\bqcf}_{\nu,\alpha}(U,\bm{p}) :=  (1-\varphi(\nu))\frac{ \partial \mathcal{E}^\a(U,\bm{p})}{\partial u_\alpha(\nu)} + \varphi(\nu)\frac{ \partial \mathcal{E}^\c(U,\bm{p})}{\partial u_\alpha(\nu)},
\end{equation}
where \dao{the \textit{blending function}, $\varphi$, satisfies} $\varphi \in \rm{C}^{2,1}(\R^d)$ with $\varphi = 0$ in $\Omega_{\rm core}$
and $\varphi = 1$ in $\R^d \setminus \Omega_{\rm a}$.  The BQCF method then seeks to solve $\mathcal{F}^{\bqcf}_{\nu,\alpha}(U,\bm{p}) = 0$ for all $\nu \notin \partial \Omega$. Equivalently, we can write the force balance equations in weak form using the variational operator
{\helen{
\begin{align}
\<&\mathcal{F}^{\bqcf}(U,\bm{p}), (W,\bm{r})\> \nonumber\\
&\; := \sum_{\nu}\sum_{\alpha}\mathcal{F}^{\bqcf}_{\nu,\alpha}(U,\bm{p})\cdot \left(W+r_{\alpha}\right)(\nu)\nonumber\\
&\; = \sum_{\nu}\sum_{\alpha} (1-\varphi(\nu))\frac{ \partial \mathcal{E}^\a(U,\bm{p})}{\partial u_\alpha(\nu)}\cdot\left(W+r_{\alpha}\right)(\nu) + \varphi(\nu)\frac{ \partial \mathcal{E}^\c(U,\bm{p})}{\partial u_\alpha(\nu)}\cdot\left(W+r_{\alpha}\right)(\nu)\nonumber\\
&\; = \<\delta\mathcal{E}^\a(U,\bm{p}),((1-\varphi)W,(1-\varphi)\bm{r})\>\nonumber\\
&\qquad\qquad+  \<\delta\mathcal{E}^\c(U,\bm{p}),(I_h(\varphi W),I_h(\varphi \bm{r}))\>,\label{v_bqcf}
\end{align}
}}
where the last equal sign comes from direct calculation.
The BQCF approximation to the defect optimization problem~\eqref{def_problem}
is then to {\it find $(U, \bm{p}) \in \bm{\mathcal{U}}_{h,0} \times \bm{\mathcal{P}}_{h,0}$
such that }
\begin{equation}\label{bqcf_approx}
   \<\mathcal{F}^{\bqcf}(U,\bm{p}), (W,\bm{r})\>  = 0,
    \quad \forall (W,\bm{r}) \in \bm{\mathcal{U}}_{h,0} \times \bm{\mathcal{P}}_{h,0}.
\end{equation}
The variational formulation is preferred for the
analysis while the force-based formulation (from which the name BQCF is derived) is
preferred for implementation.  The pointwise formulation~\eqref{f_bqcf} was essentially how the original BQCF method was proposed for Bravais lattices~\cite{badia2007force}, and this was analyzed in a finite-difference framework without defects for Bravais lattices in~\cite{lu2013, li2012positive}.  The variational formulation~\eqref{v_bqcf} was introduced in~\cite{blended2014} for Bravais lattices, and its subsequent analysis led to one of the first complete analyses of an AtC method capable of modeling defects.

%%% CO: if we want to keep this, then ideally in a formal remark
% A motivating factor for using $\mathcal{P}^1$ finite elements for the shift fields as well as the displacements is to maintain this interpretation of the BQCF method.  Also note that the forces are blended by site and not by atom, and this is importantly allows us to use the same domain and blending function definitions as the BQCF method for Bravais lattices.

%%% CO: I don't think this is needed - this paper will likely only be read
%%%     by experts anyhow
% The force interpretation~\eqref{f_bqcf} also illuminates the vital features the blending function should possess; $\varphi(x)$ should be zero near the defect core (in $\Omega_{\rm core}$) thus giving full weight to the atomistic force and no wait to the continuum force and one far away from the defect in $\Omega\setminus\Omega_\a$ so that the continuum forces are the only ones that exist.  It is only over the blending region, $\Omega_{b}$, where the atomistic and continuum forces coexist and are said to be blended together.  In this region, the blending function should ``smoothly'' transition from zero to one.

\subsection{Assumptions on the Approximation Parameters}
\label{sec:approx_params}
We now summarise the precise technical requirements on the approximation
parameters, $\varphi, \Omega, \Omega_\a, \Omega_\b, \Omega_\c, \mathcal{T}_h$,
which will be analogous to those in \cite{blended2014}.

We begin by summarising basic assumptions on the blending function:
\begin{enumerate}
\item $\varphi \in \rm{C}^{2,1}$ and $0 \leq \varphi \leq 1$

\item If $V_\xi \not\equiv V$, then $\varphi(\xi) = 0$.  This means that
   $\varphi$ vanishes near any defect, hence the pure atomistic force is
   employed in those regions.

\item There exists $K > 0$ such that $\varphi(x) = 1$ if $|x| \geq K$.
      That is, $\varphi$ is identically one far away from the defect.
\end{enumerate}

As the second step we specify the computational domain $\Omega$ and its
corresponding partition $\mathcal{T}_h$. {\helen{First, we shall require that ${\rm supp}(1-\varphi)\subset\Omega $ always holds.}} To state the required properties
for $\mathcal{T}_h$, we first precisely specify the sub-domains in terms of
$\varphi$ and $\Omega$. Let
\begin{equation*}
   r_{\rm cut} := \max \{ |\rho| :  \triple \in \mathcal{R} \}
\end{equation*}
be an interaction cut-off radius, {let $r_{\rm cell}$ be the radius of the smallest ball circumscribing the unit cell of $\mathcal{L}$, and define $r_{\rm buff} := \max\{ r_{\rm cut}, r_{\rm cell}\}$}. Then we set
%
% It is precisely the blending function that also sets the exact technical requirements for the domain decomposition procedure alluded to at the beginning of this section.  The difference is that we must pad these inexact descriptions with a buffer region accounting for the size of the interaction range.  Specifically, set
\begin{align*}
\Omega_\a :=~& {\rm supp}(1-\varphi) + B_{4r_{\rm buff}}, \quad \Omega_{\rm b} :=~ {\rm supp}(\nabla \varphi) + B_{4r_{\rm buff}}, \\
\Omega_\c :=~& {\rm supp}(\varphi) \cap \Omega + B_{4r_{\rm buff}}, \quad \Omega_{\rm core} :=~ \Omega\setminus\Omega_\c.
\end{align*}
% (The requirements on the finite element mesh are unchanged once these exact definitions are made.)
The size and shape regularity of the various subdomains are parameterized
in terms of inner and outer radii: for ${\rm t} \in \{ \a, \c, \b, {\rm core}\}$,
we set
\[
   r_{\rm t} := \sup_{r}\{r > 0: {\helen{ B_{r} }} \subset
                        \Omega_{\rm t} \cup \Omega_{\rm core} \},
   \quad R_{\rm t} := \inf_{R} \{R > 0: \Omega_{\rm t} \subset {\helen{B_{R}}}\},
\]
{\helen{where we recall the notation $B_R$ to denote the ball of radius $R$ about the origin. }}  The corresponding outer and inner radii for the complete domain $\Omega$
are, respectively, denoted by $R_{\rm o}$ and $r_{\rm i}$:
\dao{
\[
   r_{\rm i} := \sup_{r}\{r > 0: {\helen{ B_{r} }} \subset
                        \Omega \},
   \quad R_{\rm o} := \inf_{R} \{R > 0: \Omega \subset {\helen{B_{R}}}\}.
\]
}
% To apply this to $\Omega$, we define $R_{\rm o} := \inf_{R} \{R > 0: \Omega \subset B_{R}(0)\}$ so $R_{\rm o}$ is an exterior measurement of $\Omega$ and $r_{\rm i} := \sup_{r}\{r > 0: B_{r}(0) \subset \Omega\}$ so $r_{\rm i}$ is an interior measurement.
Finally, we define an overlapping exterior domain,
\begin{align*}
   \Omega_{\rm ext} := \mathbb{R}^d \setminus {\helen{ B_{r_{\rm i}/2}}},
\end{align*}
which will be used to quantify the far-field error made by truncating
to a finite computational domain.

\begin{figure}
   \centering
   {\includegraphics[width=0.45\textwidth]{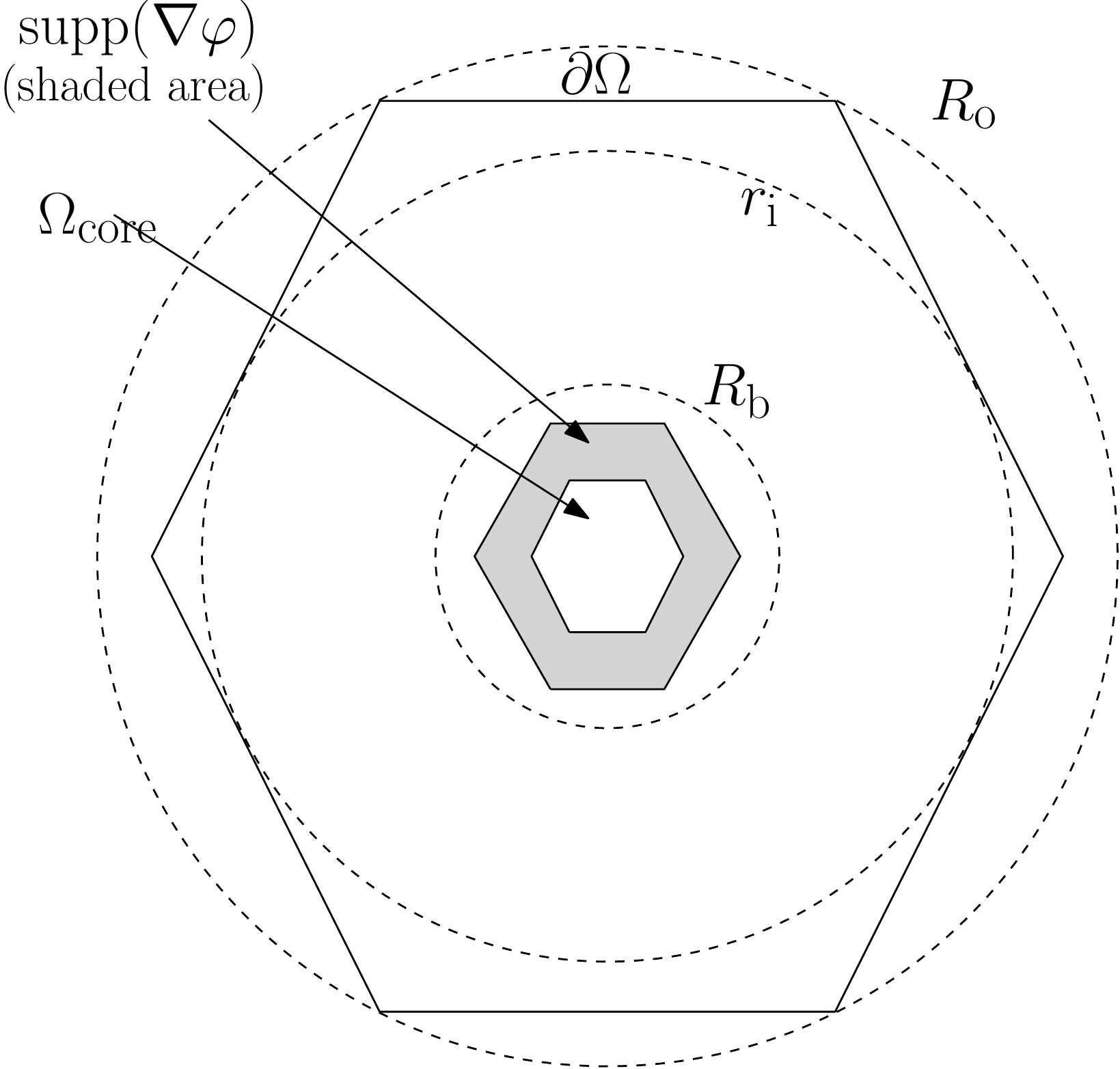}}
   \caption{A diagram showing a selected number of domains and their inner and outer radii.}\label{fig:domains}
\end{figure}

For the sake of completeness, we now restate a crucial condition on the finite
element mesh:
\begin{enumerate}[resume]
\item  The finite element mesh is fully refined over $\Omega_\a$, that is,
  if $T \cap \Omega_\a \neq \emptyset$, then $T \in \mathcal{T}_h$ if and only if $T \in \mathcal{T}_\a$.
\end{enumerate}

To conclude this discussion we note that only the blending function $\varphi$
and the finite element mesh $\mathcal{T}_h$ are free approximation parameters,
while the subdomains and corresponding radii are derived (in particular,
$\Omega = \bigcup \mathcal{T}_h$). In our analysis we will require bounds
on the ``shape regularity'' of $\varphi$, $\mathcal{T}_h$, and the domains defined above:

\begin{assumption} \label{assumption-shapereg}
   In addition to (1)--(4) there exist constants $C_{\mathcal{T}_h}, C_\varphi > 0$, which shall be
   fixed throughout, such that
   \begin{align*}
      \|\nabla^j \varphi\|_{L^\infty} \leq C_\varphi R_\a^{-j} \qquad
      \text{for $j = 1,2,3$, \quad and} \qquad
      \max_{T \in \mathcal{T}_h} \frac{\sigma_T}{\rho_T} \leq C_{\mathcal{T}_h},
   \end{align*}
   where $\sigma_T$ denotes the radius of the smallest ball
circumscribing $T$ and $\rho_T$ the radius of the largest ball contained in $T$.  Defining the mesh size function
\[
h(x) := \max_{\substack{T \in \mathcal{T}_h: \\ x \in T}} \sigma_T,
\]
there exists $s \geq 1$ such that the mesh satisfies the growth condition
\[
|h(x)| \leq C_{\mathcal{T}_h}\Big(\frac{|x|}{R_{\a}}\Big)^s, \quad   |x| \geq R_\a.
\]

Moreover, there exists  $C_{\rm o} > 0$ and a positive integer $\lambda$ such that
\begin{equation}\label{def_constant_Co}
  R_{\rm o} \leq C_{\rm o} R_{\rm core}^{\lambda} \quad \mbox{\dao{and} } \quad \frac{1}{4} R_{\a} \leq R_{\rm core} \leq \frac{3}{4} R_{\rm a}.
\end{equation}
\end{assumption}

While $C_\varphi$ will feature heavily in our analysis, the parameter
$C_{\mathcal{T}_h}$ will only enter implicitly in the form of constants in
interpolation error estimates.  The condition $\frac{1}{4} R_{\a} \leq R_{\rm core} \leq \frac{3}{4} R_{\rm a}$ greatly simplifies the analysis.  It is likely this could be weakened by an extremely refined analysis as can be done in one dimension~\cite{li2012positive}, but the asymptotic estimates obtained would be unchanged with the exception of an improved prefactor so we do not pursue this.  \dao{Moreover, though one can generate blending functions which satisfy these assumptions using splines, we point out that in practical implementations one can relax the regularity requirements on the blending functions, and this has provided no loss in performance in simulations carried out for  lattices in~\cite{bqcf13}.}

\subsection{Main Result}
\label{sec:main-result-subsec}
Our main result concerns the existence of a solution to~\eqref{bqcf_approx} and an estimate on the error committed.
\begin{theorem}\label{main_thm}
   Suppose that Assumptions~\ref{assumption1},~\ref{assumptionSite}, and~\ref{assumption2} are valid.
   Then there exists $R_{\rm core}^*$ such that, for any approximation parameters
   satisfying Assumption~\ref{assumption-shapereg} as well as
   $R_{\rm core} \geq R_{\rm core}^*$,
   there exists a solution $(U^{\bqcf}, \bm{p}^{\bqcf}) \in
   \bm{\mathcal{U}}_{h,0} \times \bm{\mathcal{P}}_{h,0}$ to the BQCF
   equations~\eqref{bqcf_approx} that satisfies
\begin{equation}\label{main_estimate}
\begin{split}
\|\nabla IU^\infty -& \nabla U^{\bqcf} \|_{L^2(\mathbb{R}^d)} +
 \| I \bm{p}^\infty - \bm{p}^{\bqcf}\|_{L^2(\mathbb{R}^d)}
 \lesssim~ \gamma_{\rm tr} \Big(
   \|h \nabla^2 \tilde{I}U^{\infty}\|_{L^2(\Omega_\c)} \\
   &\qquad + \|h \nabla \tilde{I}\bm{p}^\infty \|_{L^2(\Omega_\c)} +
   \|\nabla\tilde{I} {U}^{\infty}\|_{L^2(\Omega_{\rm ext})}
   + \|\tilde{I} \bm{p}^\infty\|_{L^2(\Omega_{\rm ext})}\Big),
\end{split}
\end{equation}
% \begin{equation}\label{main_estimate}
% \begin{split}
% \|\nabla IU^\infty -& \nabla U^{\bqcf} \|_{L^2(\mathbb{R}^d)} + \sum_{\alpha=0}^{S-1} \| Ip^\infty_\alpha - p^{\bqcf}_\alpha\|_{L^2(\mathbb{R}^d)} \lesssim~ \gamma_{\rm tr}\big(\|h(x) \nabla^2 \tilde{I}U^{\infty}\|_{L^2(\Omega_\c)} \big. \\
% &\quad+ \big. \sum_\alpha\|h \nabla \tilde{p}_\alpha^\infty\|_{L^2(\Omega_\c)} + \|\nabla\tilde{U}^{\infty}\|_{L^2(\Omega_{\rm ext})} + \sum_\alpha\|\tilde{I}p_\alpha^\infty\|_{L^2(\Omega_{\rm ext})}\big),
% \end{split}
% \end{equation}
where
\[
\gamma_{\rm tr} := \begin{cases} &\sqrt{1 + \log(R_{\rm o}/R_{\a})}, \quad \mbox{if $d = 2$,} \\
                                 &1, \quad \mbox{if $d = 3$.} \end{cases}
\]
The implied constant, as well as $R_{\rm core}^*$, may depend on $C_\varphi$ and $C_{\mathcal{T}_h}$,
the interatomic potentials $V, V_\xi$, the maximum of $|\rho|$ for $\rho \in \mathcal{R}_1$, and the stability constant, $\gamma_\a$.
% is allowed to depend upon the interaction range, $\mathcal{R}$, the ratio $\kappa$, $R_{\rm core}^*$, and the shape regularity of the finite element mesh, $C_{\mathcal{T}_h}$, and the coercivity constant, $\gamma_\a$.
\end{theorem}

\begin{remark}
   The quantity $\gamma_{\rm tr}$ arises from a trace inequality that is needed when estimating interpolants on the
   atomistic mesh in terms of interpolants on the continuum mesh, c.f.~[Lemma
   4.6]\cite{blended2014}.
\end{remark}

\medskip

Section~\ref{analysis} is devoted to proving Theorem~\ref{main_thm}, but before
we embark on this, we first demonstrate how the error estimate can be combined
with the regularity estimates of Theorem~\ref{decay_thm1} to yield an optimised
BQCF scheme with balanced approximation parameters. This is followed by a
numerical test on a Stone--Wales defect in graphene, validating our theoretical
convergence rates.

\subsection{Optimal parameter choices}
Once we restrict ourselves to a Cauchy--Born energy with $\mathcal{P}_1$ discretisation as the continuum model, the free parameters in the design of the BQCF method are the domain, $\Omega$;
blending function, $\varphi$; and finite element mesh, $\mathcal{T}_h$ in the
sense that once these are set according to Section~\ref{sec:approx_params},
then the BQCF method~\eqref{bqcf_approx} is fully formulated.  Ideally, these
parameters should be chosen in an optimal way so as to obtain the most efficient
method.

% Clearly, it is not wise to have a finite element mesh with elements
% whose size is on the order of the lattice spacing, for example.  This naturally
% gives rise to the question of how the finite element mesh \textit{should} be
% graded.  Other natural questions are the choice of blending function and how
% large of a domain $\Omega$ to take.  We have already remarked that the
% forthcoming analysis only relies on the assumptions made on the blending
% function so we simply note that there are many choices to make for the blending
% function which meet these requirements, see e.g.~\cite{bqce12}, and these
% choices will not effect the asymptotic estimates but may improve the prefactor.
% The finite element mesh and choice of $\Omega$ may however be optimized.

The choice of blending function is, in the case of the BQCF method,
arbitrary as long as Assumption~\ref{assumption-shapereg} is satisfied.
There are many choices to make for the blending
function which meet these requirements, see e.g.~\cite{bqce12}.

The finite element mesh and hence the choice of $\Omega$ may, however, be optimized.
The key to choosing the finite element mesh and size of $\Omega$ lies in applying the
decay results of Theorem~\ref{decay_thm1} to our error estimate~\eqref{main_estimate},~\cite{acta.atc,bqcf13,bqce12}.  \dao{In obtaining our optimized parameters, we do not provide rigorous proofs but instead use heuristic assumptions to arrive at approximate choices which can then be rigorously analyzed numerically.  To start, we}
 assume that the mesh size function $h(x)$ is radial, i.e.,
$h(x) \equiv h(|x|)$. Then, ignoring logarithmic factors in $\gamma_{\rm tr}$ \dao{and employing the estimate $|1+r|^{-1} \lesssim r^{-1}$ for $r \geq 1$}, the error
estimate~\eqref{main_estimate} can be further estimated by
\begin{equation*}
\begin{split}
&\|\nabla IU^\infty - \nabla U^{\bqcf} \|_{L^2(\mathbb{R}^d)}^2 +
\| I\bm{p}^\infty - \bm{p}^{\bqcf}\|_{L^2(\mathbb{R}^d)}^2
\\
&\qquad\qquad\qquad\qquad\qquad\qquad\lesssim \int_{r_{\rm core}}^{R_{\c}}|h(r)|^2r^{-3-d} \, dr  + \int_{1/2r_{\rm i}}^{\infty}r^{-1-d} \, dr
\end{split}
% \label{main_esty1}
\end{equation*}
\dao{Next, we note that from the definitions of $\Omega_{\rm c}, \Omega$, and $r_{\rm i}$, we have $r_{\rm i} = R_{\rm c} + 4r_{\rm buff}$ so that we may make the replacement $r_{\rm i} \approx R_{\rm c}$.} Denoting the number of degrees of freedom by ${\rm DoF}$ (nodes in the continuum finite element mesh times the number of species in the multilattice), we can then carry out an optimization problem consisting of minimizing this error estimate subject to a fixed number of degrees of freedom, ${\rm DoF}$.   This problem is exactly the same as for the Bravais lattice and is
\begin{align*}
\min_{h \in L^2, R_{\c} > 0}  \int_{r_{\rm core}}^{R_{\c}}|h(r)|^2r^{-3-d} \, dr  + \int_{1/2\dao{R_{\rm c}}}^{\infty}r^{-1-d} \, dr.
\end{align*}
\dao{This problem is solved in~\cite{olsonThesis} where it is found that there are approximate minimisers of the form $h(r) =
\big(r / R_\a \big)^{\frac{1+d}{1+d/2}}$.  A simplified approximate solution can be obtained by first minimizing $\int_{r_{\rm core}}^{R_{\c}}|h(r)|^2r^{-3-d} \, dr$ with respect to $h$ where the same expression for $h$ will result, but instead of also minimizing with respect to $R_{\rm c}$, one can simply note that the error then becomes
\begin{equation}\label{error_dof}
\int_{r_{\rm core}}^{R_{\c}}|h(r)|^2r^{-3-d} \, dr  + \int_{1/2R_{\rm c}}^{\infty}r^{-1-d} \, dr \lesssim r_{\rm core}^{-d-2} + R_{\c}^{-d} \lesssim~ R_{\rm a}^{-d-2} + R_{\c}^{-d}.
\end{equation}
In order to balance the sources of error, one should take  $R_{\c} = R_\a^{\dao{2/d}+1}$.  Finally, by simply writing the number of degrees of freedom as the sum of those in the atomistic and continuum regions, it is possible to derive the result that $\#{\rm DoF} \approx  R_{\rm a}^d$; further details can be found in~\cite{olsonThesis,Dev2013,acta.atc,blended2014}.}
% but one of
% the key points is that the number of degrees of freedom, ${\rm DoF}$, can be
% approximated by
% \[
% \#{\rm DoF} = \int_{R_\a}^{R_\c} h^{-d}(r)r^{d-1}\, dr +  R_\a^d
% \]
% from which it is then evident that the derived expressions for $h(r)$ and $R_{\c}$ in terms of $R_{\a}$ can then be back substituted to yield an expression for $\#{\rm DoF}$ in terms of $R_\a$ alone.

After making the estimation $\gamma_{\rm tr} \leq (\log {\rm DoF})^{1/2}$~\cite{blended2014} for $d = 2$,  the main error estimate, ~\eqref{main_estimate}, currently written in terms of solution regularity, may now be replaced \dao{by an estimate of~\eqref{error_dof} in terms of computational cost since $\#{\rm DoF} \approx  R_{\rm a}^d$:}
\begin{equation}\label{main_esty2}
\begin{split}
&\|\nabla IU^\infty - \nabla U^{\bqcf} \|_{L^2(\mathbb{R}^d)}^2 +
   \| I \bm{p}^\infty - \bm{p}^{\bqcf} \|_{L^2(\mathbb{R}^d)}^2 \\
	&\qquad\qquad\qquad\qquad\qquad\qquad\lesssim~ \begin{cases} ({\rm DoF})^{-1-2/d} \log{\rm DoF}, &\quad d = 2 , \\
	({\rm DoF})^{-1-2/d}, &\quad d = 3,
	\end{cases}
\end{split}
\end{equation}
which exactly matches the rate for the Bravais lattice case~\cite{blended2014}.
This is due to the fact that the limiting factor in both error estimates is
the $\mathcal{P}_1$ finite element approximation.

\begin{remark}
   \dao{In the Bravais lattice analysis~\cite{blended2014}, the expression of $R_{\c}$ in terms of $R_{\a}$
	 is incorrect which has led to an error in the expression for the error estimate in terms of the degrees of freedom. In that paper, a different mesh scaling is also used, but should the same mesh scaling be used, the error estimates in terms of the degrees of freedom would be identical up to a constant prefactor.}
\end{remark}

\subsection{Numerical tests} \label{num}
In addition to providing a means to estimating the computational cost of the
BQCF  method, the estimate~\eqref{main_esty2} is also convenient to verify
 numerically. We have carried this out for
a Stone--Wales defect in graphene using both the BQCF method and a fully
atomistic method.

For the latter we simply minimize the full atomistic
energy over displacements that are non-zero only on the computational domain
$\Omega$ (clamped boundary conditions).  Using the methods
discussed in Section~\ref{analysis}, it is not difficult to show that the
solution, $(U^{\rm Dir},\bm{p}^{\rm Dir})$, to this atomistic Galerkin method
exists and satisfies the error estimate
\begin{equation}\label{minor_esty1}
   \|\nabla IU^\infty - \nabla U^{\rm Dir} \|_{L^2(\mathbb{R}^d)}
      + \| I\bm{p}^\infty - \bm{p}^{\rm Dir}\|_{L^2(\mathbb{R}^d)}
      \lesssim~ ({\rm DoF})^{-1/2}.
\end{equation}

We now set the model up for the Stone--Wales defect in graphene, recalling first
 the multilattice parameter values given
in Section~\ref{model}.  We choose a Stillinger-Weber~\cite{stillinger1985}
type interatomic potential with a pair potential and bond angle potential component.
The interaction range we consider is
\begin{align*}
\mathcal{R} = \big\{&(\rho_1 00), (\rho_2 00), (-\rho_1 00),(-\rho_2 00),(\rho_1-\rho_2 00),(\rho_2-\rho_1 00),\\
&(001),(010),(-\rho_2 01),(\rho_2 10),(-\rho_1 01),(\rho_1 10),\\
& (\rho_1 11), (\rho_2 11), (-\rho_1 11),(-\rho_2 11),(\rho_1-\rho_2 11),(\rho_2-\rho_1 11)\big\},
\end{align*}
which is  depicted in Figure~\ref{fig:multilattice}.  In this notation, the site potential is given by
\begin{align*}
&\hat{V}(D\bm{y}(\xi)) = \sum_{\triple \in \mathcal{R}} \frac{1}{2}\phi(D_{\triple}\bm{y}(\xi)) + \vartheta(D_{(-\rho_1 01)}\bm{y}(\xi), D_{(-\rho_1-\rho_2 01)}\bm{y}(\xi))  \\
&\qquad+ \vartheta(D_{(-\rho_1 01)}\bm{y}(\xi), D_{(-\rho_2 01)}\bm{y}(\xi)) + \vartheta(D_{(-\rho_1-\rho_2 01)}\bm{y}(\xi), D_{(-\rho_2 01)}\bm{y}(\xi)) \\
&\qquad+ \vartheta(D_{(\rho_1 10)}\bm{y}(\xi), D_{(\rho_1+\rho_2 10)}\bm{y}(\xi)) + \vartheta(D_{(\rho_1 10)}\bm{y}(\xi), D_{(\rho_2 10)}\bm{y}(\xi)) \\
&\qquad+ \vartheta(D_{(\rho_1+\rho_2 10)}\bm{y}(\xi), D_{(\rho_2 10)}\bm{y}(\xi)),
\end{align*}
where $\phi(r) = r^{-12} - 2r^{-6}$ is a pair potential term and
\[
\vartheta(r_1, r_2) = \Big(\frac{r_1\cdot r_2}{|r_1| \,  |r_2|}+1/2\Big)^2
\]
is a three-body term that penalizes angles that differ from $\frac{2\pi}{3}$.

 % The interaction range and potential are defined to include nearest neighbor bonds for the three-body term and next nearest neighbor bonds for the pair potential component.

The Stone--Wales defect shown in Figure~\ref{fig:stone} is obtained by rotating
the bond between the two carbon atoms at the origin site by ninety degrees about
the midpoint of this bond. One way of incorporating this defect into our
framework is to define a reference configuration $(Y_0, p_1)$ where $Y_0(\xi) =
\mF \xi$ for all $\xi \neq 0$ with $\mF$ and $p_1$ given by the graphene
parameters in~\eqref{graph_param}. At the origin, we set $Y_0(0) = \mbox{Rot}(0) $
and $p_1(0) = \mbox{Rot}(p_1)$, where $\mbox{Rot}$ represents the ninety degree
rotation about the midpoint of the segment ${\rm conv}\{0, p_1\}$.  Then we set
$V_\xi(D(U,p)(\xi)) = \hat{V}(D(Y_0 + U,p_1 + p)(\xi))$.

\begin{figure}
   \centering
\subfigure[A perfect graphene sheet. ]
{\includegraphics[width=0.35\textwidth]{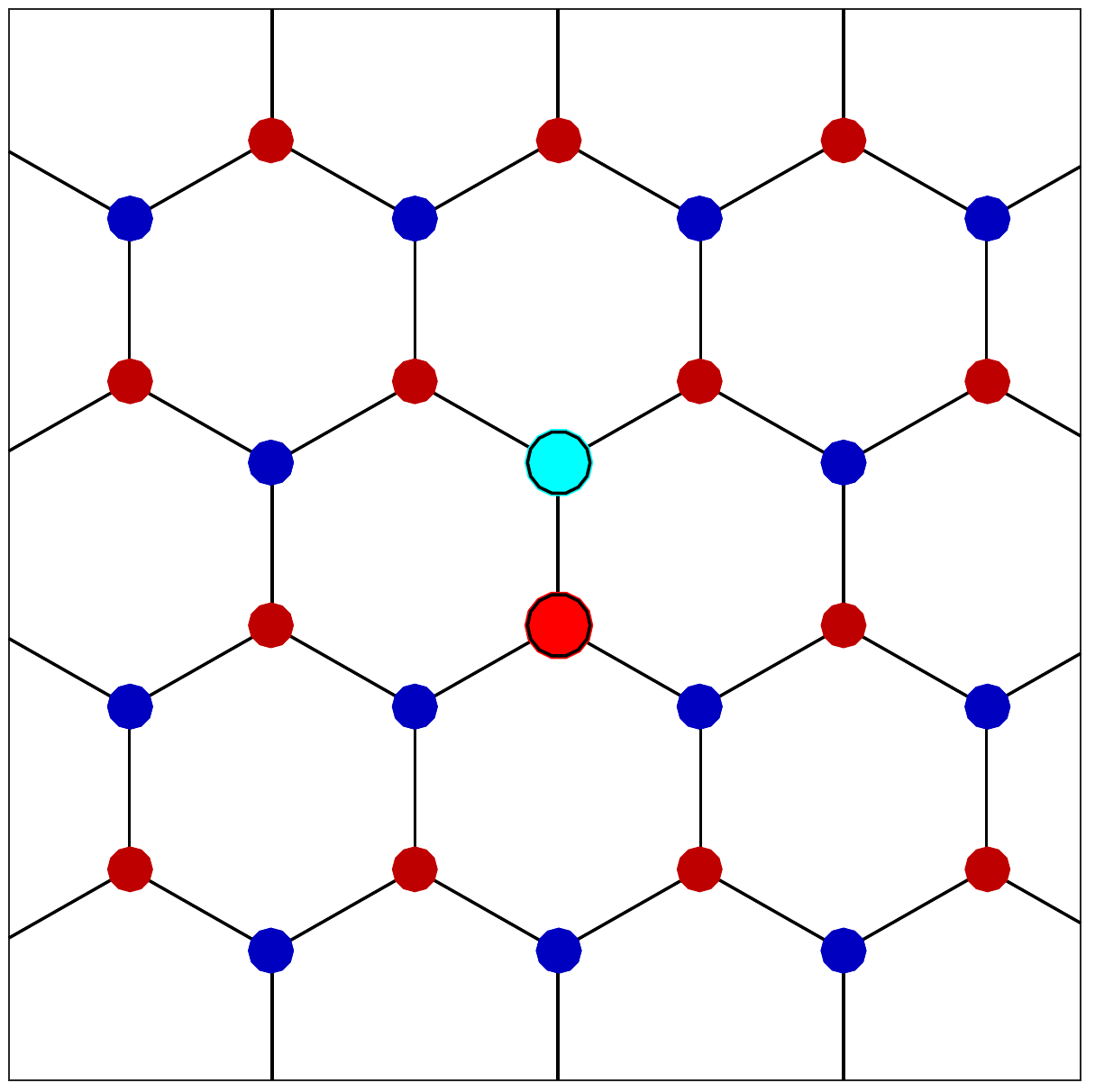}}
\subfigure[An unrelaxed Stone--Wales defect. ]
{\includegraphics[width=0.35\textwidth]{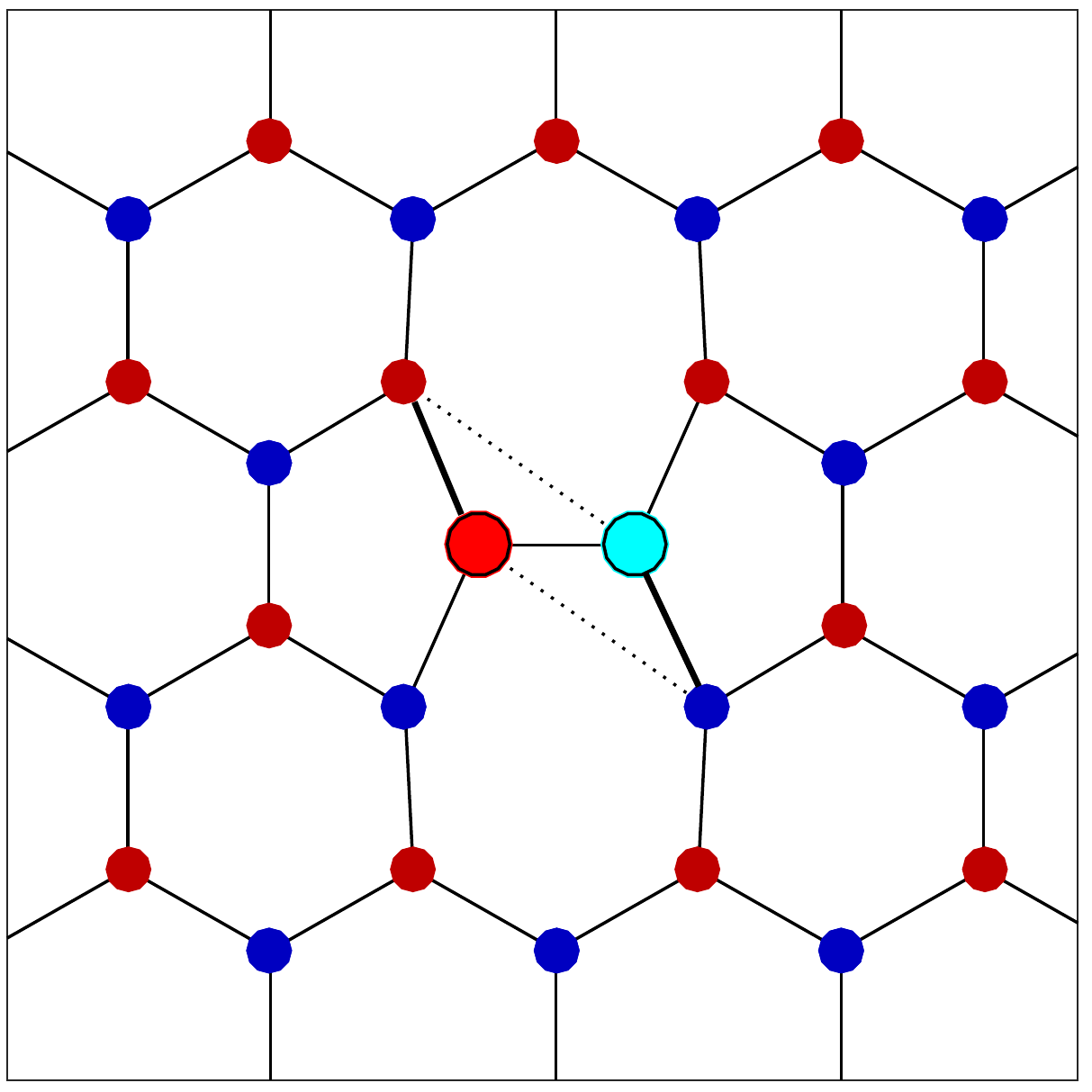}}
\caption{Examples of a perfect graphene sheet and a Stone--Wales defect.
The dotted lines in the right display indicate bonds that are broken during the rotation of the highlighted atoms.}\label{fig:stone}
\end{figure}

We choose hexagonal domains for $\Omega_{\rm core}, \Omega_\a, \Omega$, etc.,
and use a blending function which approximately minimizes the $L^2$ norm of
$\nabla^2 \varphi$ on $\Omega_{\b}$~\cite{bqce12}.  We select the inner width,
$r_{\rm core}$, of the hexagon $\Omega_{\rm core}$ to be from the range ${{R_{\a}=}}\left\{8,
12,16,20,24\right\}$ with $\kappa = 1/2$, and then the remaining domains are
chosen as scaled hexagons satisfying the requirements of Section~\ref{bqcf} and
Theorem~\ref{main_thm} \dao{(see Figure 10 in~\cite{bqcf13} for a representative illustration of this domain decomposition for a Bravais lattice)}.  Finally, our finite element mesh is graded
radially with approximate mesh size $h(r) = \big(\frac{r}{R_\a}\big)^{3/2}$ as described earlier in this
section with $d = 2$.  The BQCF equations were solved by a preconditioned nonlinear conjugate
gradient algorithm with line-search based on force-orthogonality only
{\helen{(in BQCF there is no energy functional for which descent can be imposed).}}

In Figure~\ref{fig:error} we show the error in the displacement gradients and the single graphene
shift vector for the computed BQCF solution versus the number of degrees of
freedom. Both match our theoretical predictions
from~\eqref{main_esty2} and indeed demonstrate that the error
estimates are sharp (up to logarithms). We also show the error committed by the
atomistic Galerkin method (which is estimated in~\eqref{minor_esty1}),
to demonstrate the practical gain achieved by the BQCF method.

\begin{figure}
   \centering
\subfigure[Error in displacement field for Stone--Wales defect.]
   {\includegraphics[width=0.45\textwidth]{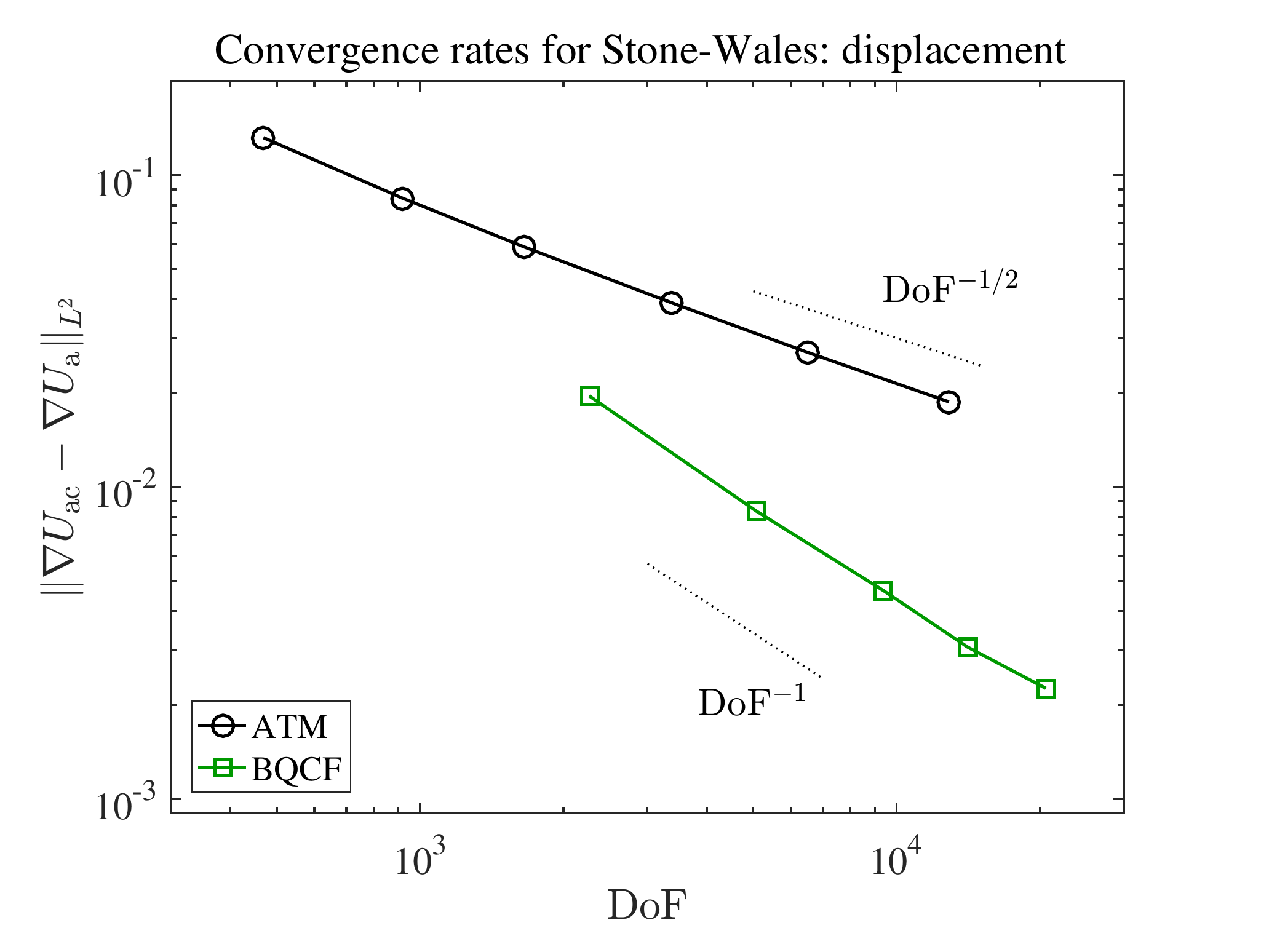}}
\subfigure[Error in shift field for Stone--Wales defect. ]
{\includegraphics[width=0.45\textwidth]{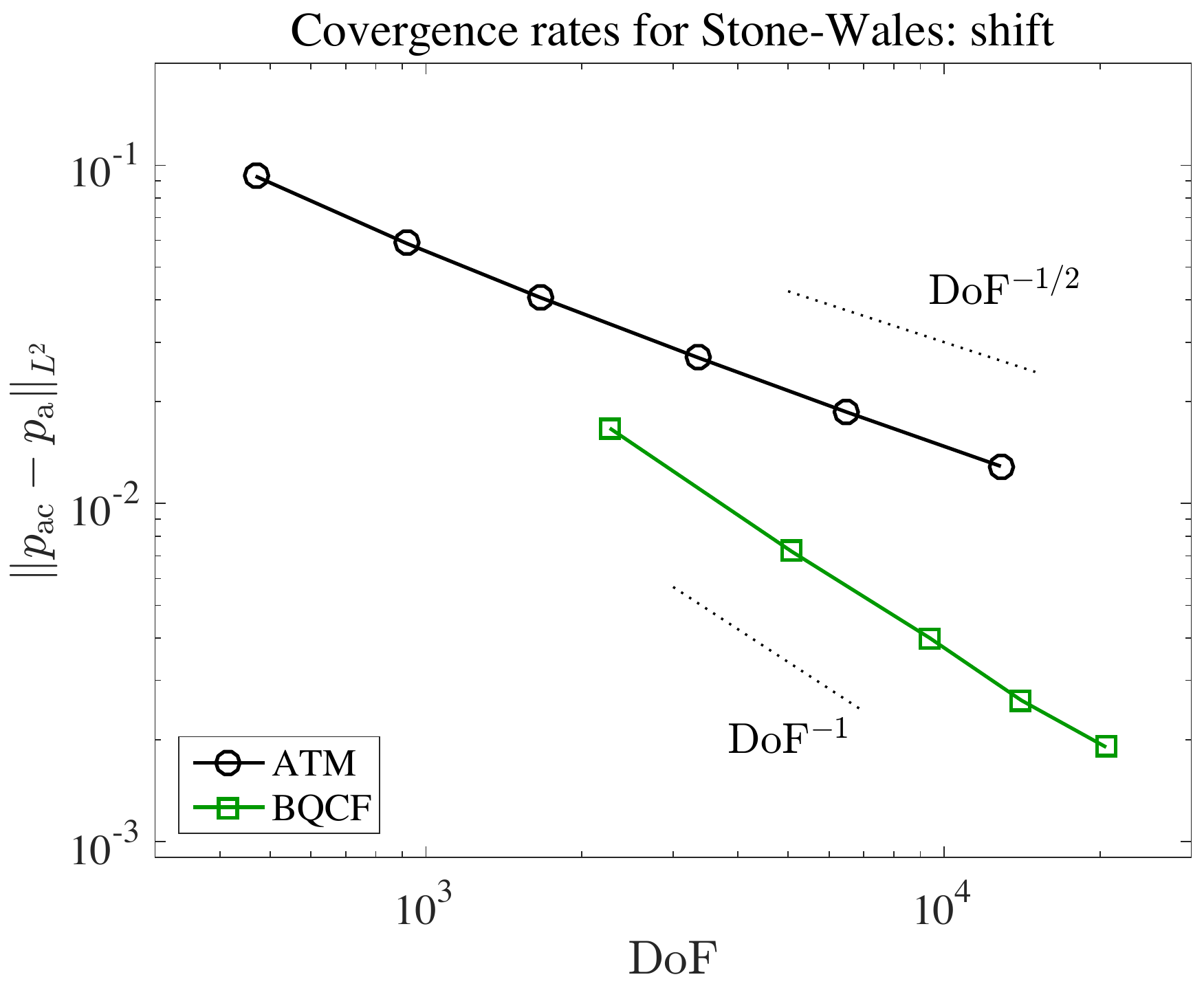}}
\caption{BQCF error plotted against degrees of freedom. {\helen{We have also plotted the ``purely atomistic'' error, denoted by ATM, which is the solution obtained by truncating the infinite dimensional atomistic problem to a finite domain using homogeneous Dirichlet boundary conditions.}}
}\label{fig:error}
\end{figure}

\FloatBarrier

\section{Proofs} \label{analysis}

The remainder of this paper is devoted to proving our main result, Theorem~\ref{main_thm}.  As in \cite{blended2014}, the abstract
framework for the proof is provided by the
inverse function theorem~\cite{acta.atc,ortnerInverse,hubbard2009}, {\helen{which we recall for reference and which is used to establish well-posedness of the nonlinear BQCF variational equation in Theorem ~\ref{main_thm}.}}

\begin{theorem}[Inverse Function Theorem \cite{ortnerInverse,hubbard2009}]\label{inverseFunctionTheorem}
Let $X$ and $Y$ be Banach spaces with $f:X \to Y$, $f \in {\rm{C}}^1(U)$ with $U \subset X$ an open set containing $x_0$.
Suppose that $\eta > 0, \sigma > 0$, and $L > 0$ exist such that $\|f(x_0)\|_Y < \eta$,  $\delta f(x_0)$ is invertible with $\|\delta f(x_0)^{-1}\|_{\calL(Y,X)} < \sigma$, $B_{2\eta\sigma}(x_0) \subset U$, $\delta f$ is Lipschitz continuous on $B_{2\eta\sigma}(x_0)$ with Lipschitz constant $L$, and $2L\eta\sigma^2 < 1$. Then there exists a ${\rm C}^1$ inverse function $g:B_{\eta}(y_0) \to B_{2\eta\sigma}(x_0)$ and thus an element $\bar{x} \in X$ such that $f(\bar{x}) = 0$ and
\begin{align*}
\|x_0 - \bar{x}\|_{X} <~& 2\eta\sigma.
\end{align*}
\end{theorem}

\medskip

The nonlinear operator we consider is the variational BQCF operator $\mathcal{F}^{\rm BQCF}(U, \bm{p})$, and the point about which we linearize is $x_0 = (U_h, \bm{p}_h) := \Pi_h(U^\infty, \bm{p}^\infty)$ where $\Pi_h$ is a projection operator defined in the following section.  In Section~\ref{cons} we prove a consistency estimate on the residual $\mathcal{F}^{\rm BQCF}(U_h, \bm{p}_h)$:
\begin{equation}\label{cons_est}
\begin{split}
\sup_{\|(W,\bm{r})\|_{\rm ml} = 1} \big|\< \mathcal{F}^{\rm BQCF}(U_h, \bm{p}_h), (W,\bm{r}) \>\big| &\lesssim \|h \nabla^2 \tilde{U}^\infty\|_{L^2(\Omega_\c)} + \|h\nabla \tilde{\bm{p}}^\infty\|_{L^2(\Omega_\c)}  \\
&\qquad\qquad + \|\nabla \tilde{U}^\infty\|_{L^2(\Omega_{\rm ext})}
  + \|\tilde{\bm{p}}^\infty\|_{L^2(\Omega_{\rm ext})}.
\end{split}
\end{equation}
The invertibility condition on the derivative of $\mathcal{F}^{\bqcf}$ is proven as a coercivity condition in Section~\ref{stab} where we show that
\begin{equation}\label{stab_est}
\<\delta \mathcal{F}^{\rm BQCF}(U_h, \bm{p}_h){\helen{(W,\bm{r})}},(W,\bm{r}) \> \gtrsim \|(W,\bm{r})\|_{\rm ml}^2, \quad \forall (W,\bm{r}) \in \bm{\mathcal{U}}_{h,0} \times \bm{\mathcal{P}}_{h,0},
\end{equation}
provided that the atomistic region is sufficiently large.

After we prove these two estimates, in Section \ref{sec:proof_main_result} we
combine them with a Lipschitz estimate on $\delta\mathcal{F}^{\rm bqcf}$ and
apply the inverse function theorem to prove Theorem~\ref{main_thm}.

% We will begin by describing a number of technical results needed throughout the analysis, which are extensions of analogous results for Bravais lattices~\cite{blended2014}.  In section~\ref{cons} we estimate the BQCF {\em consistency error}: the residual of the BQCF variational operator when evaluated at a projection of the exact solution onto the approximation space.  In section~\ref{stab} we {\em stability} of BQCF: the linearized BQCF operator is invertible at this same projected solution. These two fundamental results then allow us to invoke the inverse function theorem which completes the proof of Theorem~\ref{main_thm}.

Throughout this analysis, we continue to use the modified Vinogradov notation $x \lesssim y$, where the implied constants are allowed to depend on the shape regularity constants $C_{\mathcal{T}_h}, C_{\rm o}$ ({\helen{which are defined in Assumption~\ref{assumption-shapereg} and \eqref{def_constant_Co} }}), the interatomic potentials (and their interaction range), and the stability constant $\gamma_\a$.

\subsection{Cauchy--Born Modeling Error}\label{tech}
In preparation for the consistency analysis in Section~\ref{cons} we first establish several auxiliary results about the Cauchy--Born model.

A central technical tool in the analysis of AtC coupling methods is the ability to compare discrete atomistic displacements which are the natural atomistic kinematic variables ({\helen{recall that the atomistic displacements are equivalent to atomistic site displacements plus atomistic shift vectors}}), and continuous displacement and shift fields which capture the continuum kinematics.  We have already introduced several interpolants which serve this task: a micro-interpolant, $I$; a finite element interpolant, $I_h$; and a smooth interpolant, $\tilde{I}$.  We will also introduce a quasi-interpolant in this section which will allow us to define an analytically convenient atomistic version of stress~\cite{theil2012}.

We use $\bar{\zeta}(x)$ to denote the nodal basis function associated with the origin for the atomistic finite element mesh $\mathcal{T}_\a$ and $\bar{\zeta}_\xi(x) := \bar{\zeta}(x-\xi)$ to denote the nodal basis function at site $\xi$.   We may then write the micro-interpolant $Iu = \bar{u}$ as
\begin{equation*}\label{def_baru}
\bar{u}(x) = \sum_{\xi \in \mathcal{L}} u(\xi) \bar{\zeta}(x-\xi).
\end{equation*}
The quasi-interpolant of $u$ is then defined by a convolution with $\bar{\zeta}$
\begin{equation}\label{def_quasi_interp}
u^*(x) := (\bar{\zeta} * \bar{u})(x).
\end{equation}

It will later be important that this convolution operation is invertible and stable. This is a consequence of~\cite[Lemma 5]{atInterpolant}, which we state here for reference.

\begin{lemma}\cite[Lemma 5]{atInterpolant}\label{iso_lemma}
For a given atomistic displacement, $u$, there exists a unique atomistic displacement $\acute{u}$ with the property that $\bar{\zeta} * \bar{\acute{u}}(\xi) = u(\xi)$ for all $\xi \in \mathcal{L}$.
\end{lemma}

\medskip

One of the primary uses of the $u^*$ interpolant will be the development of an \textit{atomistic stress function} which can be compared to the continuum stress in the Cauchy--Born model~\cite{theil2012}.  The first variation of the continuum model may be written in terms of a stress tensor,
\begin{equation}\label{cont_tensor_eq1}
\begin{split}
&\la \delta \mathcal{E}^{\c} (U, \bm{q}), (W, \bm{r})\ra
=~ \int_{\mathbb{R}^d} \sum_{\triple} V_{,\triple}((\nabla_{\tau} U+q_{\delta}-q_{\gamma})_{\tripleTau \in \R})
\cdot \left(\nabla_{\rho} W+r_{\beta}-r_{\alpha}\right) \\
&~=~ \int_{\mathbb{R}^d} \sum_{\triple} V_{,\triple}(\nabla(U, \bm{q})) \otimes \rho : \nabla W  + \int_{\mathbb{R}^d} \sum_{\triple} V_{,\triple}(\nabla(U, \bm{q}))_{\tripleTau \in \R})
\cdot \left(r_{\beta}-r_{\alpha}\right)
\\
&~=~ \int_{\R^d} \sum_{\beta}[{\Scd}(U,\bm{q})(x)]_{\beta} : \nabla W
      + \int_{\R^d} \sum_{\alpha,\beta} [{\Scs}(U,\bm{q})(x)]_{\alpha\beta}
               (r_\beta - r_\alpha),
\end{split}
\end{equation}
where we defined
\begin{equation}\label{cont_stress_tensor}
\begin{split}
[{\Scd}(U,\bm{q})(x)]_{\beta} :=~& \sum_{\substack{\alpha, \rho: \\ \triple\in\mathcal{R}}} V_{,\triple}(\nabla(U, \bm{q})(x))\otimes \rho, \\
[{\Scs}(U,\bm{q})(x)]_{\alpha\beta} :=~& \sum_{\rho \in \mathcal{R}_1} V_{,\triple}(\nabla(U, \bm{q})(x)).
\end{split}
\end{equation}

To compare the atomistic and continuum models, we now construct an analogous atomistic stress tensor. Its definition will make it clear why we introduced the seemingly unnecessary sum over $\beta$ in the first group in \eqref{cont_tensor_eq1}. \dao{The basic idea is to extend the construction of~\cite{theil2012}: the argument $\nabla(U, \bm{q})(x)$ in~\eqref{cont_stress_tensor} will be replaced by a local averaging of first order finite difference approximations $D(U, \bm{q})(\xi)$ for $\xi$ near $x$.}

\begin{lemma}
   For $(U,\bm{q}) \in \bm{\mathcal{U}}$, define the atomistic stress tensors
   \begin{align}
      \label{eq:defn_Sa}
    \begin{split}
   [\Sad(U,\bm{q})(x)]_\beta :=~& \sum_{\substack{\alpha, \rho: \\ \triple\in\mathcal{R}}} \sum_{\xi \in \mathcal{L}} \big( V_{,\triple}\big(D(U,\bm{q})(\xi)\big) \otimes \rho \big) \omeRho(\xi-x),\\
   [\Sas(U,\bm{q})(x)]_{\alpha\beta} :=~& \sum_{\rho \in \mathcal{R}_1} \sum_{\xi\in\mathcal{L}} V_{,\triple} \big(D(U,\bm{q})(\xi)\big) \omega_{0}(\xi-x).
   \end{split} \\
   \label{omega_rho}
   \text{where} \qquad \omeRho(x) &:= \int_{0}^1 \barZeta(x+t\rho)dt.
   \end{align}
   Then
   \begin{equation}
      \label{atom_tensor_eq1}
      \begin{split}
         \big\la \delta \mathcal{E}^{\a}_{\rm hom}(U,\bm{q}), (W^{*}, \bm{r}^{*}) \big\ra
   =& \int_{\mathbb{R}^d}  \bigg\{ \sum_\beta [\Sad(U,\bm{q})]_\beta : \left(\nabla \bar{W}+\nabla \bar{r_\beta} \right) \\
    &\qquad\qquad+ \sum_{\alpha,\beta} [\Sas(U,\bm{q})]_{\alpha\beta}\cdot (\bar{r}_{\beta}-\bar{r}_\alpha) \bigg\} dx.
			%\big\la \delta \mathcal{E}^{\a}_{\rm hom}(U,\bm{q}), (W^{*}, \bm{r}^{*}) \big\ra
      %=& \helen{\int_{\mathbb{R}^d}  \bigg\{ \Sad(U,\bm{q}): \left(\nabla \bar{W}+\nabla \bar{\bm{r}}\right)
      % + \sum_{\alpha,\beta} [\Sas(U,\bm{q})(x)]_{\alpha\beta}\cdot (\bar{r}_{\beta}-\bar{r}_\alpha) \bigg\} dx.  }
      \end{split}
   \end{equation}
   {\helen{ where $W^{*}$ and  $\bm{r}^{*}$ are defined through  \eqref{def_quasi_interp}. }}
\end{lemma}
\begin{proof}
We start by computing {\helen{the first variation of $\mathcal{E}^{\a}_{\rm hom}(U,\bm{q})$ with the test pair $(W^{*}, \bm{r}^{*})$:}}
\begin{equation}\label{delEa_eq1}
\begin{split}
&\la \delta \mathcal{E}^{\a}_{\rm hom}(U,\bm{q}), (W^{*}, \bm{r}^{*})\ra\\
&=\sum_{\xi\in\mathcal{L}}\sum_{\triple\in\mathcal{R}}V_{,\triple}\big(D(U,\bm{q})(\xi)\big)
\cdot\Big(D_{\rho}W^*(\xi)+D_{\rho}r_{\beta}^*(\xi)+r_{\beta}^*(\xi)-r_{\alpha}^*(\xi)\Big).
\end{split}
\end{equation}
Arguing as in \cite[Eq. (2.4)]{theil2012} we obtain
\begin{align}
   \label{convo_finite_diff}
   D_{\rho}W^*(\xi)+D_{\rho}r_{\beta}^*(\xi)
   &= \int_{\mathbb{R}^d} \omeRho(\xi-x)\left( \nabla_{\rho}\bar{W}+\nabla_{\rho}\bar{r}_{\beta}\right)\, dx
   \qquad \text{and} \\
   \label{convo_beta_alpha}
   r^*_{\beta}(\xi)-r^*_{\alpha}(\xi) &= \int_{\mathbb{R}^d}\omega_0(\xi-x) \left(\bar{r}_{\beta}-\bar{r}_{\alpha}\right)\, dx.
\end{align}

% Applying \eqref{def_quasi_interp} and \eqref{omega_rho}, we can rewrite the
% finite difference as
% \begin{align*}
% D_{\rho}W^*(\xi) =~& \int_0^1 \frac{d}{dt}W^*(\xi + t\rho)\, dt =~ \int_0^1 \nabla_{\rho}W^*(\xi + t\rho)\, dt =~ \int_0^1 (\bar{\zeta}*\nabla_{\rho}\bar{W})(\xi + t\rho)\, dt\\
% =~& \int_{\mathbb{R}^d}\int_0^1\bar{\zeta}(\xi+t\rho -x)\nabla_\rho \bar{W}(x)\, dt\, dx = \int_{\mathbb{R}^d} \omeRho(\xi-x) \nabla_{\rho}\bar{W}(x)\, dx.
% \end{align*}
% Similarly,
% and similar calculations also hold for $W$ replaced by $r_{\beta}$. Thus,
% \begin{equation}\label{convo_finite_diff}
% D_{\rho}W^*(\xi)+D_{\rho}r_{\beta}^*(\xi)
% =\int_{\mathbb{R}^d} \omeRho(\xi-x)\left( \nabla_{\rho}\bar{W}+\nabla_{\rho}\bar{r}_{\beta}\right)\, dx.
% \end{equation}
%
% \dao{Meanwhile, $\omega_0  (x) = \bar{\zeta}(x)$ so
% \begin{equation}\label{convo_beta_alpha}
% r^*_{\beta}(\xi)-r^*_{\alpha}(\xi) =~ \int_{\mathbb{R}^d}\omega_0(\xi-x) \left(\bar{r}_{\beta}-\bar{r}_{\alpha}\right)(x)dx.
% \end{equation}}

Substituting \eqref{convo_finite_diff} and \eqref{convo_beta_alpha} into \eqref{delEa_eq1} and recalling the definitions of the atomistic stress tensors from~\eqref{eq:defn_Sa} yields the stated claim.
%% CO: no point in stating it again
% \begin{equation*}
% \begin{split}
% &\la \delta \mathcal{E}^{\a}_{\rm hom}(U,\bm{q}), (W^{*}, \bm{r}^{*})\ra =\int_{\mathbb{R}^d}  \bigg\{ \Sad(U,\bm{q}): \left(\nabla \bar{W}+\nabla \bar{\bm{r}}\right)
%        + \sum_{\alpha,\beta} [\Sas(U,\bm{q})(x)]_{\alpha\beta}\cdot (\bar{r}_{\beta}-\bar{r}_\alpha) \bigg\} dx.
% \end{split}
% \end{equation*}
\end{proof}

We refer to the error between the continuum and atomistic stress functions as the \textit{Cauchy--Born modeling error} and quantify it in the next lemma; see \cite{theil2012} for an analogous result for Bravais lattices.

\begin{lemma}\label{cb_error2}
   Assume that $U \in {\rm C}^{2,1}(\R^d; \R^n)$ and $p_\alpha \in {\rm C}^{1,1}(\R^d, \R^{n})$ for each $\alpha$. Fix $x \in \R^d$ and set
	\[
	r_{\rm cut} = \max_{\rho \in \mathcal{R}_1} |\rho|,  \qquad \nu_x := B_{2r_{\rm cut}}(0).
	\]
   1. If $\nabla U$ and $\bm{p}$ are constant in $\nu_x$, then
   \begin{equation}\label{cb_identity1}
   {[{\rm S}^{\rm a}_{\rm d}(U, \bm{p})(x)]_\beta = [{\rm S}^{\rm c}_{\rm d}(U, \bm{q})(x)]_\beta \quad \mbox{and} \quad [{\rm S}^{\rm a}_{\rm s}(U, \bm{p})(x)]_{\alpha\beta} = [{\rm S}^{\rm c}_{\rm s}(U, \bm{q})(x)]_{\alpha\beta}.}
   \end{equation}

   2. In general,
   {\begin{equation*}
      \label{eq:bound_Sa-Sc}
			\begin{split}
   \big| [S^{\rm a}_{\rm d}(U,p)(x)]_\beta - [S_{\rm d}^{\rm c}(U,p)(x)]_\beta\big|
   \lesssim~&
         \| \nabla^2 U \|_{L^\infty(\nu_x)} + \| \nabla \bm{q} \|_{L^\infty(\nu_x)},\\
   \big| [S^{\rm a}_{\rm s}(U,p)(x)]_{\alpha\beta} - [S_{\rm s}^{\rm c}(U,p)(x)]_{\alpha\beta}\big|
   \lesssim~&
         \| \nabla^2 U \|_{L^\infty(\nu_x)} + \| \nabla \bm{q} \|_{L^\infty(\nu_x)}.
  \end{split}
	\end{equation*}}
   with the implied constant depending only on the interatomic potential $V$.
\end{lemma}
\begin{proof}
   1. The identity \eqref{cb_identity1} is an immediate consequence of the definitions \eqref{cont_stress_tensor}, \eqref{eq:defn_Sa} and of
   \begin{equation*}\label{omega_properties_simplified}
   \sum_{\xi}\omega_\rho(\xi - x) = 1.
   \end{equation*}

   2. We define an auxiliary homogeneous displacement $(U^{\rm h}, \bm{q}^{\rm h})$ with $\nabla U^{\rm h} \equiv \nabla U(x)$ and $\bm{q}^{\rm h} \equiv \bm{q}(x)$. Then we have{
   \begin{align*}
      [{\rm S}^{\rm a}_{\rm d}(U, \bm{q})(x)]_\beta - [{\rm S}^{\rm c}_{\rm d}(U, \bm{q})(x)]_\beta
      = [{\rm S}^{\rm a}_{\rm d}(U, \bm{q})(x)]_\beta - [{\rm S}^{\rm a}_{\rm d}(U^{\rm h}, \bm{q}^{\rm h})(x)]_\beta.
   \end{align*}}
   Since we assumed that $V$ is twice continuously differentiable, with globally bounded second derivatives, we
   obtain
   \begin{align*}
      \big| [{\rm S}^{\rm a}_{\rm d} &(U, \bm{q})(x)]_\beta - [{\rm S}^{\rm c}_{\rm d}(U, \bm{q})](x)_\beta \big|
      = \big|[{\rm S}^{\rm a}_{\rm d}(U, \bm{q})(x)]_\beta - [{\rm S}^{\rm a}_{\rm d}(U^{\rm h}, \bm{q}^{\rm h})(x)]_\beta\big| \\
      &\dao{= \Big|\sum_{\substack{\alpha, \rho: \\ \triple\in\mathcal{R}}} \sum_{\xi \in \mathcal{L}} \big(\big[ V_{,\triple}\big(D(U,\bm{q})(\xi)\big) -  V_{,\triple}\big(D(U^{\rm h}, \bm{q}^{\rm h})(\xi)\big)\big]\otimes \rho \big) \omeRho(\xi-x)\Big| }\\
			&\dao{\lesssim \sum_{\substack{\alpha, \rho: \\ \triple\in\mathcal{R}}} \sum_{\xi \in \mathcal{L}} \big|D(U,\bm{q})(\xi) -  D(U^{\rm h}, \bm{q}^{\rm h})(\xi)\big| \omeRho(\xi-x)} \\
      &\lesssim \big\| \nabla U - \nabla U^{\rm h} \|_{L^\infty(\nu_x)}
         + \big\| \bm{q} - \bm{q}^{\rm h} \|_{L^\infty(\nu_x)} + \dao{\| \nabla^2 U \|_{L^\infty(\nu_x)} + \| \nabla \bm{q} \|_{L^\infty(\nu_x)}} \\
      &\lesssim \| \nabla^2 U \|_{L^\infty(\nu_x)} + \| \nabla \bm{q} \|_{L^\infty(\nu_x)},
   \end{align*}
	\dao{where in obtaining the last two inequalities we have used a Taylor expansion of the finite differences and the fact that $\omeRho(\xi-x)$ as defined in~\eqref{omega_rho} vanishes off of $\nu_x$.} The proof for the comparison of the ``shift'' stress tensors is nearly identical so is omitted.
\end{proof}

With this pointwise estimate, and using the fact that $\tilde{U}$ is piecewise polynomial, it is straightforward to deduce the following Cauchy--Born modeling error estimate over $\Omega_\c$.
\begin{corollary}\label{globel_stress}
%Assume that $U \in W^{3,\infty}(\mathbb{R}^d,\mathbb{R}^n)$ and each $q^\alpha \in W^{3,\infty}(\mathbb{R}^d,\mathbb{R}^n)$.
{For the atomistic and continuum stress tensors defined above,
\begin{equation*}\label{global_est}
\begin{split}
\big\|[{\rm{S}}^\a_{\rm d}(\tilde{U}^\infty,\tilde{\bm{q}}^\infty)]_\beta -[{\rm{S}}^\c_{\rm d}(\tilde{U}^\infty,\tilde{\bm{q}}^\infty)]_\beta \big\|_{L^2(\Omega_\c)} \lesssim~&  \|\nabla^2 \tilde{U}^\infty\|_{L^2(\Omega_\c)} + \|\nabla \tilde{\bm{q}}^\infty\|_{L^2(\Omega_\c)}, \\
\big\|[{\rm{S}}^\a_{\rm s}(\tilde{U}^\infty,\tilde{\bm{q}}^\infty)]_{\alpha\beta} -[{\rm{S}}^\c_{\rm s}(\tilde{U}^\infty,\tilde{\bm{q}}^\infty)]_{\alpha\beta} \big\|_{L^2(\Omega_\c)} \lesssim~&  \|\nabla^2 \tilde{U}^\infty\|_{L^2(\Omega_\c)} + \|\nabla \tilde{\bm{q}}^\infty\|_{L^2(\Omega_\c)}.
\end{split}
\end{equation*}}
\end{corollary}

\begin{remark}
The stress estimates for a multilattice are one order lower in terms of derivatives than the corresponding Bravais lattice estimates.  A refined analysis shows that this estimate cannot be improved without an underlying point symmetry for the multilattice.  When this symmetry is present in multilattices, it is possible to define a symmetrized Cauchy--Born energy with an improved estimate~\cite{koten2013}.
\end{remark}

\subsection{Consistency}\label{cons}
Our first task in completing the residual estimate~\eqref{cons_est} is to define the projection from atomistic functions to finite element functions satisfying the Dirichlet boundary conditions so we first truncate the solution to a finite domain.  For that, let $\eta$ be a smooth ``bump function'' with support in $B_{1}(0)$ and equal to one on $B_{3/4}(0)$.   Let \dao{$A_R$ be an ``annular region'' containing the support of $\nabla (I \eta(x/R))$, i.e,} $A_R := B_{R+2r_{\rm buff}}(0)\setminus B_{3/4R-2r_{\rm buff}} \supset {\rm supp}(\nabla (I \eta(x/R)))$ and define the truncation operator by
\begin{equation*}\label{trunc}
T_{R}u_\alpha(x) = \eta(x/R)\bigg(Iu_\alpha - \frac{1}{|A_R|}\int\limits_{A_R} Iu_0\, dx\bigg).
\end{equation*}
Further, let $S_h$ be the Scott--Zhang {\helen{quasi-interpolation}} operator~\cite{scott1990} onto the finite element mesh $\mathcal{T}_h$.  We then define the projection operator by
\begin{align}\label{proj_operator}
&\qquad\qquad\qquad\quad\Pi_{h} u_\alpha := S_h (T_{r_{\rm i}}u_\alpha), \quad \Pi_{h} \bm{u} :=~ \left\{\Pi_{h} u_\alpha\right\}_{\alpha = 0}^{S-1}, \\
&\Pi_{h} p_\alpha := \Pi_{h} (u_\alpha - u_0), \quad \Pi_{h} \bm{p} :=~ \left\{\Pi_{h} p_\alpha\right\}_{\alpha = 0}^{S-1}, \qquad \Pi_{h}(U,\bm{p}) := (\Pi_{h}U, \Pi_{h}\bm{p}).\nonumber
\end{align}
(Recall that $r_{\rm i}$ is the radius of the largest ball inscribed in $\Omega$.) Note that $\nabla \Pi_{h} u_\alpha$ as well as
\[
\Pi_{h} u_\alpha - \Pi_{h} u_\beta = S_h\big[\eta(x/r_{\rm i})\big( I u_\alpha - Iu_\beta \big)\big]
\]
have support contained in $\Omega$.
We also have the following approximation results.
\begin{lemma}\label{approx_lem}
Take $(U,\bm{p}) = \bm{u} \in \bm{\mathcal{U}}$.  Then
\begin{equation*}\label{approx_est}
\begin{split}
\|\nabla \bar{U} - \nabla \Pi_{h,R} U\|_{L^2(\mathbb{R}^d)} +  \|\bar{\bm{p}}_\alpha - \Pi_{h,R} \bm{p}_\alpha\|_{L^2(\mathbb{R}^d)} \lesssim~& \|h \nabla^2 \tilde{U}^\infty\|_{L^2(\Omega_\c)} + \|h\nabla \tilde{\bm{p}}^\infty\|_{L^2(\Omega_\c)} \\
&+ \| \nabla \tilde{U}\|_{L^2(\Omega_{\rm ext})} + \|\tilde{\bm{p}}\|_{L^2(\Omega_{\rm ext})}, \\
\|\nabla \tilde{U} - \nabla \Pi_{h,R} U\|_{L^2(\Omega_\c)} +  \|\tilde{\bm{p}}_\alpha - \Pi_{h,R} \bm{p}_\alpha\|_{L^2(\Omega_\c)} \lesssim~& \|h \nabla^2 \tilde{U}^\infty\|_{L^2(\Omega_\c)} + \|h\nabla \tilde{\bm{p}}^\infty\|_{L^2(\Omega_\c)} \\
&+ \| \nabla \tilde{U}\|_{L^2(\Omega_{\rm ext} \cap \Omega_\c)} + \|\tilde{\bm{p}}\|_{L^2(\Omega_{\rm ext} \cap \Omega_\c)}.
\end{split}
\end{equation*}
\end{lemma}
The proof is very similar to the proof of Lemma~\ref{lem:dense} (with only additional estimates required for the finite element interpolants) and therefore omitted. See also~\cite[Lemma 1.8]{olson2015} for similar estimates, the main difference being the usage of the Scott--Zhang interpolant which allows for $L^2$ interpolation bounds on $H^1$ functions, see~\cite{brenner2008,scott1990}.

We can now prove the bound~\eqref{cons_est}.

\begin{theorem}[BQCF Consistency]\label{consistency_thm}
Define $(U_h, \bm{p}_h):= \Pi_h(U^\infty, \bm{p}^\infty)$ where $(U^\infty, \bm{p}^\infty)$ satisfies Assumption~\ref{assumption2}. If Assumptions~\ref{assumption1} and~\ref{assumptionSite} are valid also and if the blending function, $\varphi$, and finite element mesh, $\mathcal{T}_h$, satisfy the requirements of Section~\ref{bqcf}, then the BQCF consistency error is bounded by
\begin{equation*}\label{cons_est1}
\begin{split}
\left|\la\mathcal{F}^{\bqcf}(U_h,\bm{p}_h), (W,\bm{r})\ra\right| \lesssim~ \gamma_{\rm tr} \, &\Big(\|h\nabla^2 \tilde{U}\|_{L^2(\Omega_\c)} + \|h\nabla \tilde{\bm{p}}_\alpha\|_{L^2(\Omega_\c)}  + \|\nabla \tilde{U}\|_{L^2(\Omega_{\rm ext})}   \\
&+  \| \tilde{\bm{p}}\|_{L^2(\Omega_{\rm ext})}\Big)\cdot \|(W,\bm{r})\|_{{\rm ml}}, \quad \forall (W,\bm{r}) \, \in \bm{\mathcal{U}}_{h,0} \times \bm{\mathcal{P}}_{h,0},
\end{split}
\end{equation*}
and $\gamma_{\rm tr}$ is a trace inequality constant (see Lemma 4.6 in \cite{blended2014}) given by
\[
\gamma_{\rm tr} = \begin{cases} &\sqrt{1 + \log(R_{\rm o}/R_{\a})}, \quad \mbox{if $d = 2$,} \\
                                 &1, \quad \mbox{if $d = 3$.} \end{cases}
\]
\end{theorem}

Before beginning the proof, we make some preliminary remarks. First, we observe that, since the Scott--Zhang interpolation operator is a projection it follows that
\[
   D_{\triple}U_h(\xi)=D_{\triple}U^\infty(\xi)
   \qquad \text{for} \quad
   \xi\in\mathcal{L}^{\a},
\]
where $\mathcal{L}^\a := \mathcal{L} \cap ({\rm supp}(1-\varphi) + \mathcal{R}_1)$.
Furthermore, since $\delta \mathcal{E}^{\a}(U^\infty, \bm{p}^\infty) = 0$, the residual error in the BQCF variational operator is equivalent to
\begin{equation}\label{test_going}
\begin{split}
\la & \mathcal{F}^{\bqcf}(U_h,\bm{p}_h), (W,\bm{r})\ra
-\la \delta \mathcal{E}^{\a}(U^\infty, \bm{p}^\infty), (U,\bm{q})\ra\\
&\quad= \la \delta \mathcal{E}^{\a}(U^\infty,\bm{p}^\infty), (1-\varphi)(W, \bm{r})\ra
+\la \delta \mathcal{E}^{\c}(U_h, \bm{p}_h), \big(I_h(\varphi W),I_h(\varphi\bm{r}) \big)\ra \\
&\qquad\qquad -\la \delta \mathcal{E}^{\a}(U^\infty, \bm{p}^\infty), (U,\bm{q})\ra,
\end{split}
\end{equation}
where $(W, \bm{r}) \in \bm{\mathcal{U}}_{h,0} \times \bm{\mathcal{P}}_{h,0}$ is an arbitrary given pair of test functions in the finite element test function space, while $(U, \bm{q})  \in \bm{\mathcal{U}} \times \bm{\mathcal{P}}$ is a test pair that we are free to choose. The obvious candidate choice is $(U, \bm{q}) = (W, \bm{r})$ in which case we would have
\begin{align*}
\la & \mathcal{F}^{\bqcf}(U_h,\bm{p}_h), (W,\bm{r})\ra
-\la \delta \mathcal{E}^{\a}(U^\infty, \bm{p}^\infty), (U,\bm{q})\ra\\
&\quad= -\la \delta \mathcal{E}^{\a}(U^\infty,\bm{p}^\infty), (\varphi)(W, \bm{r})\ra
+\la \delta \mathcal{E}^{\c}(U_h, \bm{p}_h), \big(I_h(\varphi W),I_h(\varphi\bm{r}) \big)\ra.
\end{align*}
The resulting residual error is concentrated only over $\Omega_\c$ due to $\nabla \varphi$ having support in $\Omega_\c$.  The issue in estimating this quantity is that when we convert the atomistic residual into the atomistic-stress format, the test function appears as a piecewise linear function with respect to the atomistic mesh $\mathcal{T}_\a$, whereas the test function is piecewise linear with respect to the graded mesh $\mathcal{T}_h$ in the continuum portion.  For this reason, we \dao{shall add correction terms to our previous candidate choice $(U,\bm{q}) = (W,\bm{r})$ via
\begin{equation}\label{test_choice}
U = W + (Z^* - \varphi W), \quad q_\alpha = r_\alpha + (z_\alpha^* - \varphi r_\alpha), \quad \alpha = 1,\ldots, S-1,
\end{equation}
where $(Z,\bm{z}) \in \bm{\mathcal{U}} \times \bm{\mathcal{P}}$ will be chosen to satisfy certain approximation estimates as stated in Lemma~\ref{interpolation_lemma} below.  The reason we use $Z^*$ and $z_\alpha^*$ instead of merely $Z$ and $z_\alpha$ is that we shall eventually make use of the atomistic stress representation from~\eqref{atom_tensor_eq1}.} The BQCF residual error from~\eqref{test_going} then becomes
\begin{equation}\label{residual_est_new}
\begin{split}
\la & \mathcal{F}^{\bqcf}(U_h,\bm{p}_h), (W,\bm{r})\ra
-\la \delta \mathcal{E}^{\a}(U^\infty, \bm{p}^\infty), (W + (Z^* - \varphi W),\bm{r} + (\bm{z}^* - \varphi \bm{r}))\ra\\
&\quad= \la \delta \mathcal{E}^{\c}(U_h, \bm{p}_h), \big(I_h(\varphi W),I_h(\varphi\bm{r}) \big)\ra - \la \delta \mathcal{E}^{\a}(U^\infty,\bm{p}^\infty), (Z^*, \bm{z}^*)\ra
\end{split}
\end{equation}
Moreover, since we are blending by site and using $\mathcal{P}_1$ elements for the shifts, we may use the same form for $Z$ and $\bm{z}$ as obtained in the simple lattice case~\cite{blended2014} for both displacements \textit{and} shifts.

\begin{lemma}\label{interpolation_lemma}
Suppose $W \in \bm{\mathcal{U}}_{h,0}$ and $\bm{r} \in \bm{\mathcal{P}}_{h,0}$.  Then for $f \in W^{1,2}_{\rm loc}(\mathbb{R}^d)$ and for $Z_h, Z,\zzhalpha, z_\alpha$ as defined above,
\begin{align}
\int_{\Omega_\c} f (\bar{Z} - Z_h) dx \lesssim~& \|\nabla f\|_{L^2(\Omega_\c)} \cdot \|\nabla Z_h\|_{L^2(\Omega_\c)}, \label{weight_inter}\\
\int_{\Omega_\c} f\cdot (\zzhalpha - \bar{z}_\alpha)\, dx \lesssim~& \| \nabla f\|_{L^2(\Omega_\c)} \cdot \| \zzhalpha \|_{L^2(\Omega_\c)}  \label{za_result} \\
\| \ZZh- \bar{Z}\|_{L^2(\Omega_\c)} \lesssim~& \|\nabla \ZZh\|_{L^2(\Omega_\c)}, \label{inter_diff}\\
\| \zzhalpha - \bar{z}_\alpha\|_{L^2(\Omega_\c)} \lesssim~& \| \zzhalpha\|_{L^2(\Omega_\c)}, \label{inter_shift} \\
\|\nabla \ZZh\|_{L^2(\Omega_\c)} \lesssim~& \gamma_{\rm tr}\|\nabla W \|_{L^2(\Omega_\c)},\label{h1_int_norm_est} \\
\|\zzhalpha\|_{L^2(\Omega_\c)} \lesssim~& \|r_\alpha\|_{L^2(\Omega_\c)} \label{l2_int_norm_est}.
\end{align}
\end{lemma}
%%%---------
\begin{proof}
We begin by letting $\omega_\xi := {\rm supp}(\bar{\zeta}(x-\xi))$ and $\mathcal{C} := \{\xi \in \mathcal{L} : \omega_\xi  \subset \Omega_\c \}$.  Then we observe that $Z_h$ and $\bar{Z}$ are constant on any patch $\omega_\xi$ with $\xi \notin \mathcal{C}$, \helen{and furthermore $Z_h = \bar{Z}$.  Intuitively, this should hold because if $\xi \notin \mathcal{C}$, then either $\xi$ is near the defect core where $\varphi = 0$ and hence $Z_h = 0$ and $\bar{Z} = 0$; or $\xi$ is near the exterior to the boundary of $\Omega$ where $Z_h$ is constant. For this to rigorously hold, we need to recall the buffer, $B_{4\rm buff}$, in the definition of $\Omega_\c$ which then makes proving the statement possible.  Moreover, $Z_h = \bar{Z}$ on any patch $\omega_\xi$ with $\xi \notin \mathcal{C}$ due to the normalization factor in the definition of $Z$.} For $f \in W^{1,2}_{\rm loc}(\mathbb{R}^d)$ we then have
   \begin{equation}\label{Z_result}
	\begin{split}
      &\int_{\Omega_\c} f (\bar{Z} - Z_h) dx = \sum_{\xi \in \mathcal{L}} \int_{\omega_\xi \cap \Omega_\c} f(x) \big( Z(\xi) - Z_h(x) \big) \bar{\zeta}(x-\xi) dx \\
			&= \dao{\sum_{\substack{\xi \in \mathcal{L}: \\ \omega_\xi \subset \Omega_\c}} \int_{\omega_\xi} f(x) \big( Z(\xi) - Z_h(x) \big) \bar{\zeta}(x-\xi) dx \quad \mbox{since $Z_h = Z$ is constant for $\xi \notin \mathcal{C}$}}\\
      &= \sum_{\xi \in \mathcal{C}} \int_{\omega_\xi} \bigg(f(x) - \dashint_{\omega_\xi} f \bigg) \big( Z(\xi) - Z_h(x) \big) \bar{\zeta}(x-\xi) dx \\
      &\leq \sum_{\xi \in \mathcal{C}} \bigg\| f - \dashint_{\omega_\xi}
                f  \bigg\|_{L^2(\omega_\xi)} \| Z(\xi) - Z_h \|_{L^2(\omega_\xi)} \\
      &\lesssim \sum_{\xi \in \mathcal{C}}\| \nabla f \|_{L^2(\omega_\xi)} \| \nabla Z_h \|_{L^2(\omega_\xi)} \\
		&\lesssim \| \nabla f \|_{L^2(\Omega_\c)} \| \nabla Z_h \|_{L^2(\Omega_\c)}.
   \end{split}
	\end{equation}
This proves~\eqref{weight_inter}.  Proving~\eqref{za_result} is analogous:
\begin{equation*}
\int_{\Omega_\c} f\cdot (\zzhalpha - \bar{z}_\alpha)\, dx \lesssim~ \| \nabla f\|_{L^2(\Omega_\c)} \cdot {\| \nabla \zzhalpha \|_{L^2(\Omega_\c)}} {\lesssim \| \nabla f\|_{L^2(\Omega_\c)} \cdot \| \zzhalpha \|_{L^2(\Omega_\c)} },
\end{equation*}
where in obtaining the final inequality we have used that for $T \in \mathcal{T}_\a$,
\[
\| \nabla z_h \|_{L^2(T)} \lesssim h_T \| z_h \|_{L^2(T)} \lesssim \| z_h \|_{L^2(T)}.
\]
For these choices, we also have the following norm estimates~\eqref{inter_diff} and~\eqref{inter_shift}:
\begin{align*}
\| \ZZh- \bar{Z}\|_{L^2(\Omega_\c)} \lesssim~& \|\nabla \ZZh\|_{L^2(\Omega_\c)}, \\
\| \zzhalpha - \bar{z}_\alpha\|_{L^2(\Omega_\c)} \lesssim~& \| \zzhalpha\|_{L^2(\Omega_\c)}.
\end{align*}
To obtain the first of these, we simply take {\helen{$ f= \bar{Z}-\ZZh$}} in~\eqref{weight_inter} yielding\helen{
\begin{equation*}
\begin{split}
&\| \ZZh- \bar{Z}\|_{L^2(\Omega_\c)}^2 \lesssim~ \| \nabla \ZZh- \nabla \bar{Z}\|_{L^2(\Omega_\c)}\cdot \| \nabla \ZZh\|_{L^2(\Omega_\c)} \lesssim~ \| \nabla \ZZh\|_{L^2(\Omega_\c)}^2 +\|\nabla \bar{Z}\|_{L^2(\Omega_\c)}^2 \\
&~\lesssim~ \| \nabla \ZZh\|_{L^2(\Omega_\c)}^2 +\|\nabla Z\|_{L^2(\Omega_\c)}^2 \lesssim~ \| \nabla \ZZh\|_{L^2(\Omega_\c)}^2 +\|\nabla \ZZh\|_{L^2(\Omega_\c)}^2,
\end{split}
\end{equation*}
where we have applied Young's inequality to deduce the estimate
\begin{align*}
\|\nabla Z\|_{L^2(\Omega_\c)}^2 = \|\nabla Z\|_{L^2(\mathbb{R}^d)}^2  = \Big\|\frac{(\barZeta*\nabla \ZZh)}{\int \barZeta dx}\Big\|^2_{L^2(\mathbb{R}^d)} \leq \|\nabla \ZZh\|_{L^2(\mathbb{R}^d)}^2 \|\barZeta\|_{L^1(\mathbb{R}^d)}^2 \lesssim~ \|\nabla \ZZh\|_{L^2(\Omega_\c)}^2.
\end{align*}
}
For the second of these, we simply have
\begin{equation*}
\begin{split}
\| \zzhalpha - \bar{z}_\alpha\|_{L^2(\Omega_\c)} \leq~& \| \zzhalpha\|_{L^2(\Omega_\c)} +  \|\bar{z}_\alpha\|_{L^2(\Omega_\c)} \lesssim~ \| \zzhalpha\|_{L^2(\Omega_\c)} +  \|z_\alpha\|_{L^2(\Omega_\c)} \\
\lesssim~& \| \zzhalpha\|_{L^2(\Omega_\c)} + \| \zzhalpha\|_{L^2(\Omega_\c)},
\end{split}
\end{equation*}
where we have again used Young's inequality for convolutions.  Next, upon recalling the definition
\[
\gamma_{\rm tr} = \begin{cases} &\sqrt{1 + \log(R_{\rm o}/R_{\a})}, \quad \mbox{if $d = 2$,} \\
                                 &1, \quad \mbox{if $d = 3$,} \end{cases}
\]
we have~\eqref{h1_int_norm_est} and~\eqref{l2_int_norm_est}:
\begin{equation*}
\begin{split}
\|\nabla \ZZh\|_{L^2(\Omega_\c)} \lesssim~& \gamma_{\rm tr}\|\nabla W \|_{L^2(\Omega_\c)}, \\
\|\zzhalpha\|_{L^2(\Omega_\c)} \lesssim~& \|r_\alpha\|_{L^2(\Omega_\c)}.
\end{split}
\end{equation*}
The first of these is a consequence of~\cite[Lemma 7]{blended2014}.  The second is a result of $0 \leq \varphi \leq 1$:
\begin{equation*}
   \|\zzhalpha\|_{L^2(\Omega_\c)} = \|I_h(\varphi r_\alpha) \|_{L^2(\Omega_\c)} \leq~ \|I_h(r_\alpha) \|_{L^2(\Omega_\c)} = \|r_\alpha\|_{L^2(\Omega_\c)}.
\end{equation*}

%
%first defining $\zzhalpha$ as
%\begin{equation}\label{Def_zzhalpha}
%\zzhalpha:=I_h(\varphi r_{\alpha})\quad \text{for}\quad \alpha=1,\dots,s-1,
%\end{equation}
%and then setting
%\begin{equation}\label{Def_zalpha}
%z_{\alpha} = \check{\zzhalpha}.
%\end{equation}
\end{proof}

We are now ready to prove Theorem~\ref{consistency_thm}.
%%%----------------------

\begin{proof}[Proof of Theorem~\ref{consistency_thm}]
\dao{Since $\tilde{I}u$ interpolates $u$ at $\xi \in \mathcal{L}$, we may replace discrete $U^{\infty}$ with continuous $\tilde{I}U=\tilde{U}^{\infty}$ in~\eqref{residual_est_new} which leaves us with estimating
\begin{equation}\label{residual_est}
\begin{split}
\la & \mathcal{F}^{\bqcf}(U_h,\bm{p}_h), (W,\bm{r})\ra = \la \mathcal{F}^{\bqcf}(U_h,\bm{p}_h), (W,\bm{r})\ra
-\la \delta \mathcal{E}^{\a}(U^\infty, \bm{p}^\infty), (U,\bm{q})\ra\\
% &\quad= \la \delta \mathcal{E}^{\c} (U_h,\bm{p}_h), \big(I_h(\varphi W), I_h(\varphi \bm{r})\big)\ra
% -\la \delta\mathcal{E}^{\a}(U^\infty,\bm{p}^\infty), (Z^*,{\bm{z}}^*)\ra\\
&\quad= \la \delta \mathcal{E}^{\c} (U_h,\bm{p}_h), \big(I_h(\varphi W), I_h(\varphi \bm{r})\big)\ra
-\la \delta\mathcal{E}^{\a}(\tilde{U}^\infty,\tilde{\bm{p}}^\infty), (Z^*,\bm{z}^*)\ra.
\end{split}
\end{equation}
}
Recalling that $Z_h := I_h(\varphi W)$, $\bm{z}_h := I_h(\varphi \bm{r})$, and the atomistic and continuum stress representations of~\eqref{eq:defn_Sa} and~\eqref{cont_stress_tensor}, we split this into three terms {using simple algebraic manipulations as}
\begin{align}
&\la \delta \mathcal{E}^{\c} (U_h,\bm{p}_h), \big(I_h(\varphi W), I_h(\varphi \bm{r})\big)\ra
-\la \delta\mathcal{E}^{\a}(\tilde{U}^\infty,\tilde{\bm{p}}^\infty), (Z^*,\bm{z}^*)\ra\ra\nonumber\\
&\leq \bigg| \int_{\mathbb{R}^d} \sum_\beta {\big[[{\rm S}^{\rm c}_{\rm d}(U_h,\bm{p}_h)]_\beta : \nabla Z_h - [{\rm S}^{\rm a}_{\rm d}(\tilde{U}^\infty,\tilde{\bm{p}}^\infty)]_\beta : \nabla \bar{Z}\big]} \bigg|
  + \bigg| \int_{\mathbb{R}^d} {\sum_{\alpha,\beta}} [{\rm S}^{\rm c}_{\rm s} (U_h,\bm{p}_h)]_{\alpha\beta} \cdot( \zzhalpha- \zzhbeta) \nonumber \\
	&\quad - {\sum_{\alpha,\beta}} [{\rm S}^{\rm a}_{\rm s}(\tilde{U}^\infty,\tilde{\bm{p}}^\infty)]_{\alpha\beta} \cdot (\bar{z}_\alpha - \bar{z}_\beta)  \bigg| +\bigg| \int_{\mathbb{R}^d} {\sum_\beta [{\rm S}^{\rm a}_{\rm d}(\tilde{U}^\infty,\tilde{\bm{p}}^\infty)]_\beta : \nabla \bar{z}_\beta }\bigg| \nonumber \\%\label{BQCF_consist_eq2}\\
&~=: T^1_{\rm d} + T_{\rm s}+ T^2_{\rm d}.\nonumber
\end{align}
Next, we analyze these terms separately.

{\it Term $T^1_{\rm d}$: } The $T^1_{\rm d}$ term is identical to the simple lattice case after accounting for the additional approximation of the shifts. Following the ideas
from the simple lattice case~\cite{blended2014},
\helen{we break down $T_{\rm d}^{1}$ into three additional terms as in Section 6.4.1 of~\cite{blended2014} (the difference being we do not consider a quadrature error),
 and apply the
estimates of stress differences from Corollary~\ref{globel_stress} and the approximating estimates from Lemma~\ref{approx_lem} and \eqref{h1_int_norm_est}. This produces
\begin{equation*}\label{T1d_est}
\begin{split}
T^1_{\rm d} &\lesssim~ \bigg| \int_{\mathbb{R}^d} \sum_\beta \big\{[{\rm S}^{\rm c}_{\rm d}(U_h,\bm{p}_h)]_\beta -[{\rm S}^{\rm c}_{\rm d}(\tilde{U}^\infty,\tilde{\bm{p}}^\infty)]_\beta \big\} : \nabla Z_h\, dx\bigg| \\
&\qquad +~ \bigg| \int_{\mathbb{R}^d} [{\rm S}^{\rm c}_{\rm d}(\tilde{U}^\infty,\tilde{\bm{p}}^\infty)]_\beta \big\} : (\nabla Z_h-\nabla \bar{Z})\, dx\bigg| \\
& \qquad +~ \bigg| \int_{\mathbb{R}^d} \sum_\beta \big\{[{\rm S}^{\rm c}_{\rm d}(\tilde{U}^\infty,\tilde{\bm{p}}^\infty)]_\beta -[{\rm S}^{\rm a}_{\rm d}(\tilde{U}^\infty,\tilde{\bm{p}}^\infty)]_\beta \big\} : \nabla \bar{Z}\, dx\bigg| \\
&\lesssim~ \gamma_{\rm tr}\Big(\|h \nabla^2 \tilde{U}^\infty\|_{L^2(\Omega_\c)}  + \|h\nabla \tilde{\bm{p}}^\infty\|_{L^2(\Omega_\c)} \\
& \qquad  \qquad+ \|\nabla \tilde{U}^\infty\|_{L^2(\Omega_{\rm ext})} + \|\tilde{\bm{p}}\|_{L^2(\Omega_{\rm ext})} \Big)\cdot \|\nabla W\|_{L^2(\mathbb{R}^d)}.
\end{split}
\end{equation*}
}
{\it Term $T_{\rm s}$: } For the shift term $T_{\rm s}$, we have
\begin{align*}
   T_{\rm s} &\lesssim  \bigg| \int_{\mathbb{R}^d} {\sum_{\alpha,\beta}} [{\rm S}^{\rm c}_{\rm s} (U_h,\bm{p}_h)]_{\alpha\beta} \cdot( \zzhalpha- \zzhbeta) - {\sum_{\alpha,\beta}} [{\rm S}^{\rm c}_{\rm s}(\tilde{U}^\infty,\tilde{\bm{p}}^\infty)]_{\alpha\beta} \cdot( \zzhalpha- \zzhbeta)  \bigg| \\
	&\quad + \bigg| \int_{\mathbb{R}^d} {\sum_{\alpha,\beta}} [{\rm S}^{\rm c}_{\rm s} (\tilde{U}^\infty,\tilde{\bm{p}}^\infty)]_{\alpha\beta} \cdot( \zzhalpha - \zzhbeta -(\bar{z}_\alpha - \bar{z}_\beta)) \bigg| \\
&\quad + \bigg| \int_{\mathbb{R}^d} {\sum_{\alpha,\beta}} [{\rm S}^{\rm c}_{\rm s} (\tilde{U}^\infty,\tilde{\bm{p}}^\infty)]_{\alpha\beta} \cdot(\bar{z}_\alpha - \bar{z}_\beta)  - {\sum_{\alpha,\beta}} [{\rm S}^{\rm a}_{\rm s}(\tilde{U}^\infty,\tilde{\bm{p}}^\infty)]_{\alpha\beta} \cdot (\bar{z}_\alpha - \bar{z}_\beta)  \bigg| \\
   &=: T_{\rm s,1} + T_{\rm s,2} + T_{\rm s,3}.
\end{align*}
%where $\tilde{S}^{\rm c}_{\rm s} := S^{\rm c}_{\rm s}(\tilde{U}, \tilde{\bm{p}})$.

Using Lipschitz continuity of $\delta V$ {(in the definition of ${\rm S}^{\rm c}_{\rm s}$)} and the fact that $\bm{z}_h$ is supported in $\Omega_{\rm c}$ \dao{followed by an application of the test function estimate~\eqref{l2_int_norm_est}}, we obtain
\begin{equation*}
\begin{split}
   \label{Ts1}
   |T_{\rm s,1}| \lesssim~&
   \Big(\| \nabla \Pi_h U - \nabla \tilde{U} \|_{L^2(\Omega_{\rm c})}
      + \| \Pi_h \bm{p} - \tilde{\bm{p}} \|_{L^2(\Omega_{\rm c})}
      \Big) \| \bm{z}_h \|_{L^2(\mathbb{R}^d)} \\
			\lesssim~& \Big(\| \nabla \Pi_h U - \nabla \tilde{U} \|_{L^2(\Omega_{\rm c})}
      + \| \Pi_h \bm{p} - \tilde{\bm{p}} \|_{L^2(\Omega_{\rm c})}
      \Big) \| \bm{r}\|_{L^2(\mathbb{R}^d)}
\end{split}
\end{equation*}
Using the stress estimate, Corollary~\ref{globel_stress}, \helen{followed by the application of the test function norm estimates~\eqref{inter_shift} and~\eqref{l2_int_norm_est},} we get
\begin{align*}
   \label{Ts3}
   |T_{\rm s,3}| &\lesssim
      \Big(\| \nabla^2 \tilde{U} \|_{L^2(\Omega_{\rm c})}
      + \| \nabla \tilde{\bm{p}} \|_{L^2(\Omega_{\rm c})} \Big)
      \| \bar{\bm{z}} \|_{L^2(\mathbb{R}^d)}\\
      &\lesssim \Big(\| \nabla^2 \tilde{U} \|_{L^2(\Omega_{\rm c})}
      + \| \nabla \tilde{\bm{p}} \|_{L^2(\Omega_{\rm c})} \Big)
      \| \bm{r} \|_{L^2(\mathbb{R}^d)}.
\end{align*}
Finally, to treat $\bm{z}_h - \bar{\bm{z}}$ inside $T_{\rm s,2}$, we use~\eqref{za_result} of Lemma~\ref{interpolation_lemma} with
$f= [S^\c_s(\tilde{U}^\infty,\tilde{\bm{p}}^\infty )]_{\alpha\beta}$  {\helen{followed by an application of \eqref{l2_int_norm_est}, the chain rule, and \eqref{cont_stress_tensor}:}}
{\helen{
\begin{equation*}\label{Ts2}
\begin{split}
|T_{\rm s,2}| &\lesssim~ \big\|\nabla \Big( {\rm S}^{\rm c}_{ \rm s }\big(\tilde{U}^\infty,\tilde{\bm{p}}^\infty\big)\Big)\big\|_{L^2(\Omega_{\rm c})} \|\bm{z}_h \|_{L^2(\mathbb{R}^d)}\\
&\lesssim~ \|\nabla {\rm S}^{\rm c}_{ \rm s }\big(\tilde{U}^\infty,\tilde{\bm{p}}^\infty\big)
\cdot \nabla \big(\nabla\tilde{U}^\infty+\tilde{\bm{p}}^\infty\big)\|_{L^2(\Omega_{\rm c})} \|\bm{r} \|_{L^2(\mathbb{R}^d)} .
\end{split}
\end{equation*}
}}
Combining our estimates for $T_{\rm s,1}, T_{\rm s,2}$, and $T_{\rm s,3}$ and \dao{appealing to Lemma~\ref{approx_lem} to estimate $T_{\rm s,1}$ along with the crude estimate $h \gtrsim 1$} gives
\begin{align*}\label{Ts_est}
|T_{\rm s}|
\lesssim~ \Big(\|h \nabla^2 \tilde{U}^\infty\|_{L^2(\Omega_\c)} &+\|h\nabla \tilde{\bm{p}}^\infty\|_{L^2(\Omega_\c)} \\
&+ \|\nabla \tilde{U}^\infty\|_{L^2(\Omega_{\rm ext})} + \|\tilde{\bm{p}}\|_{L^2(\Omega_{\rm ext})} \Big) \|\bm{r} \|_{L^2(\mathbb{R}^d)}.\nonumber
\end{align*}

{\it Term $T^2_{\rm d}$: }
Finally, to estimate $T^2_{\rm d}$ we split it into
\begin{equation*}\label{T2d_decomp}
\begin{split}
|T^2_{\rm d}|=~& \bigg|\int_{\mathbb{R}^d}  {\sum_\beta[{\rm S}^{\rm a}_{\rm d}(\tilde{U}^\infty,\tilde{\bm{p}}^\infty)]_\beta : \nabla \bar{z}_\beta} \bigg| \\
\lesssim~&
 \bigg|\int_{\mathbb{R}^d} {\sum_\beta \big({\rm S}^{\rm c}_{\rm d}(\tilde{U}^\infty,\tilde{\bm{p}}^\infty)]_\beta- [{\rm S}^{\rm a}_{\rm d}(\tilde{U}^\infty,\tilde{\bm{p}}^\infty)]_\beta \big) : \nabla \bar{z}_\beta} \bigg| \\
 &\quad +  \bigg|\int_{\mathbb{R}^d} {\sum_\beta[{\rm S}^{\rm c}_{\rm d}(\tilde{U}^\infty,\tilde{\bm{p}}^\infty)]_\beta : \nabla \bar{z}_\beta} \bigg| \\
=:~& T^2_{\rm d, 1} + T^2_{\rm d,2}.
\end{split}
\end{equation*}
%{\helen{Since $V_{\xi,\triple}(\widetilde{Y},\widetilde{\bm{p}})$ is a function of lattice site $\xi$ only, thus
%$\mathsf{div}_{\rho} $ only works on $\omeRho(\xi-x)$, so we have decay issues on this term \eqref{T3_eq1}.}}
To estimate $T^2_{\rm d, 1}$, we note that it is similar to $T^1_{\rm d}$ in that $\nabla_{\rho}\overline{z_{\beta}}$ is zero off $\Omega_\c$
 \dao{(which is due to the support of the blending function and the definition of $\Omega_{\rm c}$; see the proof of Lemma~\ref{interpolation_lemma} for further explanation)}
 so we utilize the stress estimate in Corollary~\ref{globel_stress} along with the bound
\dao{
\[
\|\nabla \bar{z}_\beta\| \lesssim~ \|\bar{z}_\beta\| \lesssim~ \|r_\beta\|
\]
which follows from
\begin{align*}
\|\bar{z}_\beta\| \lesssim~& \| \zzhbeta\|_{L^2(\Omega_\c)} \qquad \mbox{by~\eqref{inter_shift}} \\
\lesssim~& \|r_\beta\|_{L^2(\Omega_\c)} \qquad \, \mbox{  by~\eqref{l2_int_norm_est}.}
\end{align*}
This produces
\begin{align*}
 T^2_{\rm d, 1}\lesssim~& \left(\|\nabla^2 \tilde{U}^\infty\|_{L^2(\Omega_\c)} + \| \nabla \tilde{\bm{p}}^\infty\|_{L^2(\Omega_\c)}\right) \| \nabla \bar{\bm{z}} \|_{L^2(\mathbb{R}^d)} \\
\lesssim~& \left(\|\nabla^2 \tilde{U}^\infty\|_{L^2(\Omega_\c)} + \| \nabla \tilde{\bm{p}}^\infty\|_{L^2(\Omega_\c)}\right) \|  \bm{r} \|_{L^2(\mathbb{R}^d)}.
\end{align*}}
Meanwhile, we may integrate $T^2_{\rm d,2}$ by parts and use the aforementioned fact that $\|\bar{z}_\beta\| \lesssim~ \|r_\beta\|$ to obtain
\begin{align*}
T^2_{\rm d,2} \lesssim~
{\sum_\beta \left\|{\rm div} \left(  [{\rm S}^{\rm c}_{\rm d}(\tilde{U}^\infty,\tilde{\bm{p}}^\infty )]_\beta  \right)\right\|_{L^2(\Omega^c)} \|\bm{r} \|_{L^2(\mathbb{R}^d)}}.
\end{align*}
{\helen{Applying the chain rule to ${\rm div} \left(  [{\rm S}^{\rm c}_{\rm d}(\tilde{U}^\infty,\tilde{\bm{p}}^\infty )]_\beta  \right)$
(just like for $T_{\rm s, 2}$), we get }}
\begin{align*}
|T^2_{\rm d}|\lesssim T^2_{\rm d,1} +T^2_{\rm d,2}&\lesssim \left(\|\nabla^2 \tilde{U}^\infty\|_{L^2(\Omega_\c)} +\| \nabla \tilde{\bm{p}}^\infty\|_{L^2(\Omega_\c)}\right) \| \bm{r} \|_{L^2(\mathbb{R}^d)}\\
&\lesssim \left(\|h\nabla^2 \tilde{U}^\infty\|_{L^2(\Omega_\c)} +\| h\nabla \tilde{\bm{p}}^\infty\|_{L^2(\Omega_\c)}\right) \| \bm{r} \|_{L^2(\mathbb{R}^d)}.
\end{align*}

%  defined in \eqref{BQCF_consist_eq2} is estimated in Lemma~\ref{lem:est_T2d}  by
% \begin{equation}\label{T2d_est}
% |T^2_{\rm d}|\lesssim~ (\|\nabla^2 \tilde{U}^\infty\|_{L^2(\Omega_\c)} + \| \nabla \tilde{\bm{p}}^\infty\|_{L^2(\Omega_\c)})\cdot \| \bm{r} \|_{L^2(\mathbb{R}^d)}.
% \end{equation}
Combining our estimates for $T^1_{\rm d},T_{\rm s}$,
and $T^2_{\rm d}$ yields the stated result.
\end{proof}

\subsection{Stability}\label{stab}
The second key ingredient in our proof of Theorem~\ref{main_thm} is the stability estimate~\eqref{stab_est}; \dao{this in turn implies a bound on the inverse of {\helen{the linearised BQCF operator}}, which we will use in a quantitative version of the inverse function theorem to establish existence of the solution to our BQCF equations.}  Conceptually, the proof of stability is similar to that of the simple lattice case presented in~\cite{blended2014}.

% The precise two results we obtain in Sections~\ref{stab_def_free} and~\ref{stab_defect_present} are as follows.

\begin{theorem}[Stability of BQCF]\label{stab_theorem_full}
 Suppose that Assumptions~\ref{assumption1},~\ref{assumptionSite}, and~\ref{assumption2} hold. There exists a critical size, $R_{\rm core}^*$, of the atomistic region such that, for all shape regular meshes and blending functions meeting the requirements of Section~\ref{bqcf} and $R_{\rm core} \geq R_{\rm core}^*$,
\[
\frac{\gamma_{\a}}{2}\|(W,\bm{r})\|_{\rm ml}^2 \leq~ \<\delta \mathcal{F}^{\bqcf}(\Pi_h(U^\infty, \bm{p}^\infty)) (W, \bm{r}),(W, \bm{r})\>, \quad \forall \, (W, \bm{r}) \in \bm{\mathcal{U}}_{h,0} \times \bm{\mathcal{P}}_{h,0}.
\]
\end{theorem}
As an intermediate step we also prove stability of the reference state.

\begin{theorem}[Stability of BQCF at Reference State]\label{stab_theorem}
Suppose that Assumptions~\ref{assumption1},~\ref{assumptionSite}, and~\ref{assumption2} hold.  There exists a critical size $R_{\rm core}^*$ of the atomistic region such that, for all meshes having shape regularity constant bounded below by $C_{\mathcal{T}_h}$ and blending functions meeting the requirements of Section~\ref{bqcf} and $R_{\rm core} \geq R_{\rm core}^*$,
\[
   \frac{3}{4}\gamma_\a\|(W,\bm{r})\|_{\rm ml}^2 \leq~ \<\delta \mathcal{F}^{\bqcf}_{\rm hom}(0) (W, \bm{r}),(W, \bm{r})\>, \quad \forall \, (W, \bm{r}) \in \bm{\mathcal{U}}_{h,0} \times \bm{\mathcal{P}}_{h,0}.
\]
\end{theorem}

Before we present the proofs of these results in Sections~\ref{stab_def_free} and~\ref{stab_defect_present} we apply them to complete the proof of our main result, Theorem~\ref{main_thm}.

\subsection{Proof of the main result}
\label{sec:proof_main_result}
% With these two theorems, we have all of the tools necessary to prove our main result, Theorem~\ref{main_thm}.

\begin{proof}[Proof of Theorem~\ref{main_thm}]
We apply the inverse function theorem, Theorem~\ref{inverseFunctionTheorem},  to the BQCF variational operator $\mathcal{F}^{\bqcf}$ at the linearization point $\Pi_h(U^\infty, \bm{p}^\infty)$.  The parameters $\eta$ and $\sigma$ defined in Theorem~\ref{inverseFunctionTheorem} are
\begin{equation*}
\begin{split}
&\eta :=~ \gamma_{\rm tr}\big(\|h\nabla^2 \tilde{U}\|_{L^2(\Omega_\c)} + \|h\nabla \tilde{\bm{p}}\|_{L^2(\Omega_\c)}  + \|\nabla \tilde{U}\|_{L^2(\Omega_{\rm ext})} \big. \\
&\big. \qquad\qquad +  \| \tilde{\bm{p}}\|_{L^2(\Omega_{\rm ext})}\big)\cdot \|(W,\bm{r})\|_{\rm ml}, \quad \forall (W,\bm{r}) \, \in \bm{\mathcal{U}}_{h,0} \times \bm{\mathcal{P}}_{h,0},
\end{split}
\end{equation*}
which is the consistency error from Theorem~\eqref{consistency_thm}, and
\[
\sigma^{-1} := \frac{\gamma_\a}{2},
\]
which is the coercivity constant from Theorem~\eqref{stab_theorem_full} that exists so long as $R_{\rm core} \geq R_{\rm core}^*$, where $R_{\rm core}^*$ is furnished by Theorem~\eqref{stab_theorem_full}.  (The requirement $R_{\rm core} \geq R_{\rm core}^*$ means the domain decomposition procedure meets the requirements stated in Theorem~\ref{main_thm}.)  The Lipschitz estimate on $\delta \mathcal{F}^{\bqcf}$ is a direct result of the assumptions made on the site potential in Assumption~\ref{assumption1}.  Applying the inverse function theorem with these parameters gives existence of $(U^{\bqcf}, \bm{p}^{\bqcf})$ and the stated error estimate,~\eqref{main_estimate}, follows from the inverse function theorem and the approximation lemma, Lemma~\ref{approx_lem}.
\end{proof}

The remainder of the paper is devoted to proving Theorems~\ref{stab_theorem_full} and~\ref{stab_theorem}.

\subsection{Stability of BQCF at defect-free reference state}\label{stab_def_free}

We first prove Theorem \ref{stab_theorem}, that is, coercivity of the
homogeneous BQCF operator,
\begin{align*}
\<\delta \mathcal{F}^{\bqcf}_{\rm hom}(0)(W,\bm{r}),(W,\bm{r})\> =~& \<\delta^2\mathcal{E}^\a_{\rm hom}(0)((1-\varphi)W,(1-\varphi)\bm{r}), (W, \bm{r})\> \\
&\qquad +  \<\delta^2\mathcal{E}^\c(0)(I_h(\varphi W),I_h(\varphi \bm{r})), (W, \bm{r})\>.
\end{align*}
That is, we want to show that there exists $\gamma_{\bqcf}$ independent of the approximation parameters such that, for sufficiently large $R_{\rm core}$,
\begin{equation} \label{eq:gamma_bqcf}
   0 < \gamma_{\rm bqcf}\|(W,\bm{r})\|_{\rm ml}^2 \leq \<\delta \mathcal{F}^{\bqcf}_{\rm hom}(0) (W, \bm{r}),(W, \bm{r})\>.
\end{equation}
The proof via contradiction is involved; hence we first outline and motivate the procedure and then give a number of technical results required to prove the theorem at the end of this section. \dao{The main idea is that the linearized BQCF operator consists of an atomistic second variation and a continuum second variation.  Each of these can be individually shown to be coercive so intuitively, we would expect this linearized operator to be coercive for any test pair $(W,\bm{r})$ with support concentrated near the origin (in which case the blending function is zero) and for $(W,\bm{r})$ with support concentrated far from the origin (in which case the blending function would be one).  Thus, we expect the only possible instabilities to occur with test pairs having some support over the blending region.  Since there is no defect in the homogeneous case, any such instability should also occur for any geometric setup, i.e., we can consider the BQCF method for a sequence of growing atomistic domain sizes and should still have an unstable mode.  Thus we shall consider such a sequence and then rescale this sequence so that the atomistic region in each case is contained in a ball of fixed radius about the origin and such that these unstable modes converge (in a sense to be made precise momentarily) to some continuum limit.  We then consider evaluating the suitably rescaled linearized BQCF operator on this sequence and show using the aforementioned stability of the atomistic and continuum components \textit{and} convergence of the test pairs $(W,\bm{r})$ that this leads to a contradiction.  One of the main technical difficulties encountered here is that due to blending by forces, the individual atomistic/continuum components and hence the linearized BQCF operator is not a symmetric bilinear form.  Thus we must take some care in converting the force-based formulation to a form suitable to using the existing coercivity estimates on the atomistic and continuum Hessians.}

The negation of \eqref{eq:gamma_bqcf} is: ``for all atomistic region sizes $R_{\rm a}$, there exists a blending function $\varphi$ and a mesh $\mathcal{T}_h$ compatible with the assumptions of Section~\ref{sec:approx_params} (and in particular Assumption~\ref{assumption-shapereg}), {\helen{as well as a test pair, $(W, \bm{r})$ with norm scaled to one, such that}}
\begin{equation} \label{eq:negation}
\frac{3}{4} \gamma_{\a}> \<\delta \mathcal{F}_{\rm hom}^{\bqcf}(0) (W, \bm{r}),(W, \bm{r})\>.\mbox{''}
\end{equation}
Thus, for contradiction, suppose that there exists a sequence $R_{\a,n} \to \infty$ with associated meshes $\mathcal{T}_{h,n}$, blending functions $\varphi_n$, finite element spaces $\bm{\mathcal{U}}_{h,0}^n \times \bm{\mathcal{P}}_{h,0}^n$, and test pairs $(W_n, \bm{r}_n) \in \bm{\mathcal{U}}_{h,0}^n \times \bm{\mathcal{P}}_{h,0}^n$ with norm one such that
\begin{equation}\label{contra_seq}
\begin{split}
&\frac{3}{4}\gamma_\a  >
\sum_{\xi \in \mathcal{L}} \sum_{(\rho\alpha\beta)}\sum_{(\tau\gamma\delta)} V_{,(\rho\alpha\beta)(\tau\gamma\delta)}:D_{(\rho\alpha\beta)}((1-{\varphi}_n) {W}_n,(1-{\varphi}_n){\bm{r}}_n):D_{(\rho\alpha\beta)}({W}_n,{\bm{r}}_n) \\
& \qquad \qquad + \int_{\mathbb{R}^d} \sum_{(\rho\alpha\beta)}\sum_{(\tau\gamma\delta)} V_{,(\rho\alpha\beta)(\tau\gamma\delta)}:\nabla_{(\rho\alpha\beta)}( {I_{h}}  ( {\varphi}_n( {W}_n,  {\bm{r}}_n))) :\nabla_{(\rho\alpha\beta)}( {W}_n,  {\bm{r}}_n) \, dx,
\end{split}
\end{equation}
where we have omitted the argument, $0$, in $V_{,(\rho\alpha\beta)(\tau\gamma\delta)}(0)$ and where $I_{h}$ is now the piecewise linear interpolant on $\mathcal{T}_{h,n}$.

We now rescale space in \eqref{eq:negation} and derive a continuum scaling limit,
from which we will be able to obtain a contradiction.  To that end, let $\epsilon_n = 1/R_{\a,n}$, and define the set of scaled parameters
\begin{equation} \label{eq:stab:rescaling}
\begin{split}
\hat{\xi}_n =~& \epsilon_n \xi \\
\hat{x}_n =~& \epsilon_n x \\
\hat{r}_n(\dao{\hat{x}_n}) =~& \epsilon_n^{-d/2} r_n(\dao{\hat{x}_n}/\epsilon_n) \\
\hat{W}_n(\dao{\hat{x}_n}) =~& \epsilon_n^{1-d/2} W_n(\dao{\hat{x}_n}/\epsilon_n) \\
\hat{\varphi}_n(\dao{\hat{x}_n}) =~& \varphi_n(\dao{\hat{x}_n}/\epsilon_n).
\end{split}
\end{equation}
In terms of these rescaled quantities, {\helen{we define $\hat{\nabla}:=\epsilon_n^{-1}\nabla_x=\nabla_{\hat{x}_n}$} (when the subscript $n$ is clear we use $\hat{\nabla}$)} and then have
{\helen{
\begin{align*}
\| \nabla_{\hat{x}_n} \hat{W}_n \|^2_{L^2(\mathbb{R}^d)} = \| \nabla_{x} W_n \|^2_{L^2(\mathbb{R}^d)},& \,\, \|\epsilon_n\nabla_{\hat{x}_n} \hat{r}_n\|^2_{L^2(\mathbb{R}^d)} = \|\nabla_x r_n\|^2_{L^2(\mathbb{R}^d)}, \\
 \| \hat{r}_n^\alpha \|^2_{L^2(\mathbb{R}^d)} =~& \| r_n^\alpha \|^2_{L^2(\mathbb{R}^d)},
\end{align*}
}}
and the rescaled BQCF operator is
{\helen{
\begin{equation}\label{bqcfHessian}
\begin{split}
&
 \<\delta \mathcal{F}_{{\rm hom},n}^{\bqcf}(0) (\hat{W}_n, \hat{\bm{r}}_n),(\hat{W}_n, \hat{\bm{r}}_n)\>  := \\
% &=~ \<\delta^2 \mathcal{E}^\a_{\rm hom}(0), (1-\varphi)(W_n,\bm{r}_n), (W_n, \bm{r}_n)\> + \<\delta^2 \mathcal{E}^\c(0), I_h\varphi(W_n, \bm{r}_n), (W_n, \bm{r}_n)\> \\
&\qquad \epsilon^d_n \sum_{\hat{\xi} \in \epsilon_n\mathcal{L}} \mathbb{C}:D_{n}((1-\hat{\varphi}_n)(\hat{W}_n,\hat{\bm{r}}_n)):D_n(\hat{W}_n,\hat{\bm{r}}_n)(\hat{\xi}) \\
&\qquad + \int_{\mathbb{R}^d} \mathbb{C}:\hat{\nabla}( I_{h,n}(\hat{\varphi}_n(\hat{W}_n, \hat{\bm{r}}_n))) :\hat{\nabla}(\hat{W}_n, \hat{\bm{r}}_n) \, d\hat{x}_n,
\end{split}
\end{equation}
}}
where $I_{h,n}$ is the piecewise linear interpolant on $\epsilon_n\mathcal{T}_{h,n}$ and
{\helen{
\begin{align*}
D_{n}( \hat{W} , \hat{\bm{r}} ) :=~& \big(D_{\triple,n}( \hat{W} , \hat{\bm{r}} )\big)_{\triple \in \mathcal{R}}, \\
D_{\triple,n}( \hat{W} , \hat{\bm{r}} ) (\hat{\xi}):=~& \frac{\hat{W}(\hat{\xi} + \epsilon_n \rho) + \epsilon_n \hat{r}_{n}^{\beta}(\hat{\xi} + \epsilon_n \rho) - \hat{W}(\hat{\xi}) - \epsilon_n \hat{r}_{n}^{\alpha}(\hat{\xi})}{\epsilon_n}.
\end{align*}
}}
The rescaling of the shifts $\hat{r}_n^\alpha$ is one order lower than the rescaling of displacements,  which is due to the fact that shifts are already discrete gradients.

We also define an interpolant onto the scaled lattice $\epsilon_n\mathcal{L}$ by $I_n$, a projection operator from the scaled lattice to finite element spaces $\bm{\mathcal{U}}^{n}_{h,0} \times \bm{\mathcal{P}}^n_{h,0}$ on $\mathcal{T}_{h,n}$ by $\Pi_{h,n} := S_{h,n} T_{r_{{\rm i},n}}$, and the scaled finite element basis function
\[
\bar{\zeta}_n(x) := \epsilon_n^{-d}\bar{\zeta}(x/\epsilon_n).
\]

Since $\dao{\hat{\nabla}} \hat{W}_n$ is bounded in $L^2$ and since each $\hat{r}_n^\alpha$ is also bounded (both having norm less than one), we may extract weakly convergent subsequences.  Furthermore, $\epsilon_n \dao{\hat{\nabla}} \hat{r}_n^\alpha$ is also bounded in $L^2$ so we may take it to be weakly convergent as well.  By replacing the original sequences with these weakly convergent subsequences (for notational convenience), we have $\dao{\hat{\nabla}} \hat{W}_n \weakto \dao{\hat{\nabla}} \hat{W}_0$, $\hat{r}_n^\alpha \weakto \hat{r}_0^\alpha$,
and $\epsilon_n \dao{\hat{\nabla}} \hat{r}_n^\alpha \weakto \hat{R}^\alpha_0$ in $L^2(\mathbb{R}^d)$ for some functions $\hat{W}_0, \hat{r}_0^\alpha$, and $\hat{R}^\alpha_0$ for each $\alpha$.  However, since $\hat{r}_n^\alpha$ is bounded in $L^2$ and $\epsilon_n\hat{r}_n^\alpha \to 0$ in $L^2$, $\hat{R}_0^\alpha = 0$.

Next, we choose explicit equivalence representatives for $\hat{W}_n$; namely, we choose $\hat{W}_n$ such that $\int_{B_1(0)} \hat{W}_n = 0$.   For this choice, we have $\|\hat{W}_n\|_{L^2(B_1(0))} \lesssim \|\dao{\hat{\nabla}} \hat{W}_n\|_{L^2(B_1(0))}$, and as $H^1$ is compactly embedded in $L^2$, there exists a strongly convergent subsequence, which we again denote by $\hat{W}_n$, such that $\hat{W}_n \to \hat{W}_0$ strongly in $L^2(B_1(0))$.

We also note here that $\hat{W}_n \weakto \hat{W}_0$ in the space
\[
\dot{\bm{H}}^{1} {\helen{(\mathbb{R}^d,\mathbb{R}^n)}} := \left\{f \in H^1_{\rm loc}(\mathbb{R}^d,\mathbb{R}^n)/\mathbb{R}^n : \|\nabla f\|_{L^2(\mathbb{R}^d)} < \infty \right\},
\]
and so $\hat{W}_0 \in \dot{\bm{H}}^1(\mathbb{R}^d, \mathbb{R}^n)$ as well~\cite{suli2012}.

The purpose of these subsequences is to use the pairs $(\hat{W}_n, \hat{\bm{r}}_n)$ to test with $\delta\mathcal{F}_{{\rm hom},n}^{\bqcf}(0)$. However, as these test pairs only consist of weakly convergent sequences and since the inner product of two weakly convergent sequences is not necessarily convergent, we further split $\hat{W}_n$ and $\hat{\bm{r}}_n$ into the sum of a strongly convergent sequence and a sequence weakly convergent to zero.

This splitting is accomplished by setting
\begin{equation}\label{splitting}
\hat{X}_n := \Pi_{h,n}(\eta_{j_n}  * \hat{W}_0), \qquad
\hat{s}_n^\alpha := \Pi_{h,n}(\eta_{j_n}  * \hat{r}_0^\alpha),
\end{equation}
where $\eta$ is a standard mollifier, $\eta_j(x) = j^{-d} \eta(x/j)$, and $j_n \to 0$ sufficiently slowly to ensure that  the sequences $\hat{X}_n$ and $\hat{s}_n^\alpha$ are strongly convergent to, respectively, $\hat{W}_0$ and $\hat{r}_0^\alpha$. We will impose several further properties on the sequence $j_n$ in Lemma~\ref{seq_lemma} below, 
but for the remainder of the present section, we make the following conventions for notational convenience.

\begin{remark}\label{remark_drop_hats}
{\helen{
To simplify and lessen the notations hereafter, we drop the hat notation on the sequences $X_n, Z_n, \bm{s}_n, \bm{t}_n$ as well as on their derivatives, and so forth.
}}
\end{remark}

Further, we define
\[
   \psi_n := 1-\varphi_n, \quad \mbox{and}
   \quad V_{,\triple\tripleTau} := V_{,\triple\tripleTau}\big(0\big),
\]
and use the notation
\begin{align*}\label{stab_notation}
V_{,\triple\tripleTau}\big( \cdot \big):v:w :=~& w^{\transpose}\big[V_{,\triple\tripleTau}\big( \cdot \big)\big]v \quad \forall v,w \in \mathbb{R}^n, \nonumber \\
\mathbb{C} : D(W,\bm{q}): D(Z,\bm{r}) :=~& \sum_{\triple \in \mathcal{R}} \sum_{\tripleTau \in \mathcal{R}} V_{,\triple\tripleTau}:D_{\triple}(W,\bm{q}):D_{\tripleTau}(Z,\bm{r}), \\
\mathbb{C} : \nabla (W,\bm{q}): \nabla (Z,\bm{r}) :=~& \sum_{\triple \in \mathcal{R}} \sum_{\tripleTau \in \mathcal{R}} V_{,\triple\tripleTau}:(\nabla(W,\bm{q})):(\nabla(Z,\bm{r})),
\end{align*}
where the argument of $V_{,\triple\tripleTau}(\cdot )$ is omitted if evaluated at the reference state.

\begin{lemma}\label{seq_lemma}
There exists $\psi_0 \in {\rm C}^1$ is such that $\psi_n \to \psi_0$ in ${\rm C}^1(B_1(0))$.  Furthermore, there exists a sequence $j_n \to 0$ such that the sequences defined by $X_n, \bm{s}_n$ in~\eqref{splitting} and $Z_n := W_n - X_n$ and $t_n^\alpha := r_n^\alpha - s_n^\alpha$ satisfy the following convergence properties, where $\to$ and $\rightharpoonup$ denote respectively strong and weak $L^2(\R^d)$ convergence.
\begin{equation}\label{verge_result}
\begin{split}
\nabla  W_n \weakto~& \nabla  W_0, \quad
r^\alpha_n \weakto~ r^\alpha_0, \quad
\epsilon_n\nabla r^\alpha_n \weakto 0, \quad
\nabla  X_n \to~ \nabla  W_0, \quad
s_n^\alpha \to~  r^\alpha_0,  \\
\epsilon_n \nabla s^\alpha_n \to~& 0, \quad
\nabla Z_n \weakto~ 0, \quad
t_n^\alpha \weakto~ 0, \quad
\epsilon_n \nabla t_n^\alpha \weakto~ 0, \\
W_n \to~& W_0 \, \mbox{in $L^2(B_1(0))$,} \quad X_n \to~ W_0 \, \mbox{in $L^2(B_1(0))$}, \quad Z_n \to~ 0 \, \mbox{in $L^2(B_1(0))$}
\end{split}
\end{equation}
Moreover, {\helen{let $I$ denote the identity and upon defining the quantities}}
\begin{align*}
 &{\rm R}^{\rm def}_n(x) {\helen{ :={\rm R}^{\rm def}_n\left(\psi_n\right)(x)=~ }} \\
 &\epsilon_n^d \sum_{\xi \in \epsilon_n\mathcal{L}}\sum_{\substack{\triple \\ \tripleTau}} V_{,(\rho\alpha\beta)\tripleTau}(0)D_{\tripleTau,n}(\psi_n( X_n,s_n))\otimes\frac{\rho}{\epsilon_n}\int_0^{\epsilon_n} \zeta_n(\xi + t\rho - x)\, dt, \\
&{\rm R}^{\rm shift}_n(x) {\helen{ :={\rm R}^{\rm shift}_n\left(\psi_n\right)(x)=~ }} \\
&\epsilon_n^d \sum_{\xi \in \epsilon_n\mathcal{L}}\sum_{\substack{\triple \\ \tripleTau}}  V_{,(\rho\alpha\beta)\tripleTau}(0)D_{\tripleTau,n}(\psi_n{X}_n,\psi_n{s}_n) \bar{\zeta}_n(\xi- x), \\
 &{\rm S}^{\rm def}_n(x) {\helen{ := {\rm R}^{\rm def}_n\left(I\right)(x),}} \quad {\rm S}^{\rm shift}_n(x) {\helen{ :={\rm R}^{\rm shift}_n\left(I\right)(x) }} \\
% &\epsilon_n^d \sum_{\xi \in \epsilon_n\mathcal{L}}\sum_{\substack{\triple \\ \tripleTau}} V_{,(\rho\alpha\beta)\tripleTau}(0)D_{\tripleTau,n}(X_n,s_n)\otimes\frac{\rho}{\epsilon_n}\int_0^{\epsilon_n} \zeta_n(\xi + t\rho - x)\, dt, \\
  %&\epsilon_n^d \sum_{\xi \in \epsilon_n\mathcal{L}}\sum_{\substack{\triple \\ \tripleTau}} V_{,(\rho\alpha\beta)\tripleTau}(0)D_{\tripleTau,n}({X}_n,{s}_n) \bar{\zeta}_n(\xi- x), \\
&{\rm S}_n^{\rm inner}(x):=~\\
&\epsilon_n^d \sum_{\xi \in \epsilon_n\mathcal{L}} \sum_{\substack{\triple \\ \tripleTau}} V_{,(\rho\alpha\beta)\tripleTau}:D_{\triple,n}(\psi_n X_n, \psi_n \bm{s}_n):  D_{\tripleTau,n}(X_n,\bm{s}_n),  \\
\end{align*}
the sequence $j_n$ may further be chosen so that
\begin{align}\label{verge_res_2}
%\begin{split}
{\rm S}^{\rm def}_n(x)  &\to \sum_{\substack{\triple \\ \tripleTau}} V_{(\rho\alpha\beta)\tripleTau}(0)\nabla_{\tripleTau}({W}_0,{\bm{s}}_0),\nonumber \\
{\rm S}^{\rm shift}_n(x)  &\to\sum_{\substack{\triple \\ \tripleTau}} V_{(\rho\alpha\beta)\tripleTau}(0)\nabla_{\tripleTau}({W}_0,{\bm{s}}_0),\nonumber\\
 {\helen{ {\rm R}^{\rm def}_n(x)}}  &\helen{\to  \sum_{\substack{\triple \\ \tripleTau}} V_{,(\rho\alpha\beta)\tripleTau}(0)\nabla_{\tripleTau}(\psi_0{W}_0,\psi_0{\bm{s}}_0), } \\
{\rm R}^{\rm shift}_n(x) &\to \sum_{\substack{\triple \\ \tripleTau}} V_{,(\rho\alpha\beta)\tripleTau}(0)\nabla_{\tripleTau}(\psi_0{W}_0,\psi_0{\bm{s}}_0), \nonumber\\
{\rm S}_n^{\rm inner}(x) &\to \int_{\mathbb{R}^d} \sum_{\substack{\triple \\ \tripleTau}} V_{,\triple\tripleTau} :\big(\nabla_{\triple}(\psi_0( W_0, \bm{s}_0))\big):\big(\nabla_{\tripleTau}( W_0,\bm{s}_0)\big) dx,\nonumber
%\end{split}
\end{align}
with convergence being in $L^2(\mathbb{R}^d)$.
\end{lemma}

\begin{proof}
The key fact in proving this result is that $j_n$ may be chosen to tend to zero sufficiently slowly such that any one of these properties holds individually, and by appropriately selecting subsequences using a diagonalization argument, they may be chosen so that all hold simultaneously.  The full proof is given in the Appendix.
\end{proof}

We now state a convergence result for ``cross-terms'' appearing in $\delta \mathcal{F}^{\bqcf}_{{\rm hom},n}(0)$ involving products of strongly and weakly convergent (to zero) sequences. The proof is given in the appendix.

\begin{lemma}\label{more_lemma}
With $Z_n, X_n, \bm{t}_n$, and $\bm{s}_n$ as defined in Lemma~\ref{seq_lemma},
\begin{align}
&\epsilon^d \sum_{\xi \in \epsilon_n\mathcal{L}} \mathbb{C}:D_{n}(\psi_n {Z}_n,\psi_n \bm{t}_n):D_{n}({X}_n,\bm{s}_n) \to 0,
\quad \text{and} \label{more_1} \\
&\epsilon^d \sum_{\xi \in \epsilon_n\mathcal{L}} \mathbb{C}:D_{n}(\psi_n {X}_n,\psi_n\bm{s}_n):D_{n}( {Z}_n, \bm{t}_n) \to 0 \label{more_2}.
\end{align}
\end{lemma}

The next lemma manipulates the product of two weakly convergent sequences.  The idea is that we may shift the blending function \dao{function $\psi_n = 1-\varphi_n$} in a way to use coercivity of the atomistic and continuum Hessians. The proof is again given in the appendix.

\begin{lemma}\label{weak_lemma}
Let $Z_n, X_n, \bm{t}_n$, $\bm{s}_n$, $\theta_n = \sqrt{\psi_n}$, and $\theta_0 = \sqrt{\psi_0}$ be as defined above in Lemma~\ref{seq_lemma}.  Then
\begin{align*}
&\lim_{n\to\infty}\epsilon_n^d \sum_{\xi \in \epsilon_n\mathcal{L}} \mathbb{C}:D_{n}(\theta^2_nZ_n,\theta^2_n\bm{t}_n):D_{n}(Z_n,\bm{t}_n) \\
&=~ \lim_{n\to\infty}\epsilon_n^d \sum_{\xi \in \epsilon_n\mathcal{L}} \mathbb{C}:D_{n}(\theta_nZ_n,\theta_n\bm{t}_n):D_{n}(\theta_nZ_n,\theta_n\bm{t}_n).
\end{align*}
\end{lemma}

We are now positioned to prove Theorem~\ref{stab_theorem}.

\begin{proof}[Proof of Theorem~\ref{stab_theorem}, Stability of BQCF at Reference State]
We use the scaling~\eqref{bqcfHessian} and substitute \dao{(from Lemma~\ref{seq_lemma}) the quantities} $W_n = Z_n + X_n$, $r_n^\alpha = t_n^\alpha + s_n^\alpha$, $\psi_n = 1-\varphi_n$, and $\theta_n = \sqrt{1- {\varphi}_n}$.  We divide the proof into three steps: (1) we derive an expression for the atomistic portion of $\delta\mathcal{F}_{{\rm hom},n}^{\bqcf}(0)$ in the $\liminf$ as $n \to \infty$, (2) we derive an expression for the continuum component of $\delta\mathcal{F}_{{\rm hom},n}^{\bqcf}(0)$, and (3) we combine the results and use stability of the individual atomistic and continuum components to derive a contradiction.

\medskip
\noindent \textit{Step 1:}  The first variation, $\delta\mathcal{F}_{{\rm hom},n}^{\bqcf}(0)$, computed in~\eqref{bqcfHessian} is a sum of an atomistic and continuum component. The discrete, atomistic contribution is
\begin{align}
& \big\< \delta^2 \mathcal{E}^{\rm a}_{{\rm hom}, n } (0)(1-\varphi_n)(W_n,\bm{r}_n),(W_n,\bm{r}_n)\>  \nonumber\\
&= \epsilon_n^d \sum_{\xi \in \epsilon_n\mathcal{L}} \mathbb{C}:D_{n}(\theta^2_n {W}_n,\theta^2_n {\bm{r}}_n):D_{n}( {W}_n, {\bm{r}}_n) \nonumber\\
&= \epsilon_n^d \sum_{\xi \in \epsilon_n\mathcal{L}} \mathbb{C}:D_{n}(\theta^2_n {Z}_n + \theta^2_n {X}_n,\theta^2_n {\bm{t}}_n + \theta^2_n{\bm{s}}_n):D_{n}( {Z}_n +  {X}_n, {\bm{t}}_n +  {\bm{s}}_n) \nonumber\\
&=~ \epsilon_n^d \sum_{\xi \in \epsilon_n\mathcal{L}} \mathbb{C}:\Big[D_{n}(\theta^2_n {Z}_n,\theta^2_n {\bm{t}}_n):D_{n}( {Z}_n, {\bm{t}}_n) +D_{n}(\theta^2_n {Z}_n,\theta^2_n {\bm{t}}_n):D_{n}({X}_n,{\bm{s}}_n)\nonumber \\
&\quad+~ D_{n}(\theta^2_n {X}_n,\theta^2_n {\bm{s}}_n):D_{n}({Z}_n,{\bm{t}}_n) + D_{n}(\theta^2_n {X}_n,\theta^2_n{\bm{s}}_n):D_{n}({X}_n,{\bm{s}}_n)\Big].
\label{eq:step1}
\end{align}
\helen{This final expression consists of four different pairings of the form $D_n(\cdot,\cdot): D_n(\cdot,\cdot)$; upon taking $\liminf$ as $n \to \infty$, we use Lemma~\ref{weak_lemma} on the first pairing, Lemma~\ref{more_lemma} on the second and third pairings, and the final convergence property of $S_n^{\rm inner}(x)$ from Lemma~\ref{seq_lemma} on the fourth pairing to arrive at the following expression for the atomistic contribution}:
\begin{equation}\label{soup_1}
\begin{split}
    &\liminf_{n\to\infty} \big\< \delta^2 \mathcal{E}^{\rm a}_{{\rm hom}, n} (0)(1-\varphi)(W_n,\bm{r}_n),(W_n,\bm{r}_n)\>\\
    &\;=~ \liminf_{n\to\infty}\epsilon_n^d \sum_{\xi \in \epsilon_n\mathcal{L}} \mathbb{C}:D_{n}(\theta_nZ_n,\theta_n\bm{t}_n):D_{n}(\theta_nZ_n,\theta_n\bm{t}_n)  \\
		&\qquad +\int_{\mathbb{R}^d}\mathbb{C}:\nabla (\theta_0^2 W_0, \theta^2_0 \bm{r}_0):\nabla (W_0,\bm{r}_0)\, dx.
\end{split}
\end{equation}

\medskip
\noindent \textit{Step 2:}  Meanwhile, the continuum component of $\delta\mathcal{F}_{{\rm hom},n}^{\bqcf}(0)$ from~\eqref{bqcfHessian} is
\begin{equation}\label{bqcfHessian_cont_part}
\begin{split}
&\big\< \delta^2 \mathcal{E}^{\rm c}(0)  I_{h,n}(\varphi_n W_n,\varphi_n \bm{r}_n), (W_n, \bm{r}_n) \big\> =~ \int_{\mathbb{R}^d}\mathbb{C}:\nabla \big( I_{h,n}(\varphi_n W_n), I_{h,n}(\varphi_n\bm{r}_n)\big): \nabla (W_n,\bm{r}_n)\, dx.
\end{split}
\end{equation}

Using standard $\mathcal{P}_1$-nodal interpolation error estimates \dao{and the fact that each $\nabla \varphi_n$ has support on $B_1$}, it is straightforward to
prove that (c.f. Lemma~\ref{p1_lemma})
\begin{equation} \label{eq:conv-(Ih-I)(phiW)}
   \begin{split}
   & \lim_{n\to \infty}\|\nabla I_{h,n}(\varphi_n W_n) - \nabla (\varphi_n W_n)\|_{L^2(\mathbb{R}^d)} = 0, \\
   &  \lim_{n\to \infty}\|I_{h,n}(\varphi_n r^\alpha_n) - (\varphi_n r^\alpha_n)\|_{L^2(\mathbb{R}^d)} = 0.
   \end{split}
\end{equation}

Thus, taking the $\liminf$ of~\eqref{bqcfHessian_cont_part} and applying \eqref{eq:conv-(Ih-I)(phiW)} we obtain
\begin{equation}\label{bqcfHessian_cont_part_no_interp}
\liminf_{n\to\infty}\big\< \delta^2 \mathcal{E}^{\rm c}(0) I_{h,n}(\varphi W_n,\varphi\bm{r}_n), (W_n, \bm{r}_n) \big\>
= \liminf_{n\to\infty} \int_{\mathbb{R}^d}\mathbb{C}:\nabla(\varphi_nW_n, \varphi_n\bm{r}_n): \nabla(W_n,\bm{r}_n)\, dx.
\end{equation}
Substituting the decomposition $(W_n, \bm{r}_n) := (Z_n + X_n, \bm{t}_n + \bm{s}_n)$ into~\eqref{bqcfHessian_cont_part_no_interp} yields
\begin{equation}\label{long_limit}
\begin{split}
&\liminf_{n\to\infty}\big\< \delta^2 \mathcal{E}^{\rm c}(0)  I_{h,n}(\varphi W_n,\varphi\bm{r}_n), (W_n, \bm{r}_n) \big\>  \\
&=~\liminf_{n\to\infty}\int_{\mathbb{R}^d}\Big[ \mathbb{C}: \nabla (\varphi_n Z_n, \varphi_n\bm{t}_n):\nabla (Z_n, \bm{t}_n) + \mathbb{C}: \nabla (\varphi_n Z_n, \varphi_n\bm{t}_n):\nabla (X_n, \bm{s}_n) \\
&\qquad +  \mathbb{C}: \nabla (\varphi_n X_n, \varphi_n\bm{s}_n):\nabla (Z_n, \bm{t}_n) + \mathbb{C}: \nabla (\varphi_n X_n, \varphi_n\bm{s}_n):\nabla (X_n, \bm{s}_n)\Big]\, dx.
\end{split}
\end{equation}
\dao{This final expression again gives four pairings just as in step one but now of the form $\nabla(\cdot, \cdot): \nabla (\cdot, \cdot)$. The first pairing we momentarily leave alone, the second and third pairings both converge to zero by virtue of strong convergence of $\nabla X_n, \bm{s}_n$  and weak convergence of $\nabla Z_n, \bm{t}_n$ to $0$ from Lemma~\ref{seq_lemma}, and the final pairing converges to $\nabla(\varphi_0 W_0,\varphi_0\bm{r}_0):\nabla(W_0,\bm{r}_0)$ again as a result of the strong convergence properties of $\nabla X_n, \bm{s}_n$ from Lemma~\ref{seq_lemma}.  These facts simplify~\eqref{long_limit} to}
\begin{equation}\label{soup11}
\begin{split}
&\liminf_{n\to\infty}\big\< \delta^2 \mathcal{E}^{\rm c}(0) I_{h,n}(\varphi W_n,\varphi\bm{r}_n), (W_n, \bm{r}_n) \big\> \\
&=~ \liminf_{n\to\infty}\int_{\mathbb{R}^d}\big[\mathbb{C}: \nabla (\varphi_n Z_n, \varphi_n \bm{t}_n): \nabla(Z_n,\bm{t}_n) \\
&\qquad\qquad\qquad\qquad+ \mathbb{C}:\nabla(\varphi_0 W_0,\varphi_0\bm{r}_0):\nabla(W_0,\bm{r}_0)\big]\, dx.
\end{split}
\end{equation}
As in the atomistic case, our goal is again to think of $\varphi_n$ as a square, $\varphi_n := \sqrt{\varphi_n}^2$ and to shift one factor of $\sqrt{\varphi_n}$ to each component of the duality pairing.  Using an argument very similar to that in the proof of Lemma~\ref{weak_lemma} (which we therefore omit) we obtain
\begin{equation*}\label{limit_term1_soup11}
\begin{split}
&\liminf_{n\to\infty}\int_{\mathbb{R}^d}\mathbb{C}: \nabla(\varphi_n Z_n,\varphi_n \bm{t}_n): \nabla(Z_n,\bm{t}_n)\\
&\quad=~ \liminf_{n\to\infty} \int_{\mathbb{R}^d} \mathbb{C}: \nabla(\sqrt{\varphi_n} Z_n, \sqrt{\varphi_n} \bm{t}_n) : \nabla(\sqrt{\varphi_n} Z_n,  \sqrt{\varphi_n}\bm{t}_n).
\end{split}
\end{equation*}

% To do that, we use the following lemma whose proof is very similar to that of Lemma~\ref{weak_lemma} and is thus omitted.
% %%%
% {\begin{lemma}\label{lem:term1_soup11}
% With $Z_n$ and $\bm{t}_n$ defined as in Lemma~\ref{seq_lemma},
% \begin{equation}\label{limit_term1_soup11}
% \begin{split}
% &\liminf_{n\to\infty}\int_{\mathbb{R}^d}\mathbb{C}: \nabla(\varphi_n Z_n,\varphi_n \bm{t}_n): \nabla(Z_n,\bm{t}_n)\\
% &\quad=~ \liminf_{n\to\infty} \int_{\mathbb{R}^d} \mathbb{C}: \nabla(\sqrt{\varphi_n} Z_n, \sqrt{\varphi_n} \bm{t}_n) : \nabla(\sqrt{\varphi_n} Z_n,  \sqrt{\varphi_n}\bm{t}_n).
% \end{split}
% \end{equation}
% \end{lemma}
% }
%%%

Inserting the last result into \eqref{soup11}, we obtain {
\begin{align}\label{limit_bqcfH_cont_part_no_interp}
%\begin{split}
&\liminf_{n\to\infty}\big\< \delta^2 \mathcal{E}^{\rm c}(0)  I_{h,n}(\varphi W_n,\varphi\bm{r}_n), (W_n, \bm{r}_n) \big\> \\
&=~ \liminf_{n\to\infty} \int_{\mathbb{R}^d} \big[\mathbb{C}: \nabla(\sqrt{\varphi_n} Z_n, \sqrt{\varphi_n} \bm{t}_n) : \nabla(\sqrt{\varphi_n} Z_n,  \sqrt{\varphi_n}\bm{t}_n) +\mathbb{C}:\nabla (\varphi_0 W_0,\varphi_0 \bm{r}_0): \nabla (W_0,\bm{r}_0)\big]\, dx.\nonumber
%\end{split}
\end{align}
}
\medskip
\noindent \textit{Step 3:}  Upon adding the atomistic components from~\eqref{soup_1} to {{ the continuum contributions \eqref{limit_bqcfH_cont_part_no_interp}}} and recalling that $\theta_0^2 = 1-\varphi_0$, we have the following expression for $\delta\mathcal{F}_{{\rm hom},n}^{\bqcf}(0)$:
\begin{equation}\label{short_limit}
\begin{split}
&\liminf_{n\to\infty}\<\delta \mathcal{F}_{{\rm hom},n}^{\bqcf}(0)(W_n, \bm{r}_n), (W_n, \bm{r}_n)\> \\
&=~ \liminf_{n\to\infty} \int_{\mathbb{R}^d} \big[\mathbb{C}: \nabla(\sqrt{\varphi_n} Z_n, \sqrt{\varphi_n} \bm{t}_n) : \nabla(\sqrt{\varphi_n} Z_n,  \sqrt{\varphi_n}\bm{t}_n) +\mathbb{C}:\nabla (W_0,\bm{r}_0): \nabla (W_0,\bm{r}_0) \big]\, dx\\
&\qquad +~ \liminf_{n\to\infty}\epsilon_n^d \sum_{\xi \in \epsilon_n\mathcal{L}} \mathbb{C}:D_{n}(\sqrt{1-\varphi_n}Z_n,\sqrt{1-\varphi_n}\bm{t}_n):D_{n}(\sqrt{1-\varphi_n}Z_n,\sqrt{1-\varphi_n}\bm{t}_n)
\end{split}
\end{equation}
\dao{Next, using stability of the homogeneous atomistic model {\helen{in this scaling}},
\[
\la \delta^2 \mathcal{E}^{\a}_{{\rm hom},n}(0)(W_n, \bm{r}_n), (W_n, \bm{r}_n)\ra
\ge \gamma_\a \|(W_n, \bm{r}_n)\|_{\rm a}^2,
\]
(which can easily be proven (c.f.~\cite{olsonOrtner2016,Ehrlacher2013}) due to Assumption~\ref{assumption2}) and the fact that atomistic stability implies Cauchy--Born Stability~\cite[Theorem 3.6]{olsonOrtner2016}, that is,
\begin{equation*}\label{CB_stability}
\begin{split}
\la \delta^2& \mathcal{E}^{\c}(0)(W_n, \bm{r}_n), (W_n, \bm{r}_n)\ra
\ge \gamma_\a \|(W, \bm{r})\|_{\rm ml}^2,
\end{split}
\end{equation*}
we hence have from~\eqref{short_limit} that
\begin{align}\label{red_pill}
&\liminf_{n\to\infty} \la \delta  \mathcal{F}_{{\rm hom},n}^{\bqcf}(0) (W_n, \bm{r}_n),(W_n, \bm{r}_n) \\
&=~ \liminf_{n\to\infty}\big[\la\delta^2 \mathcal{E}^\c (\sqrt{\varphi_n} Z_n, \sqrt{\varphi_n} \bm{t}_n), (\sqrt{\varphi_n} Z_n, \sqrt{\varphi_n} \bm{t}_n) \ra + \la\delta^2 \mathcal{E}^\c (W_0,\bm{r}_0), (W_0,\bm{r}_0) \ra \nonumber\\
& \qquad +~ \la \delta^2 \mathcal{E}^\a_{{\rm hom},n} (\sqrt{1-\varphi_n}Z_n,\sqrt{1-\varphi_n}\bm{t}_n), (\sqrt{1-\varphi_n}Z_n,\sqrt{1-\varphi_n}\bm{t}_n)\ra \big]\nonumber\\
&\geq \liminf_{n\to\infty} \gamma_\a\Big[\|\nabla (\sqrt{\varphi_n} Z_n)\|^2_{L^2(\mathbb{R}^d)} + \|\sqrt{\varphi_n}\bm{t}_n\|^2_{L^2(\mathbb{R}^d)} + \|\nabla{W}_0\|^2_{L^2(\mathbb{R}^d)} + \|\bm{r}_0\|^2_{L^2(\mathbb{R}^d)}  \nonumber\\
&\qquad +  \|\nabla I_n(\sqrt{1-\varphi_n} Z_n)\|^2_{L^2(\mathbb{R}^d)} + \|I_n(\sqrt{1-\varphi_n}\bm{t}_n)\|^2_{L^2(\mathbb{R}^d)}\Big].\nonumber
%&\geq~ \gamma_\a\left[ \|\nabla\hat{W}_0\|^2 + \sum\| {r}_0^\beta -  {r}_0^\alpha\|^2 + \|\nabla {Z}_n\|^2 + \sum\| {t}_n^\beta - {t}_n^\alpha\|^2\right]
\end{align}
}
Similar to \eqref{eq:conv-(Ih-I)(phiW)} (c.f. Lemma~\ref{p1_lemma}), standard nodal interpolation error estimates imply that
\begin{align*}
 \lim_{n\to \infty}\|\nabla I_n(\sqrt{1-\varphi_n} Z_n) - \nabla (\sqrt{1-\varphi_n} Z_n)\|_{L^2(\mathbb{R}^d)} =& 0, \quad \text{and} \\
%\| I_n(\sqrt{1-\varphi_n}t_n^\alpha)-  (\sqrt{1-\varphi_n}t_n^\alpha)\|_{L^2(\mathbb{R}^d)} \to~& 0.\\
   \lim_{n\to \infty}\| I_n(\sqrt{1-\varphi_n}\bm{t}_n)-  (\sqrt{1-\varphi_n}\bm{t}_n)\|_{L^2(\mathbb{R}^d)}  =& 0.
\end{align*}
Thus,~\eqref{red_pill} becomes
\begin{equation}\label{soup14}
\begin{split}
&\liminf_{n\to\infty} \la \delta  \mathcal{F}_{{\rm hom},n}^{\bqcf}(0) (W_n, \bm{r}_n),(W_n, \bm{r}_n) \\
&\quad\geq~ \liminf_{n\to\infty} \gamma_\a\Big[\|\nabla (\sqrt{\varphi_n} Z_n)\|^2_{L^2(\mathbb{R}^d)} + \|\sqrt{\varphi_n}\bm{t}_n\|^2_{L^2(\mathbb{R}^d)} + \|\nabla{W}_0\|^2_{L^2(\mathbb{R}^d)} + \|\bm{r}_0\|^2_{L^2(\mathbb{R}^d)} \\
&\qquad+~ \|\nabla (\sqrt{1-\varphi_n} Z_n)\|^2_{L^2(\mathbb{R}^d)} + \|\sqrt{1-\varphi_n}\bm{t}_n\|^2_{L^2(\mathbb{R}^d)} \Big] \\
&\quad=~ \liminf_{n\to\infty}\gamma_\a\Big[\|\nabla (\sqrt{\varphi_n} Z_n)\|^2_{L^2(\mathbb{R}^d)} + \|\nabla (\sqrt{1-\varphi_n} Z_n)\|^2_{L^2(\mathbb{R}^d)} + \|\bm{t}_n\|^2_{L^2(\mathbb{R}^d)} \\
&\qquad\qquad\qquad\qquad  + \|\nabla{W}_0\|^2_{L^2(\mathbb{R}^d)} + \|\bm{r}_0\|^2_{L^2(\mathbb{R}^d)}\Big].
\end{split}
\end{equation}
Next observe
\begin{equation}\label{soup15}
\begin{split}
&\|\nabla (\sqrt{\varphi_n} Z_n)\|^2_{L^2(\mathbb{R}^d)} + \|\nabla (\sqrt{1-\varphi_n} Z_n)\|^2_{L^2(\mathbb{R}^d)} \\ % =~ \int |\nabla(\sqrt{\varphi_n} Z_n)|^2 + |\nabla (\sqrt{1-\varphi_n} Z_n)|^2 \\
&=~ \int \Big[|\nabla(\sqrt{\varphi_n}) \otimes Z_n + \sqrt{\varphi_n}\nabla Z_n|^2 + |\nabla (\sqrt{1-\varphi_n}) \otimes Z_n + \sqrt{1-\varphi_n}\nabla Z_n|^2\Big] \, dx\\
&=~ \int \Big[2\nabla(\sqrt{\varphi_n}) \otimes Z_n : \sqrt{\varphi_n}\nabla Z_n + |\nabla(\sqrt{\varphi_n}) \otimes Z_n|^2 + \varphi_n|\nabla Z_n|^2\Big]\, dx  \\
& +\int \Big[2\nabla(\sqrt{1-\varphi_n}) \otimes Z_n : \sqrt{1-\varphi_n}\nabla Z_n + |\nabla(\sqrt{1-\varphi_n}) \otimes Z_n|^2 + (1-\varphi_n)|\nabla Z_n|^2\big]\, dx.
\end{split}
\end{equation}
Since $Z_n$ converges strongly to zero in $L^2({\rm supp}(\nabla(\sqrt{1-\varphi_n})))$ by Lemma~\ref{seq_lemma} (${\rm supp}(\nabla(\sqrt{1-\varphi_n})) \subset B_1(0)$), it follows from~\eqref{soup15} that
\begin{equation}\label{soup16}
\liminf_{n\to\infty} \|\nabla (\sqrt{\varphi_n} Z_n)\|^2_{L^2(\mathbb{R}^d)} + \|\nabla (\sqrt{1-\varphi_n} Z_n)\|^2_{L^2(\mathbb{R}^d)} =~ \liminf_{n\to\infty} \|\nabla Z_n\|^2_{L^2(\mathbb{R}^d)}.
\end{equation}
Substituting~\eqref{soup16} into~\eqref{soup14} produces
\begin{equation}\label{soup17}
\begin{split}
&\liminf_{n\to\infty} \la \delta  \mathcal{F}_{{\rm hom},n}^{\bqcf}(0) (W_n, \bm{r}_n),(W_n, \bm{r}_n)\ra \\
&\quad\geq~ \liminf_{n\to\infty} \gamma_\a\Big[\|\nabla Z_n\|^2_{L^2(\mathbb{R}^d)} + \|\bm{t}_n\|^2_{L^2(\mathbb{R}^d)} + \|\nabla{W}_0\|^2_{L^2(\mathbb{R}^d)} + \|\bm{r}_0\|^2_{L^2(\mathbb{R}^d)}\Big] \\
&\quad\geq~ \liminf_{n\to\infty} \gamma_\a\Big[\|\nabla Z_n\|^2_{L^2(\mathbb{R}^d)} + \|\bm{t}_n\|^2_{L^2(\mathbb{R}^d)} + \|\nabla{X}_n\|^2_{L^2(\mathbb{R}^d)} + \|\bm{s}_n\|^2_{L^2(\mathbb{R}^d)}\Big] \\
&\quad=~\liminf_{n \to \infty} \gamma_\a \big[ \|\nabla {W}_n\|^2_{L^2(\mathbb{R}^d)} + \|\bm{r}_n\|^2_{L^2(\mathbb{R}^d)} \big] = \gamma_\a,
\end{split}
\end{equation}
% Finally, note
% \begin{equation}\label{soup18}
% \begin{split}
% &\|\nabla {W}_0\|^2_{L^2(\mathbb{R}^d)} + \helen{\|\bm{r}_0\|^2_{L^2(\mathbb{R}^d)}} + \|\nabla {Z}_n\|^2_{L^2(\mathbb{R}^d)} + \helen{\|\bm{t}_n\|^2_{L^2(\mathbb{R}^d)}} \\
% &=~ \|\nabla {W}_0\|^2_{L^2(\mathbb{R}^d)} - \|\nabla  {X}_n\|^2_{L^2(\mathbb{R}^d)} + \|\nabla  {X}_n\|^2_{L^2(\mathbb{R}^d)} + \|\nabla {Z}_n\|^2_{L^2(\mathbb{R}^d)} +    \\
% &\qquad\qquad + \sum\limits_{\alpha}\big[\|{r}_0^\alpha\|^2_{L^2(\mathbb{R}^d)} - \| {s}_n^\alpha\|^2_{L^2(\mathbb{R}^d)} + \|{s}_n^\alpha\|^2_{L^2(\mathbb{R}^d)} + \|{t}_n^\alpha\|^2_{L^2(\mathbb{R}^d)}\big] \\
% &=~ \|\nabla (X_n + Z_n)\|^2_{L^2(\mathbb{R}^d)} - 2(\nabla X_n, \nabla Z_n)_{L^2(\mathbb{R}^d)} + \sum\limits_{\alpha}\big[\| {s}_n^\alpha+  {t}_n^\alpha\|^2_{L^2(\mathbb{R}^d)} - 2( {s}_n^\alpha,  {t}_n^\alpha)_{L^2(\mathbb{R}^d)}  \big] \\
% &=~ \|\nabla {W}_n\|^2_{L^2(\mathbb{R}^d)} + \helen{\|\bm{r}_n\|^2_{L^2(\mathbb{R}^d)}} .
% \end{split}
% \end{equation}
% Upon substituting~\eqref{soup18} into~\eqref{soup17}, we arrive at
% \[
% \liminf_{n\to\infty} \la \delta  \mathcal{F}_{\rm hom}^{\bqcf}(0) (W_n, \bm{r}_n),(W_n, \bm{r}_n)\ra \geq~ \liminf_{n\to\infty} \gamma_\a\left[\|\nabla {W}_n\|^2_{L^2(\mathbb{R}^d)} + \|\bm{r}_n\|^2_{L^2(\mathbb{R}^d)}\right] = \gamma_\a,
% \]
which contradicts our assumption in~\eqref{contra_seq}.
\end{proof}

\subsection{Reference Stability Implies Defect Stability}\label{stab_defect_present}

Having established stability of the homogeneous BQCF operator we obtain
stability of $\delta \mathcal{F}^{\bqcf}(\Pi_{h,n}(U^\infty,\bm{p}^\infty))$,
i.e. Theorem \ref{stab_theorem_full}, as a relatively straightforward
consequence. Before entering into the proof we remark that we now no longer employ the rescalings of Section~\ref{stab_def_free}. \dao{The basic idea of the proof is that the linearized homogeneous BQCF operator and linearized BQCF operator agree for any $(W,\bm{r})$ which is zero in a large enough neighborhood about the origin.  Thus, to prove stability of the true linearized BQCF operator, we again consider the possibility of a sequence, $(W_n, \bm{r}_n)$, of unstable modes whose support is contained in larger and larger balls about the origin. We will then split each $(W_n,\bm{r}_n)$ into components concentrated near the origin (where we can use atomistic stability) and correction terms supported far from the origin where we use stability of the linearized homogeneous operator. As before, the main difficulty is converting the atomistic component of the BQCF operator to a form where we may utilize atomistic coercivity.}

\begin{proof}[Proof of Theorem~\ref{stab_theorem_full}]
We prove this result by contradiction as well. Therefore suppose, as in the proof of Theorem~\ref{stab_theorem}, that there exists $R_{\a,n} \to \infty$ with associated meshes $\mathcal{T}_{h,n}$, blending functions $\varphi_n$, {\helen{and test pairs $(W_n, \bm{r}_n) \in \bm{\mathcal{U}}_{h,0}^n \times \bm{\mathcal{P}}_{h,0}^n$ with norm scaled to one, such that}}
\begin{equation}\label{contra_seq_def}
\begin{split}
&\frac{\gamma_\a}{2} >
\sum_{\xi \in \mathcal{L}} \sum_{(\rho\alpha\beta)}\sum_{(\tau\gamma\delta)} V_{,(\rho\alpha\beta)(\tau\gamma\delta)}(D\bm{U}_n):D_{(\rho\alpha\beta)}((1-\hat{\varphi}_n)\hat{W}_n,(1-\hat{\varphi}_n)\hat{\bm{r}}_n):D_{(\rho\alpha\beta)}(\hat{W}_n,\hat{\bm{r}}_n) \\
& \qquad \qquad + \int_{\mathbb{R}^d} \sum_{(\rho\alpha\beta)}\sum_{(\tau\gamma\delta)} V_{,(\rho\alpha\beta)(\tau\gamma\delta)}(\nabla \bm{U}_n):\nabla_{\rho\alpha\beta}( I_{h,n}(\hat{\varphi}_n(\hat{W}_n, \hat{\bm{r}}_n))) :\nabla_{\rho\alpha\beta}(\hat{W}_n, \hat{\bm{r}}_n) \, dx \\
&\quad=: \<\delta \mathcal{F}^{\bqcf}_n(\bm{U}_n)(W_n, \bm{r}_n),(W_n, \bm{r}_n)\>,
\end{split}
\end{equation}
where, for notational simplicity we have defined $\bm{U}_n := \Pi_{h,n}(U^\infty,\bm{p}^\infty)$ and redefined $\delta \mathcal{F}^{\bqcf}_n$ from the previous section without a scaling by $\epsilon$.

Upon extracting a subsequence, we may assume without loss of generality that $\nabla W_n \weakto \nabla W_0$ for $W_0 \in \dot{\bm{H}}^1$ and $\bm{r}_n \weakto \bm{r}_0 \in L^2$. For each $R_{\a,n}$, $W_n$ and $\bm{r}_n$ are piecewise linear with respect to the mesh $\mathcal{T}_\a$ on $\Omega_{\a, n}$.  Hence the convergence is strong on any finite collection of elements on $\mathcal{T}_\a$ since weak convergence implies strong convergence on finite dimensional spaces. It also follows from the full refinement of the mesh assumption on $\Omega_{\a, n}$ that $W_0$ and $\bm{r}_0$ are also piecewise linear with respect to $\mathcal{T}_\a$.

Having established these basic facts, we will yet again split $(W_n, \bm{r}_n)$ into the sum of a strongly convergent sequence and weakly convergent sequence as in~\cite[Theorem 4.9]{blended2014}.  For each $n$, we take $\eta_n(x)$ to be a smooth bump function satisfying $\eta_n(x) = 1$ on $B_{1/2r_{{\rm core},n}}(0)$ and $\eta_n(x)$ has support contained in $B_{r_{{\rm core},n}}(0)$.  Similar to the definition of $\Pi_h$, we then set 
\[
\dao{A_{n} := B_{r_{{\rm core},n}}\setminus B_{(1/2)r_{{\rm core},n}} + B_{2r_{\rm buff}}}
\]
 and
\begin{equation}\label{DefectStab_test}
X_n := I_n(\eta_n W_0) - I_n(\eta_n)\dashint_{A_n}  W_0\, dx ,\quad Z_n := W_n - X_n,\quad \bm{s}_n := I_n(\eta_n \bm{r}_0),\quad \bm{t}_n := \bm{r}_n - \bm{s}_n.
\end{equation}
Similar to Lemma~\ref{approx_lem}, we have, with these definitions,
\begin{equation*}\label{convo_props}
\nabla X_n \to \nabla W_0, \quad \mbox{and} \quad \nabla Z_n \weakto 0
   \qquad \text{in } L^2(\R^d)
\end{equation*}
and
\begin{align*}\label{convo_props1}
\bm{s}_n \to \bm{r}_0, \quad &\mbox{and} \quad \bm{t}_n \weakto 0
\qquad \text{in } L^2(\R^d).
\end{align*}
Then we note that the norm defined by
\[
\| (U,\bm{p}) \|_{\a_1}^2 := \sum_{\xi \in \mathcal{L}} |D(U,\bm{p})(\xi)|^2, \quad \mbox{where} \quad |D(U,\bm{p})(\xi)|^2 := \sum_{\triple \in \mathcal{R}} |D_{\triple} (U,\bm{p})(\xi)|^2.
\]
is equivalent to the $\|\cdot\|_{\a}$ norm on $\bm{\mathcal{U}}$ by~\cite[Lemma 2.1]{olsonOrtner2016}. Thus, \dao{since we are dealing with functions which are $\mathcal{P}^1$ with respect to $\mathcal{T}_\a$ on a growing atomistic region, then the continuous convergence results for $\nabla X_n, \nabla Z_n, \bm{s}_n$, and $\bm{t}_n$ imply corresponding discrete convergence results:}
\begin{align}\label{convo_props2}
D(X_n, \bm{s}_n) \to D(W_0, \bm{r}_0) \quad  &\mbox{and} \quad D(Z_n, \bm{t}_n) \weakto 0 \qquad \mbox{in $\ell^2(\mathcal{L})$}.
\end{align}
With this decomposition, \dao{we now substitute the test pair {\helen{$(W_n, \bm{r_n}) = (X_n + Z_n, \bm{s}_n + \bm{t}_n)$ from \eqref{DefectStab_test} into}}}
\begin{equation}\label{steamroll}
\begin{split}
&\<\delta \mathcal{F}^{\bqcf}_{ n}(\bm{U}_n)(X_n + Z_n,\bm{s}_n + \bm{t}_n),(X_n + Z_n,\bm{s}_n + \bm{t}_n)\> \\
&=~ \<\delta  \mathcal{F}^{\bqcf}_{n} (\bm{U}_n)(X_n,\bm{s}_n),(X_n,\bm{s}_n)\> + \<\delta \mathcal{F}^{\bqcf}_{n} (\bm{U}_n)(X_n,\bm{s}_n),(Z_n,\bm{t}_n)\> \\
&\qquad +~  \<\delta  \mathcal{F}^{\bqcf}_{ n} (\bm{U}_n)(Z_n,\bm{t}_n),(X_n,\bm{s}_n)\> + \<\delta \mathcal{F}^{\bqcf}_{ n} (\bm{U}_n)(Z_n,\bm{t}_n),(Z_n,\bm{t}_n)\>.
\end{split}
\end{equation}
{\helen{
Also recall the definition of $\delta\mathcal{F}_n^{\rm bqcf}$, which is
\begin{align*}
\la&\delta \mathcal{F}^{\rm bqcf}_n(\bm{U}_n)(W_n, \bm{r}_n),(W_n, \bm{r}_n)\ra\\
%&:=\sum_{\xi \in \mathcal{L}} \sum_{\substack{(\rho\alpha\beta) \\ (\tau\gamma\delta)}} V_{,(\rho\alpha\beta)(\tau\gamma\delta)}(D\bm{U}_n):D_{(\rho\alpha\beta)}((1-\varphi_n)(W_n,\bm{r}_n)):D_{\tripleTau}(W_n,\bm{r}_n) \\
%&\qquad ~+ \int_{\mathbb{R}^d} \sum_{\substack{(\rho\alpha\beta) \\ (\tau\gamma\delta)}} V_{,(\rho\alpha\beta)(\tau\gamma\delta)}(\nabla \bm{U}_n):\nabla_{\rho\alpha\beta}( I_{h,n}(\varphi_n(W_n, \bm{r}_n))) :\nabla_{\tripleTau}(W_n, \bm{r}_n) \, dx\\
&= \la\delta^2 \mathcal{E}^{\rm a}(\bm{U}_n)\big((1-\varphi_n)(W_n, \bm{r}_n)\big),(W_n, \bm{r}_n)\ra \\
& \qquad +\la\delta^2 \mathcal{E}^{\rm c}(\bm{U}_n)\big(\varphi_n(W_n, \bm{r}_n)\big),(W_n, \bm{r}_n)\ra.
\end{align*}
}}
Since $D(X_n, \bm{s}_n)$ each have support where $\varphi_n = 0$ and $\Pi_{h,n}(\bm{U}_n) = (\bm{U}_n)$ there, we can rewrite the first three terms of~\eqref{steamroll} without the blending function as
\begin{equation}\label{stew1}
\begin{split}
&\<\delta  \mathcal{F}^{\bqcf}_{n} (\bm{U}_n)(X_n + Z_n,\bm{s}_n + \bm{t}_n),(X_n + Z_n,\bm{s}_n + \bm{t}_n)\> \\
&=~ \<\delta^2 \mathcal{E}^{\a}(\bm{U}_n)(X_n,\bm{s}_n),(X_n,\bm{s}_n)\> + \<\delta^2 \mathcal{E}^{\a} (\bm{U}_n)(X_n,\bm{s}_n),(Z_n,\bm{t}_n)\> \\
&\qquad +~  \<\delta^2 \mathcal{E}^{\a} (\bm{U}_n)(Z_n,\bm{t}_n),(X_n,\bm{s}_n)\> + \<\delta \mathcal{F}^{\bqcf}_{ n} (\bm{U}_n)(Z_n,\bm{t}_n),(Z_n,\bm{t}_n)\>.
\end{split}
\end{equation}
Moreover, $D(Z_n, \bm{t}_n)$ has support only where $V_\xi \equiv V$ and so from the convergence properties~\eqref{convo_props2}, it follows that $\<\delta^2 \mathcal{E}^{\a}(\bm{U}_n)(X_n,\bm{s}_n),(Z_n,\bm{t}_n)\>$ and $\<\delta^2 \mathcal{E}^{\a}(\bm{U}_n)(Z_n,\bm{t}_n),(X_n,\bm{s}_n)\>$ both go to zero as $n \to \infty$.

For the first term in \eqref{stew1}, using the atomistic stability assumption, Assumption~\ref{assumption2}, we obtain
\begin{equation}\label{stew2}
\<\delta^2 \mathcal{E}^{\a}  (\bm{U}_n)(X_n,\bm{s}_n),(X_n,\bm{s}_n)\> \geq~ \gamma_\a \|(X_n, \bm{s}_n)\|_{\rm ml}^2.
\end{equation}
Thus, taking the lim inf as $n \to \infty$ in~\eqref{stew1} yields
\begin{equation}\label{steam_peas}
\begin{split}
&\liminf_{n \to \infty}\<\delta  \mathcal{F}^{\bqcf}_{n} (\bm{U}_n)(X_n + Z_n,\bm{s}_n + \bm{t}_n),(X_n + Z_n,\bm{s}_n + \bm{t}_n)\> \\
&\qquad \geq \liminf_{n \to \infty} \gamma_\a \|(X_n, \bm{s}_n)\|_{\rm ml}^2 + \liminf_{n \to \infty} \<\delta \mathcal{F}^{\bqcf}_{ n} (\bm{U}_n)(Z_n,\bm{t}_n),(Z_n,\bm{t}_n)\>
\end{split}
\end{equation}
Thus, we are only left to treat $\<\delta \mathcal{F}^{\bqcf}_{ n} (\bm{U}_n)(Z_n,\bm{t}_n),(Z_n,\bm{t}_n)\>$, the
far-field contribution, as defined in \eqref{DefectStab_test}. The strategy here is that far from the defect core we
may replace $\delta\mathcal{F}^{\bqcf}_n(\bm{U}_n)$ with $\delta\mathcal{F}^{\bqcf}_{{\rm hom},n}(0)$ and then apply Theorem \ref{stab_theorem}. Thus, we first estimate,
\begin{equation}\label{stew3}
\begin{split}
&\<\delta \mathcal{F}^{\bqcf}_{n}(\bm{U}_n)(Z_n,\bm{t}_n),(Z_n,\bm{t}_n)\>  \\
&=~ \<\delta \mathcal{F}_{{\rm hom}, n }^{\bqcf}(\bm{U}_n)(Z_n,\bm{t}_n),(Z_n,\bm{t}_n)\> \\
%&=~ \<\delta^2 \mathcal{E}^\a_{\rm hom}(\Pi_{h,n}(U^\infty,\bm{p}^\infty))(1-\varphi_n)(Z_n,\bm{t}_n),(Z_n,\bm{t}_n)\> + \<\delta^2 \mathcal{E}^\c(\Pi_{h,n}(U^\infty,\bm{p}^\infty))I_{h,n}(\varphi_n(Z_n,\bm{t}_n)),(Z_n,\bm{t}_n)\> \\
%&-~ \<\delta^2 \mathcal{E}^\a_{\rm hom}(0)(1-\varphi_n)(Z_n,\bm{t}_n),(Z_n,\bm{t}_n)\> + \<\delta^2 \mathcal{E}^\c(0)I_{h,n}(\varphi_n(Z_n,\bm{t}_n)),(Z_n,\bm{t}_n)\> \\
%&+~ \<\delta^2 \mathcal{E}^\a_{\rm hom}(0)(1-\varphi_n)(Z_n,\bm{t}_n),(Z_n,\bm{t}_n)\> + \<\delta^2 \mathcal{E}^\c(0)I_{h,n}(\varphi_n(Z_n,\bm{t}_n)),(Z_n,\bm{t}_n)\>.
&=~ \<\delta \mathcal{F}_{{\rm hom}, n} ^{\bqcf}(0)(Z_n,\bm{t}_n),(Z_n,\bm{t}_n)\>
   + \big\<\big[ \delta \mathcal{F}_{{\rm hom}, n}^{\bqcf}(\bm{U}_n) - \delta \mathcal{F}_{{\rm hom}, n }^{\bqcf}(0) \big] (Z_n,\bm{t}_n),(Z_n,\bm{t}_n)\big \>\\
% &\qquad + \big\<[(Z_n,\bm{t}_n),(Z_n,\bm{t}_n)\> \\
&\geq~ \frac{3}{4}\gamma_\a \|(Z_n, \bm{t}_n)\|_{\rm ml}^2 +
\big\<\big[ \delta \mathcal{F}_{{\rm hom},n}^{\bqcf}(\bm{U}_n) - \delta \mathcal{F}_{{\rm hom},n }^{\bqcf}(0) \big] (Z_n,\bm{t}_n),(Z_n,\bm{t}_n)\big \>.
\end{split}
\end{equation}
where we applied Theorem~\ref{stab_theorem} in the final step. ({\helen{Note that there is a slight}} notational discrepancy in that our $\mathcal{F}^\bqcf_{{\rm hom},n}$ is indexed by $n$ here while there is no index in Theorem~\ref{stab_theorem}.  However, we may still use this theorem since $R_{{\rm core},n} \to \infty$ so we may assume $R_{{\rm core},n} \geq R_{\rm core}^*$ in the statement of that theorem.)

Next, we estimate the remaining group in \eqref{stew3},
\begin{equation*}\label{stew4}
\begin{split}
   &\hspace{-1cm}\big\<\big[ \delta \mathcal{F}_{{\rm hom}, n }^{\bqcf}(\bm{U}_n) - \delta \mathcal{F}_{{\rm hom},n}^{\bqcf}(0) \big] (Z_n,\bm{t}_n),(Z_n,\bm{t}_n)\big \> \\
&\leq~ \big|\big\<\big[\delta^2 \mathcal{E}^\a_{\rm hom}(\bm{U}_n) - \delta^2\mathcal{E}^\a_{\rm hom}(0)\big] (1-\varphi_n)(Z_n,\bm{t}_n),(Z_n,\bm{t}_n)\> \big| \\
& \qquad\qquad +~ \big|\big\<\big[\delta^2 \mathcal{E}^\c(\bm{U}_n) - \delta^2\mathcal{E}^\c(0)\big]  I_{h,n}(\varphi_n(Z_n,\bm{t}_n)),(Z_n,\bm{t}_n)\big\>\big| \\
&\leq~  \sum_{\triple\tripleTau}\|V_{,\triple\tripleTau}(D \bm{U}_n) - V_{,\triple\tripleTau}(0) \|_{\ell^\infty({\rm supp}(D(Z_n,\bm{t}_n))} \\
&\qquad \qquad\qquad \cdot \|D_{\triple}((1-\varphi_n)(Z_n,\bm{t}_n))\|_{\ell^2(\mathbb{R}^d)}\|D_{\tripleTau}(Z_n,\bm{t}_n)\|_{\ell^2(\mathbb{R}^d)} \\
&+~ \sum_{\triple\tripleTau}\|V_{,\triple\tripleTau}\big(\nabla\bm{U}_n ) - V_{,\triple\tripleTau}(0)\|_{L^\infty({\rm supp}(\nabla(Z_n,\bm{t}_n))} \\
&\qquad\qquad\qquad \cdot \|\nabla_{\triple}((1-\varphi_n)(Z_n,\bm{t}_n))\|_{L^2(\mathbb{R}^d)}\|\nabla_{\tripleTau}(Z_n,\bm{t}_n)\|_{L^2(\mathbb{R}^d)}.
\end{split}
\end{equation*}

From Lemma~\ref{approx_lem} and the decay rates from Theorem~\ref{decay_thm1} we have
\begin{equation}\label{stew5}
\begin{split}
\|V_{,\triple\tripleTau}(D\bm{U}_n  ) - V_{,\triple\tripleTau}(0)\|_{\ell^\infty({\rm supp}(D(Z_n,\bm{t}_n))} \to~& 0,
\quad \text{and} \\
\|V_{,\triple\tripleTau}(\nabla \bm{U}_n) - V_{,\triple\tripleTau}(0)\|_{L^\infty({\rm supp}(\nabla(Z_n,\bm{t}_n))} \to~& 0.
\end{split}
\end{equation}
Consequently,
\[
\big\<\big[ \delta \mathcal{F}_{{\rm hom}, n }^{\bqcf}(\bm{U}_n) - \delta \mathcal{F}_{{\rm hom},n}^{\bqcf}(0) \big] (Z_n,\bm{t}_n),(Z_n,\bm{t}_n)\big \> \to 0,
\]
and from~\eqref{stew3},
\begin{equation}\label{inf_mod}
\liminf_{n \to \infty} \<\delta \mathcal{F}^{\bqcf}_{n}(\bm{U}_n)(Z_n,\bm{t}_n),(Z_n,\bm{t}_n)\> \geq \frac{3}{4}\gamma_\a \|(Z_n, \bm{t}_n)\|_{\rm ml}^2.
\end{equation}

Combining~\eqref{steam_peas} and~\eqref{inf_mod}, we can therefore conclude that
\begin{align}\label{stew100}
&\liminf_{n \to \infty} \<\delta \mathcal{F}^{\bqcf}_{n }(\Pi_{h,n}(\bm{U}_n))(X_n + Z_n,\bm{s}_n + \bm{t}_n),(X_n + Z_n,\bm{s}_n + \bm{t}_n)\> \nonumber\\
& \geq \liminf_{n \to\infty} \big[ \gamma_\a \|(X_n, \bm{s}_n)\|_{\rm ml}^2 + \frac{3}{4}\gamma_\a\|(Z_n, \bm{t}_n)\|_{\rm ml}^2  \big]  \\
& {\helen{\geq\liminf_{n \to\infty} \frac{3}{4}\gamma_\a\big[ \|\nabla X_n\|_{L^2(\mathbb{R}^d)}^2 +\|\bm{s}_n\|_{L^2(\mathbb{R}^d)}^2 + \|\nabla Z_n\|_{L^2(\mathbb{R}^d)}^2+\|\bm{t}_n\|_{L^2(\mathbb{R}^d)}^2 \big] }}.\nonumber
\end{align}
{\helen{
Notice that we have
\[
\begin{split}
\|\nabla W_n\|^2_{L^2(\mathbb{R}^d)}&=\la \nabla W_n, \nabla W_n\ra= \la \nabla (X_n+Z_n), \nabla (X_n+Z_n)\ra\\
&=\|\nabla X_n\|^2_{L^2(\mathbb{R}^d)}+2\la \nabla X_n, \nabla Z_n\ra+ \|\nabla Z_n\|^2_{L^2(\mathbb{R}^d)},
\end{split}
\]
so we get
\[
\|\nabla X_n\|^2_{L^2(\mathbb{R}^d)}+ \|\nabla Z_n\|^2_{L^2(\mathbb{R}^d)}
= \|\nabla W_n\|^2_{L^2(\mathbb{R}^d)}-2\la \nabla X_n, \nabla Z_n\ra.
\]
Applying the same treatments to $\|\bm{r}_{n}\|^2$, we have from \eqref{stew100} that
\[
\begin{split}
&\liminf_{n \to \infty} \<\delta \mathcal{F}^{\bqcf}_{n }(\Pi_{h,n}(\bm{U}_n))(X_n + Z_n,\bm{s}_n + \bm{t}_n),(X_n + Z_n,\bm{s}_n + \bm{t}_n)\> \\
&\geq \liminf_{n \to\infty} \frac{3}{4}\gamma_\a\Big[\|\nabla W_n\|^2_{L^2(\mathbb{R}^d)} - 2(\nabla Z_n, \nabla X_n)_{L^2(\mathbb{R}^d)} + \sum_\alpha\|r_n^\alpha\|^2_{L^2(\mathbb{R}^d)} \\
&\qquad\qquad\qquad\qquad\qquad\qquad\qquad\qquad- \sum_\alpha 2(s_n^\alpha,t_n^\alpha)_{L^2(\mathbb{R}^d)} \Big] \\
& \geq \liminf_{n \to\infty} \frac{3\gamma_\a}{4}\Big[\|\nabla W_n\|^2_{L^2(\mathbb{R}^d)} + \sum\|r_n^\alpha\|^2_{L^2(\mathbb{R}^d)}\Big] = \frac{3}{4}\gamma_\a,
\end{split}
\]
}}
which is a contradiction to~\eqref{contra_seq_def}.  In attaining the last equality, we have used $(\cdot, \cdot)_{L^2}$ to denote the $L^2$ inner product, and we have again used the fact that the inner product of {\helen{the strongly convergent sequence $\nabla X_n$ and weakly convergent sequence $\nabla Z_n$ (c.f. Lemma~\ref{seq_lemma}) converges to zero and similarly for the inner product of the strongly convergent $\bm{s}_n$ and weakly convergent $\bm{t}_n$.}}
\end{proof}

\section{Discussion}\label{discussion}

We presented the first complete error analysis of an atomistic-to-continuum coupling method for multilattices capable of incorporating defects in the analysis.  Our results for the blended force-based quasicontinuum method extend the existing results for Bravais lattices~\cite{blended2014}, with the striking conclusion that the convergence rates in the simple and multi-lattice cases coincide for the optimal mesh coarsening. Our computational results for a Stone-Wales defect in graphene confirm our theoretical predictions.

We have concerned ourselves here with the case of point defects, though we see no conceptually challenging obstacles to include dislocations in the analysis so long as there is an analogous decay result to Theorem~\ref{decay_thm}.  However, as previously mentioned, we are still limited in our ability to model physical effects such as bending or rippling in two-dimensional materials such as graphene due to several factors.  First, our assumption concerning stability of the multilattice, Assumption~\ref{assumption2} uses a norm, $\| \nabla I U\|_{L^2} + \|I\bm{p}\|_{L^2}$, which does not take any bending energy into account and so we do not guarantee our lattice is stable in this situation.  We could have of course formulated a different assumption using a discrete variant of $\|\nabla^2 U_3\|$ (where $U_3$ represents the out of plane displacement), but it is a very challenging question to extend the BQCF method and its analysis to such a situation.  The next issue that must be answered is what continuum model to use since the Cauchy--Born model used herein is not adequate to model such effects.  Possible alternatives would be to use higher-order Cauchy--Born rules~\cite{ericksen2008cauchy,yang2006generalized} which rely on higher-order strain gradients, or the so-called exponential Cauchy--Born rule~\cite{expCauchy}. In either of these cases, to use a similar analysis to what we have presented, one would have to establish new stress estimates akin to Corollary~\ref{globel_stress} as well ensuring that the continuum model chosen is stable provided the atomistic model is. We are also confronted with the problem of choosing a finite element space capable of approximating $H^2$ functions, which likewise challenges the analysis as well as the implementation. 

Finally, we remark that extensions to charged defects in ionic crystals, which represent a wide class of important multilattice crystals, represent yet another difficult challenge, largely due to the {\helen{long-range}} nature of the interatomic forces.
%%%%%%%%%%%%%%%%%%%%%%%%%%%%%%%%%%%%%%%%%%%%%%%%%%%%%%%%%%%%%%%%
\appendix

\section{Proofs of Convergence Lemmas}

The following elementary lemma will be used to construct the ``diagonal'' sequence alluded to when we first introduced Lemma~\ref{seq_lemma}.
\begin{lemma}\label{sub_lemma}
Let $\{\mathcal{L}_m^\alpha\}_{m =1}^\infty$ be a sequence of functions from $H^1(\mathbb{R}^d) \times (L^2(\mathbb{R}^d))^S \to V^\alpha$ for $V^\alpha$ a Banach space where $\alpha$ ranges in some finite index set $S$.  Let $(v_0,\bm{s}_0) \in H^1(\mathbb{R}^d) \times (L^2(\mathbb{R}^d))^S $ with $\eta$ the standard mollifier \dao{($\eta_j(x) = j^{-d} \eta(x/j)$)}, $v_j := \eta_{j} * v_0, \bm{s}_j := \eta_{j} * \bm{s}_0$, $j> 0$, and $f^\alpha:(H^1, (L^2)^S) \to V^\alpha$ continuous. Further assume that for each fixed $j > 0$, $\mathcal{L}^\alpha_m(v_j,\bm{s}_j) \to f^\alpha(v_j,\bm{s}_j)$ in $V^\alpha$ as $m \to \infty$ for each $\alpha$.  Then there exists a sequence $j_n \to 0$ such that
\[
\mathcal{L}^\alpha_{n}(v_{j_n},\bm{s}_{j_n}) \to f^\alpha(v_0, \bm{s}_0) \quad \mbox{for each $\alpha$}.
\]
Moreover, the sequence $j_n$ may be taken to satisfy $j_n \geq 1/\sqrt{R_{\a,n}}$ where $R_{\a,n}$ is taken from~\eqref{eq:stab:rescaling}
\end{lemma}

\begin{proof}
Fix $\alpha$ and $j>0$ and note that $\mathcal{L}^\alpha_m (v_{j},\bm{s}_{j}) \to f^\alpha(v_{j},\bm{s}_{j})$ as $m \to \infty$.  Thus we may choose $n_0^\alpha(j)$ large enough such that for $n \geq n_0^\alpha(j)$
\begin{equation}\label{pluto1}
\|\mathcal{L}^\alpha_{n} (v_{j},\bm{s}_{j}) - f^\alpha(v_{j},\bm{s}_{j})\|_{V^\alpha} \leq j.
\end{equation}
Since $\alpha$ belongs to a finite set $S$, we may define $n_0(j) = \max_\alpha n_0^\alpha(j)$.  Now define $j_n$ by
\begin{align*}
      j_n &= \max\{1, 1/\sqrt{R_{\a,1}}\} \quad \text{for } n = 1, \dots, n_0(1/2)-1, \\
      j_n &= \max\{1/2,1/\sqrt{R_{\a,2}}\}  \quad \text{for } n = n_0(1/2), \dots, n_0(1/3)-1, \\
      j_n &= \max\{1/m, 1/\sqrt{R_{\a,m}}\}  \quad \text{for } n = n_0(1/m), \dots, n_0(1/(m+1))-1.
   \end{align*}
For $n \geq n_0(1)$, we obtain from~\eqref{pluto1} that
\begin{equation}\label{charon}
\|\mathcal{L}^\alpha_{n} (v_{j_n},\bm{s}_{j_n}) - f^\alpha(v_{j_n},\bm{s}_{j_n})\|_{V^\alpha} \leq j_n \to 0.
\end{equation}
By continuity of $f^\alpha$, $f^\alpha(v_{j_n},\bm{s}_{j_n}) \to f^\alpha(v_{0},\bm{s}_{0})$, and so~\eqref{charon} implies the desired result.
\end{proof}

\begin{remark}
In the above proof, we made the requirement $j_n \geq 1/\sqrt{R_{\a,n}}$ so that $1/R_{\a,n} = \epsilon_{n} \leq j_n^2$, where $\epsilon_n$ is used in the scaling~\eqref{eq:stab:rescaling}.  This ensures that
\begin{align*}
&\|\epsilon_{n}\nabla (\eta_{j_n}*\hat{s}_0)\|_{L^2(\mathbb{R}^d)} =~ \epsilon_{n} \|(\nabla \eta_{j_n})*\hat{s}_0\|_{L^2(\mathbb{R}^d)} \\
&\leq~ \epsilon_{n}\|(\nabla \eta_{j_n})\|_{L^1(\mathbb{R}^d)}\|\hat{s}_0\|_{L^2(\mathbb{R}^d)} \quad \mbox{by Young's Inequality} \\
&=~  \epsilon_{n} \cdot 1/j_n\|\nabla \eta\|_{L^1(\mathbb{R}^d)}\|\hat{s}_0\|_{L^2(\mathbb{R}^d)} \\
&\leq~ j_n\|\nabla \eta\|_{L^1(\mathbb{R}^d)}\|\hat{s}_0\|_{L^2(\mathbb{R}^d)} \to 0.
\end{align*}
\end{remark}

The collection of operators and continuous functions that we apply this lemma to are the ones enumerated in Lemma~\ref{seq_lemma}, which may now prove.

\begin{proof}[Proof of Lemma~\ref{seq_lemma}]
We begin by noting \dao{that the properties of~\eqref{verge_result} are clear from the definitions of $W_n, X_n, Z_n, \bm{r}_n, \bm{s}_n, \bm{t}_n$ provided that we can show $\nabla X_n \to \nabla W_0$ and $\bm{s}_n \to \bm{r}_0$.}  Moreover, we can immediately note that with $\psi_m = 1 - \varphi_m$, $\nabla \psi_m$ and $\nabla^2 \psi_m$ are uniformly bounded on the compact set ${\rm supp}(\psi_m) \subset \epsilon_m B_{R_{\a,m}}(0) = B_1(0)$ by the assumptions on the blending function in Assumption~\ref{assumption-shapereg} and the definition of $\epsilon_n$. Thus, by Arzela-Ascoli and by replacing the original sequence by a subsequence if necessary, we may assume $\psi_m \to \psi_0$ in ${\rm C}^1$ for some $\psi_0 \in {\rm C}^1(B_1(0))$.

\dao{We now set about proving $\nabla X_n \to \nabla W_0$ and $\bm{s}_n \to \bm{r}_0$ and the remaining properties in Lemma~\eqref{seq_lemma}.  To that end, define}
\begin{align*}
\mathcal{L}^1_m:=~& \Pi_{h,m}, \, \quad f^1(W,\bm{s}) = (W,\bm{s}),\\
\mathcal{L}^2_m(W,\bm{s}) :=~& \epsilon_m^d \sum_{\xi \in \epsilon_m\mathcal{L}}\sum_{(\rho\alpha\beta)}\sum_{(\tau\gamma\delta)} V_{,(\rho\alpha\beta)(\tau\gamma\delta)} :D_{\tripleTau,m}( W,\bm{s})\otimes\frac{\rho}{\epsilon_m}\int_0^{\epsilon_m} \bar{\zeta}_m(\xi + t\rho - x),\\
f^2(W,\bm{s}) :=~& \sum_{(\rho\alpha\beta)}\sum_{(\tau\gamma\delta)} V_{,(\rho\alpha\beta)(\tau\gamma\delta)}: \nabla_{\tripleTau}( W(x), \bm{s}(x))\otimes\rho, \\
\mathcal{L}^3_m(W,\bm{s}) :=~& \epsilon_m^d \sum_{\xi \in \epsilon_m\mathcal{L}}\sum_{(\rho\alpha\beta)}\sum_{(\tau\gamma\delta)} V_{,(\rho\alpha\beta)(\tau\gamma\delta)}:(D_{\tripleTau,m}( W,\bm{s})) \bar{\zeta}_m(\xi- x), \\
f^3(W,\bm{s}) :=~& \sum_{(\rho\alpha\beta)}\sum_{(\tau\gamma\delta)} V_{,(\rho\alpha\beta)(\tau\gamma\delta)} :\nabla_{\tripleTau}( W(x), \bm{s}(x)),\\
\mathcal{L}^4_m(W,\bm{s}) :=~& \epsilon_m^d \sum_{\xi \in \epsilon_m\mathcal{L}}\sum_{(\rho\alpha\beta)}\sum_{(\tau\gamma\delta)} V_{,(\rho\alpha\beta)(\tau\gamma\delta)}: D_{\tripleTau,m}( \psi_m W, \psi_m\bm{s})\otimes\frac{\rho}{\epsilon_m}\int_0^{\epsilon_m} \bar{\zeta}_m(\xi + t\rho - x), \\
f^4(W,\bm{s}) :=~& \sum_{(\rho\alpha\beta)}\sum_{(\tau\gamma\delta)} V_{,(\rho\alpha\beta)(\tau\gamma\delta)} :\nabla_{\tripleTau}(\psi_0W(x),\psi_0\bm{s}(x))\otimes\rho, \\
\mathcal{L}^5_m(W,\bm{s}) :=~& \epsilon_m^d \sum_{\xi \in \epsilon_m\mathcal{L}}\sum_{(\rho\alpha\beta)}\sum_{(\tau\gamma\delta)} V_{,(\rho\alpha\beta)(\tau\gamma\delta)}: (D_{\tripleTau,m}(\psi_m W,\psi_m \bm{s})) \bar{\zeta}_m(\xi- x), \\
f^5(W,\bm{s}) :=~& \sum_{(\rho\alpha\beta)}\sum_{(\tau\gamma\delta)} V_{,(\rho\alpha\beta)(\tau\gamma\delta)}: \nabla_{\tripleTau}(\psi_0 W(x),\psi_0 \bm{s}(x)), \\
\mathcal{L}^6_m(W,\bm{s}) :=~& \epsilon_m^d \sum_{\xi \in \epsilon_m\mathcal{L}} \sum_{(\rho\alpha\beta)}\sum_{\tau\gamma\delta} V_{,(\rho\alpha\beta)(\tau\gamma\delta)}:D_{\triple,m}(\psi_m W, \psi_m \bm{s}):  D_{\tau\gamma\delta,m}(W,\bm{s}),  \\
f^6(W, s) :=~&  \int_{\mathbb{R}^d} \sum_{(\rho\alpha\beta)}\sum_{(\tau\gamma\delta)}V_{,\triple\tripleTau} :\big(\nabla_\rho (\psi_0 W) +  \psi_0(s^\beta - s^\alpha)\big):\big(\nabla_\tau W + s^\delta - s^\gamma\big)\, dx.
\end{align*}

If we can show that each of these satisfies the hypothesis of Lemma~\ref{sub_lemma}, then we may apply the conclusion of that lemma to deduce Lemmma~\ref{seq_lemma}. We focus primarily on $\mathcal{L}^4$ and $f^4$ and briefly touch on the other cases at the end.

We fix $j$ and set $W := v_j$ and $\bm{s} := \bm{s}_j$ and set $\mu_m(x) := \frac{1}{\epsilon_m}\int_0^{\epsilon_m} \bar{\zeta}_m(x + t\rho)\, dt$.  Then note that $\mu_m(\xi -x) = 0$ unless $|x-\xi| \lesssim~ \epsilon_m|\rho|$.  Hence
\begin{align*}
&\lim_{m \to \infty}\mathcal{L}^2_m(W,\bm{s})(x) \\
&~=~ \lim_{m\to\infty} \epsilon_m^d \sum_{\xi \in \epsilon_m\mathcal{L}}\sum_{(\rho\alpha\beta)}\sum_{(\tau\gamma\delta)} V_{,(\rho\alpha\beta)(\tau\gamma\delta)} :D_{\tripleTau,m}(\psi_m W(\xi),\psi_m\bm{s}(\xi))\otimes\rho\mu_m(\xi - x) \\
&~=~ \lim_{m \to \infty}\epsilon_m^d \sum_{\xi \in \epsilon_m\mathcal{L}}\sum_{(\rho\alpha\beta)}\sum_{(\tau\gamma\delta)} V_{,(\rho\alpha\beta)(\tau\gamma\delta)} :\left[D_{\tripleTau,m}(\psi_mW(\xi),\psi_m\bm{s}(\xi))-\nabla_{\tripleTau} (\psi_mW(x),\psi_m\bm{s}(x)) \right. \\
&\qquad\qquad\qquad \left. + \nabla_{\tripleTau} (\psi_mW(x),\psi_m\bm{s}(x)) \right]\otimes\rho\mu_m(\xi - x) \\
&~=~ \lim_{m \to \infty}\epsilon_m^d \sum_{\xi \in \epsilon_m\mathcal{L}}\sum_{(\rho\alpha\beta)}\sum_{(\tau\gamma\delta)} V_{,(\rho\alpha\beta)(\tau\gamma\delta)} :\left[D_{\tripleTau,m}(\psi_mW(\xi),\psi_m\bm{s}(\xi)) \right. \\
&\qquad\qquad\qquad \left. -~\nabla_{\tripleTau} (\psi_mW(x),\psi_m\bm{s}(x))\right]\otimes\rho\mu_m(\xi-x) \\
&\qquad     +\lim_{m \to \infty}  \sum_{(\rho\alpha\beta)}\sum_{(\tau\gamma\delta)} V_{,(\rho\alpha\beta)(\tau\gamma\delta,m)} :\nabla_{\tripleTau} (\psi_mW(x),\psi_m\bm{s}(x))\otimes\rho \epsilon_m^d\sum_{\xi \in \epsilon_m\mathcal{L}}\mu_m(\xi-x) \\
&~=~ \lim_{m \to \infty}\epsilon_m^d \sum_{\xi \in \epsilon_m\mathcal{L}}\sum_{(\rho\alpha\beta)}\sum_{(\tau\gamma\delta)} V_{,(\rho\alpha\beta)(\tau\gamma\delta)} :\left[D_{\tripleTau,m}(\psi_mW(\xi),\psi_m\bm{s}(\xi)) \right. \\
&\qquad\qquad\qquad \left. -~\nabla_{\tripleTau} (\psi_mW(x),\psi_m\bm{s}(x))\right]\otimes\rho\mu_m(\xi-x) \\
&\qquad    +\sum_{(\rho\alpha\beta)}\sum_{(\tau\gamma\delta)} V_{,(\rho\alpha\beta)(\tau\gamma\delta,m)}: \nabla_{\tripleTau} (\psi_0W(x),\psi_0\bm{s}(x))\otimes\rho,
\end{align*}
since $\epsilon_m^d\sum_{\xi \in \epsilon_m\mathcal{L}}\mu_m(\xi-x) = 1$ and since $\psi_m \to \psi_0$ in ${\rm C}^1$.
%
%&~=~ \lim_{m \to \infty}\epsilon_m^d \sum_{\xi \in \epsilon_m\mathcal{L}}\sum_{(\rho\alpha\beta)}\sum_{(\tau\gamma\delta)} V_{,(%\rho\alpha\beta)(\tau\gamma\delta)} \left[D_{\tripleTau,m}(W(\xi),\bm{s}(\xi)) - D_{\tripleTau,m}(W(x),\bm{s}(x)) \right. \\
%&\qquad \left. + D_{\tripleTau,m}(W(x),\bm{s}(x))-\nabla_{\tripleTau} (W(x),\bm{s}(x))\right] \otimes\rho\mu_m(\xi-x) \\
%&\qquad\qquad + \sum_{(\rho\alpha\beta)}\sum_{(\tau\gamma\delta,m)} V_{,(\rho\alpha\beta)(\tau\gamma\delta,m)} \nabla_{\tripleTau} (W(x),\bm{s}%(x))\otimes\rho
%\end{align*}
For the first limit above, we note that $\mu_m(\xi -x) = 0$ unless $|x-\xi| \lesssim~ \epsilon_m|\rho|$ implies
\begin{align*}
&\left[D_{\tripleTau,m}(\psi_mW(\xi),\psi_m\bm{s}(\xi)) - \nabla_{\tripleTau} (\psi_mW(x),\psi_m\bm{s}(x))\right]  = \mathcal{O}(\epsilon_m)\\
%&=~ D_{\tau,m}(\psi_m W)(\xi) - \nabla_\tau (\psi_m W)(x) + (\psi_m s^\delta)(\xi+\epsilon_m\rho) - (\psi_m s^\delta)(x) + (\psi_m s^\gamma)(x)- (\psi_m s^\gamma)(\xi) \\
%&=~ \mathcal{O}(\epsilon_m).
\end{align*}
Thus
\begin{align*}
&\lim_{m \to \infty}\epsilon_m^d \sum_{\xi \in \epsilon_m\mathcal{L}}\sum_{(\rho\alpha\beta)}\sum_{(\tau\gamma\delta)} V_{,(\rho\alpha\beta)(\tau\gamma\delta)} :\left[D_{\tripleTau,m}(\psi_mW(\xi),\psi_m\bm{s}(\xi)) \right. \\
&\qquad\qquad\qquad\qquad\qquad\qquad \left. -~\nabla_{\tripleTau} (\psi_mW(x),\psi_m\bm{s}(x))\right]\otimes\rho\mu_m(\xi-x) \\
&=~ \lim_{m \to \infty}\sum_{(\rho\alpha\beta)}\sum_{(\tau\gamma\delta)} V_{,(\rho\alpha\beta)(\tau\gamma\delta)} :\left[\mathcal{O}(\epsilon_m)\right]\otimes\rho\epsilon_m^d \sum_{\xi \in \epsilon_m\mathcal{L}}\mu_m(\xi-x) \\
&=~\lim_{m \to \infty}\sum_{(\rho\alpha\beta)}\sum_{(\tau\gamma\delta)} V_{,(\rho\alpha\beta)(\tau\gamma\delta)}  :\left[\mathcal{O}(\epsilon_m)\right]\otimes\rho = 0.
\end{align*}
We have thus shown that
\[
\lim_{m \to \infty}\mathcal{L}^4_m(W,\bm{s}) = \sum_{(\rho\alpha\beta)}\sum_{(\tau\gamma\delta)} V_{(\rho\alpha\beta)(\tau\gamma\delta)}: \nabla_{\tripleTau}(\psi_0 W(x), \psi_0 \bm{s}(x))\otimes\rho = f^4(W, \bm{s}).
\]
The proof for $\mathcal{L}^5$ proceeds in exactly the same manner, and the proofs for $\mathcal{L}^2$ and $\mathcal{L}^3$ are likewise very similar with the exception that $\psi_m$ is no longer present.

For $\mathcal{L}^6$, using a Taylor expansion and the fact that $\nabla^2 \psi_m$ is uniformly bounded in $m$ by Assumption~\ref{assumption-shapereg},
\begin{align*}
D_{\tau,m}\psi_m(\xi) =~& D_{\tau,m}\psi_m(\xi) - \nabla_\tau \psi_m(\xi) + \nabla_\tau \psi_m(\xi) \\
=~& \mathcal{O}(\epsilon_m) +  \nabla_\tau \psi_m(\xi) \to \nabla_\tau \psi_0(\xi).
\end{align*}
Furthermore, since $W$ and $\bm{s}$ are smooth,  $D_{\tau,m}W$ converges to $\nabla_\tau W$ in $L^2$ and $\ell^2$ on compact subsets, and the same holds for $D_{\tau,m} s_\beta$ converging to $\nabla_\tau s_\beta$.  Consequently,
\[
D_{\tripleTau,m}(\psi_m W,\psi_m\bm{s}) \to \nabla_{\tripleTau}(\psi_0 W, \psi_0 \bm{s}) \quad \mbox{in $L^2$ and $\ell^2$ on bounded sets.}
\]
Convergence of $\mathcal{L}^6(W,\bm{s})$ to $f^6(W,\bm{s})$ now follows from this and convergence of the quadrature rule
\[
\lim_{m\rightarrow \infty}\epsilon_m^d \sum_{\xi \in \epsilon_m\mathcal{L}} \sum_{(\rho\alpha\beta)}\sum_{\tau\gamma\delta} V_{,(\rho\alpha\beta)(\tau\gamma\delta)}:\nabla_{\triple,m}(\psi_0 W, \psi_0 \bm{s}):  \nabla_{\tau\gamma\delta,m}(W,\bm{s}),
\]
to
\[
\int_{B_1(0)} \sum_{(\rho\alpha\beta)}\sum_{\tau\gamma\delta} V_{,(\rho\alpha\beta)(\tau\gamma\delta)}:\nabla_{\triple,m}(\psi_0 W, \psi_0 \bm{s}):  \nabla_{\tau\gamma\delta,m}(W,\bm{s})\, dx.
\]
Lastly, we note that $\mathcal{L}^1(W,\bm{s}) \to f^1(W,\bm{s})$ as a result of first approximating $(W,\bm{s})$ by smooth functions having support contained in $B_{\epsilon_m^{-\gamma}}(0)$ for some $0 < \gamma < 1$ and then using standard interpolation estimates for smooth functions along with the mesh growth assumption of Assumption~\ref{assumption-shapereg}.  Specifically, let $W_m,\bm{s}_m \in {\rm C}_0^\infty$ such that $\nabla W_m \to \nabla W, \nabla^2 W_m \to \nabla^2 W$, $\bm{s}_m \to \bm{s}$, and $\nabla \bm{s}_m \to \nabla \bm{s}$ in $L^2$ where $\nabla W_m$ and $\bm{s}_m$ have support in $B_{\epsilon_m^{-\gamma}}(0)$.  Then
\begin{equation}\label{stab_of_pi}
\begin{split}
\|\Pi_{h,m}(W,\bm{s}) - (W,\bm{s})\|_{\rm ml} \leq~& \|\Pi_{h,m}(W,\bm{s}) - \Pi_{h,m}(W_{m}, \bm{s}_m)\|_{\rm ml} + \|\Pi_{h,m}(W_m,\bm{s}_m) - (W_{m}, \bm{s}_m)\|_{\rm ml} \\
&\qquad+ \|(W_{m},\bm{s}_m) - (W,\bm{s})\|_{\rm ml} \\
\lesssim~& \|(W,\bm{s}) - (W_{m},\bm{s}_m)\|_{\rm ml} + \|\Pi_{h,m}(W_m,\bm{s}_m) - (W_{m}, \bm{s}_m)\|_{\rm ml},
\end{split}
\end{equation}
after using stability of $\Pi_{h,m}$ with respect to the ${\rm ml}$ norm.  By our choice of $(W_m, \bm{s}_m)$, it follows that
\begin{equation}\label{to_zero1}
\|(W,\bm{s}) - (W_{m},\bm{s}_m)\|_{\rm ml}  \to 0.
\end{equation}
Using the definition of the ${\rm ml}$ norm and standard interpolation estimates, we also possess
\begin{equation}\label{app_stand}
\begin{split}
\|\Pi_{h,m}(W_m,\bm{s}_m) - (W_{m}, \bm{s}_m)\|_{\rm ml} \lesssim~& \|h_{m} \nabla^2 W_m\|_{L^2(\mathbb{R}^d)} + \|h_m \nabla \bm{s}_m\|_{L^2(\mathbb{R}^d)}.
\end{split}
\end{equation}
Since, $\nabla W_m$ and $\bm{s}_m$ have support in $B_{\epsilon_m^{-\gamma}}(0)$ and since we have the mesh growth $|h_m(x)| \lesssim \epsilon_m(\epsilon_m|x/\epsilon_m|)^s$ from our choice of scaling and Assumption~\ref{assumption-shapereg}, it follows that $|h_m(x)| \lesssim \epsilon_m^{1-s\gamma}$ on $B_{\epsilon_m^{-\gamma}}(0)$.  It then follows from~\eqref{app_stand} that
\begin{equation}\label{to_zero2}
\|\Pi_{h,m}(W_m,\bm{s}_m) - (W_{m}, \bm{s}_m)\|_{\rm ml} \lesssim~ \epsilon_m^{1-s\gamma}\|\nabla^2 W_m\|_{L^2(\mathbb{R}^d)} + \epsilon_m^{1-s\gamma}\|\nabla\bm{s}_m\|_{L^2(\mathbb{R}^d)} \to 0,
\end{equation}
for appropriately chosen $\gamma$.  Applying the results~\eqref{to_zero1} and~\eqref{to_zero2} to~\eqref{stab_of_pi} then produces
\[
\|\Pi_{h,m}(W,\bm{s}) - (W,\bm{s})\|_{\rm ml}  \to 0,
\]
as desired.

Then as a consequence of Lemma~\ref{sub_lemma} applied to each of these operators, we obtain a sequence $j_n \to 0$ which satisfies the desired convergence results of~\eqref{seq_lemma}.
\end{proof}

\begin{proof}[Proof of Lemma~\ref{more_lemma}]
To prove the first convergence result~\eqref{more_1}, we note by Lemma~\ref{iso_lemma}, there exists $Y_n = \acute{(\psi_n Z_n)}$ such that $\psi_n Z_n(\xi) = \big(\bar{\zeta}_n * \bar{Y}_n\big)(\xi) =: \big(Y_n^*\big)(\xi)$ for all $\xi \in \epsilon_n\mathcal{L}$, and similarly there exists $\bm{b}_n = \acute{(\psi_n \bm{t}_n)}$ such that $\psi_n \bm{t}_n(\xi) = \big(\bar{\zeta}_n *\bar{\bm{b}}_n\big)(\xi) := \big(\bm{b}_n^*\big)(\xi)$.  We further note that from the proof of~\cite[Lemma 4.10, Step 3]{blended2014} and the fact that $\nabla (\psi_n Z_n) \weakto 0$, we also have $\nabla \bar{Y}_n \weakto 0$.  To show that $\bar{b}_n^\alpha \weakto 0$, observe for a smooth function $\mu$ with compact support that
% \begin{align*}
% \lim_{n\to\infty}\int \bar{\acute{t}}_n^\alpha \cdot \mu =~& \lim_{n\to\infty}\int \bar{\acute{t}}_n^\alpha \cdot (\bar{\zeta}_n *\mu) + \lim_{n\to\infty}\int \bar{\acute{t}}_n^\alpha \cdot (\mu -\bar{\zeta}_n *\mu) \\
% =~& \lim_{n\to\infty}\int (\bar{\acute{t}}_n^\alpha *\bar{\zeta}_n) \cdot \mu \\
% =~& \lim_{n\to\infty}\sum_{\xi \in \epsilon_n\mathcal{L}}\int_{\nu_\xi} (\bar{\acute{t}}_n^\alpha *\bar{\zeta}_n) \cdot \mu \\
% =~& \lim_{n\to\infty}\sum_{\xi \in \epsilon_n\mathcal{L}}\int_{\nu_\xi} \big(\bar{\acute{t}}_n^\alpha *\bar{\zeta}_n(\xi) + \mathcal{O}(\epsilon_n)\big) \cdot \mu \\
% =~& \lim_{n\to\infty}\sum_{\xi \in \epsilon_n\mathcal{L}}\int_{\nu_\xi} t_n^\alpha(\xi) \cdot \big(\mu(\xi) + \mathcal{O}(\epsilon_n)\big) \\
% =~& \lim_{n\to\infty}\epsilon_n^d\sum_{\xi \in \epsilon_n\mathcal{L}}t_n^\alpha(\xi) \cdot \mu(\xi).
% \end{align*}
% \todo{Then since $t_n^\alpha$ is piecewise linear and $t_n^\alpha \weakto 0$, this latter expression also converges to zero. (This last statement uses that $\epsilon_n \nabla t_n^\alpha$ also converges weakly to zero.)}
\begin{align*}
\lim_{n\to\infty}\int \bar{b}_n^\alpha \cdot \mu &= \lim_{n\to\infty}\int \bar{b}_n^\alpha \cdot (\bar{\zeta}_n *\mu) + \lim_{n\to\infty}\int \bar{b}_n^\alpha \cdot (\mu -\bar{\zeta}_n *\mu) \\
&= \lim_{n\to\infty} \int (\bar{b}_n^\alpha *\bar{\zeta}_n) \cdot \mu
= \lim_{n \to \infty} \int t_n^\alpha \big( \psi_n \mu \big).
%=: \lim_{n \to \infty} a_n.
\end{align*}
From Lemma~\ref{seq_lemma}, we have $\psi_n \to \psi_0$ in $L^2$ (since these functions are compactly supported), and this latter expression is then an inner product of a weakly convergent sequence ($t_n^\alpha \weakto 0$) and strongly convergent sequence which therefore converges to zero.
%From the choice of scalings in \eqref{eq:stab:rescaling} and the assumptions on $\varphi$ it is readily seen that $\psi_n$ is bounded in $C^2$.
%Let $(a_{n_j})$ be any convergence subsequence of $(a_n)$, then there exists a further subsequence $(a_{n_j'})$ such that $\psi_{n_j'} \mu$ converges strongly in $L^2$. Since $t_n^\alpha \rightharpoonup 0$ it follows that $a_{n_j'} \to 0$. Since the subsequence was arbitrary it finally follows that $a_n \to 0$.

% Since $\bar{\acute{t}}_n^\alpha *\bar{\zeta}_n = \psi_n t_n^\alpha \weakharpoonup 0$

Having established that $\nabla \bar{Y}_n, \bar{b}_n^\alpha \rightharpoonup 0$, we can use the functions ${\rm S}^{\rm def}_n(x)$ and ${\rm S}^{\rm shift}_n(x)$ from Lemma~\ref{seq_lemma} to write
\begin{align*}
&\epsilon^d \sum_{\xi \in \epsilon_n\mathcal{L}} \sum_{(\rho\alpha\beta)}\sum_{(\tau\gamma\delta)} V_{,(\rho\alpha\beta)(\tau\gamma\delta)}:D_{(\rho\alpha\beta),n}(\psi_n {Z}_n,\psi_n \bm{t}_n):D_{(\tau\gamma\delta),n}({X}_n,\bm{s}_n) \\
& \qquad \qquad =~ \int{\rm S}^{\rm def}_n(x):(\nabla \bar{Y}_n + \epsilon_n\nabla \bar{b}^\beta_n) + \int{\rm S}^{\rm shift}_n(x) : (\bar{b}^\beta_n - \bar{b}^\alpha_n) \to 0,
\end{align*}
using the strong convergence of ${\rm S}^{\rm def}_n(x)$ and ${\rm S}^{\rm shift}_n(x)$ from Lemma~\ref{seq_lemma} and the weak convergence: $\nabla \bar{Y}_n, \bar{b}_n^\alpha \rightharpoonup 0$.

The second convergence result~\eqref{more_2} is proven in nearly an identical manner by choosing $Y_n = \acute{Z}_n$ and $\bm{b}_n = \acute{\bm{t}}_n$ such that $Z_n = \bar{\zeta}_n * \bar{\acute{Z}}_n$ and $\bm{t}_n = \bar{\zeta}_n *\bar{\acute{\bm{t}}}_n$ and using the convergence results for ${\rm R}^{\rm def}_n(x)$ and ${\rm R}^{\rm shift}_n(x)$ from Lemma~\ref{seq_lemma}.
\end{proof}

\begin{proof}[Proof of Lemma~\ref{weak_lemma}]
Observe
\begin{equation}\label{base1}
\begin{split}
&\lim_{n\to\infty}\epsilon_n^d \sum_{\xi \in \epsilon_n\mathcal{L}} \sum_{(\rho\alpha\beta)}\sum_{(\tau\gamma\delta)} V_{,(\rho\alpha\beta)(\tau\gamma\delta)}:D_{(\rho\alpha\beta),n}(\theta^2_nZ_n,\theta^2_n\bm{t}_n):D_{(\tau\gamma\delta),n}(Z_n,\bm{t}_n), \\
&=\lim_{n\to\infty}\epsilon_n^d \sum_{\xi \in \epsilon_n\mathcal{L}} \sum_{(\rho\alpha\beta)}\sum_{(\tau\gamma\delta)} V_{,(\rho\alpha\beta)(\tau\gamma\delta)}:\left(D_{\rho,n}(\theta^2_nZ_n) + \epsilon_n D_{\rho,n}(\theta^2_nt^\beta_n) - \theta^2_n(t^\alpha_n - t^\beta_n)\right):
\left(D_{\tau,n}(Z_n) + \right. \\
&\qquad\qquad\qquad\qquad\qquad\qquad\qquad\qquad \left.\epsilon_nD_{\tau,n}(t^\delta_n) - (t^\gamma_n-t^\delta_n)\right) \\
&=~ \lim_{n\to\infty}\epsilon_n^d \sum_{\xi \in \epsilon_n\mathcal{L}} \sum_{(\rho\alpha\beta)}\sum_{(\tau\gamma\delta)} V_{,(\rho\alpha\beta)(\tau\gamma\delta)}: D_{\rho,n}(\theta^2_nZ_n):D_{\tau,n}(Z_n) \\
&+~ \lim_{n\to\infty}\epsilon_n^d \sum_{\xi \in \epsilon_n\mathcal{L}} \sum_{(\rho\alpha\beta)}\sum_{(\tau\gamma\delta)} V_{,(\rho\alpha\beta)(\tau\gamma\delta)}:D_{\rho,n}(\theta^2_nZ_n):\left(\epsilon_n D_{\tau,n}(t^\delta_n) - (t^\gamma_n-t^\delta_n)\right) \\
&+~ \lim_{n\to\infty}\epsilon_n^d \sum_{\xi \in \epsilon_n\mathcal{L}} \sum_{(\rho\alpha\beta)}\sum_{(\tau\gamma\delta)} V_{,(\rho\alpha\beta)(\tau\gamma\delta)}:(T_{\rho,n}(\theta^2_nt^\beta_n) - \theta^2_nt^\alpha_n):D_{\tau,n}(Z_n) \\
&+~ \lim_{n\to\infty}\epsilon_n^d \sum_{\xi \in \epsilon_n\mathcal{L}} \sum_{(\rho\alpha\beta)}\sum_{(\tau\gamma\delta)} V_{,(\rho\alpha\beta)(\tau\gamma\delta)}:(T_{\rho,n}(\theta^2_nt^\beta_n) - \theta^2_nt^\alpha_n):(\epsilon_nD_{\tau,n}(t^\delta_n) - (t^\gamma_n-t^\delta_n)).
\end{split}
\end{equation}
This gives four terms to manipulate, which we label in order as $A_1^n, A_2^n, A_3^n$, and $A_4^n$.  The first of these is, after using the product rule for finite differences (and the associated notation $T_{\rho, n} f(\xi) = f(\xi + \epsilon_n \rho)$),
\begin{equation}\label{soupy1}
\begin{split}
&\lim_{n \to \infty}A_1^n\\
&=~\lim_{n \to \infty}\epsilon_n^d \sum_{\xi \in \epsilon_n\mathcal{L}} \sum_{(\rho\alpha\beta)}\sum_{(\tau\gamma\delta)} V_{,(\rho\alpha\beta)(\tau\gamma\delta)}:
\big\{\theta_nZ_n D_{\rho,n}(\theta_n):D_{\tau,n}(Z_n) + T_{\rho,n}(\theta_n)D_{\rho,n}(\theta_nZ_n):D_{\tau,n}(Z_n)\big\}.
\end{split}
\end{equation}
Recall that $Z_n \to 0$ in $L^2$ on $B_1(0) \supset {\rm supp}(\theta_n)$ and hence also in $\ell^2$ due to being piecewise linear.  Therefore, continuing the limit in~\eqref{soupy1},
\begin{equation}\label{soupy15}
\begin{split}
&\lim_{n \to \infty}A_1^n\\
&=~ \lim_{n \to \infty}\epsilon_n^d \sum_{\xi \in \epsilon_n\mathcal{L}} \sum_{(\rho\alpha\beta)}\sum_{(\tau\gamma\delta)} V_{,(\rho\alpha\beta)(\tau\gamma\delta)}:\big\{T_{\rho,n}(\theta_n)D_{\rho,n}(\theta_nZ_n):D_{\tau,n}(Z_n) \\
&\qquad+~ D_{\rho,n}(\theta_nZ_n):D_{\tau,n}(\theta_n)Z_n\big\} \\
&=~ \lim_{n \to \infty}\epsilon_n^d\sum_{\xi \in \epsilon_n\mathcal{L}} \sum_{(\rho\alpha\beta)}\sum_{(\tau\gamma\delta)} V_{,(\rho\alpha\beta)(\tau\gamma\delta)}:\big\{D_{\rho,n}(\theta_nZ_n):(T_{\rho,n}(\theta_n)-T_{\tau,n}(\theta_n))D_{\tau,n}(Z_n) \\
&\qquad+~  D_{\rho,n}(\theta_nZ_n):T_{\tau,n}(\theta_n)D_{\tau,n}(Z_n) + D_{\rho,n}(\theta_nZ_n):D_{\tau,n}(\theta_n)Z_n \big\}\\
&=~ \lim_{n \to \infty}\epsilon_n^d \sum_{\xi \in \epsilon_n\mathcal{L}} \sum_{(\rho\alpha\beta)}\sum_{(\tau\gamma\delta)} V_{,(\rho\alpha\beta)(\tau\gamma\delta)}:\big\{D_{\rho,n}(\theta_nZ_n):D_{\tau,n}(\theta_nZ_n)\big\},
\end{split}
\end{equation}
where in arriving at the last line we used the fact that $(T_{\rho,n}(\theta_n)-T_{\tau,n}(\theta_n)) \to 0$ in $L^\infty$ and $H^1$ boundedness of $Z_n$.

By replacing $Z_n$ with $\epsilon_n t^\delta_n$, we have may write the second of the terms in~\eqref{base1} as
\begin{equation}\label{soup25}
\begin{split}
&\lim_{n \to \infty} A_2^n \\
&=~ \lim_{n \to \infty}\epsilon_n^d \sum_{\xi \in \epsilon_n\mathcal{L}} \sum_{(\rho\alpha\beta)}\sum_{(\tau\gamma\delta)} V_{,(\rho\alpha\beta)(\tau\gamma\delta)}:\big\{D_{\rho,n}(\theta_nZ_n):D_{\tau,n}(\epsilon_n \theta_n t^\delta_n ) + D_{\rho,n}(\theta^2_nZ_n):(t^\delta_n - t^\gamma_n)\big\} \\
&=~ \lim_{n \to \infty}\epsilon_n^d \sum_{\xi \in \epsilon_n\mathcal{L}} \sum_{(\rho\alpha\beta)}\sum_{(\tau\gamma\delta)} V_{,(\rho\alpha\beta)(\tau\gamma\delta)}:\big\{D_{\rho,n}(\theta_nZ_n):D_{\tau,n}(\epsilon_n \theta_n t^\delta_n ) \\
&\qquad + (D_{\rho,n}(\theta_n)(\theta_nZ_n)+ T_{\rho,n}(\theta_n)D_{\rho,n}(\theta_nZ_n)):(t^\delta_n - t^\gamma_n)\big\} \\
&=~ \lim_{n \to \infty}\epsilon_n^d \sum_{\xi \in \epsilon_n\mathcal{L}} \sum_{(\rho\alpha\beta)}\sum_{(\tau\gamma\delta)} V_{,(\rho\alpha\beta)(\tau\gamma\delta)}:\big\{D_{\rho,n}(\theta_nZ_n):D_{\tau,n}(\epsilon_n \theta_n t^\delta_n ) \\
&\qquad + D_{\rho,n}(\theta_nZ_n):(T_{\rho,n}(\theta_n) - \theta_n + \theta_n)(t^\delta_n - t^\gamma_n)\big\} \\
&=~ \lim_{n \to \infty}\epsilon_n^d \sum_{\xi \in \epsilon_n\mathcal{L}} \sum_{(\rho\alpha\beta)}\sum_{(\tau\gamma\delta)} V_{,(\rho\alpha\beta)(\tau\gamma\delta)}:\big\{D_{\rho,n}(\theta_nZ_n):\big(D_{\tau,n}(\epsilon_n \theta_n t^\delta_n ) + \theta_n(t^\delta_n - t^\gamma_n)\big)\big\}.
\end{split}
\end{equation}

The third term from~\eqref{base1} is, after using the previously established convergence properties of $Z_n, \bm{t}_n$ and $\theta_n$,
\begin{equation}\label{soup5}
\begin{split}
&\lim_{n \to \infty}A_3^n\\
&=~ \lim_{n \to \infty}\epsilon_n^d \sum_{\xi \in \epsilon_n\mathcal{L}} \sum_{(\rho\alpha\beta)}\sum_{(\tau\gamma\delta)} V_{,(\rho\alpha\beta)(\tau\gamma\delta)}:(T_{\rho,n}(\theta^2_nt^\beta_n) - \theta^2_nt^\alpha_n):D_{\tau,n}(Z_n) \\
&=~ \lim_{n \to \infty}\epsilon_n^d \sum_{\xi \in \epsilon_n\mathcal{L}} \sum_{(\rho\alpha\beta)}\sum_{(\tau\gamma\delta)} V_{,(\rho\alpha\beta)(\tau\gamma\delta)}:((T_{\rho,n}(\theta_n)-\theta_n)T_{\rho,n}(\theta_nt^\beta_n) + \theta_nT_{\rho,n}(\theta_n t^\beta_n) - \theta^2_nt^\alpha_n):D_{\tau,n}(Z_n) \\
&=~ \lim_{n \to \infty}\epsilon_n^d \sum_{\xi \in \epsilon_n\mathcal{L}} \sum_{(\rho\alpha\beta)}\sum_{(\tau\gamma\delta)} V_{,(\rho\alpha\beta)(\tau\gamma\delta)}:\left\{(T_{\rho,n}(\theta_nt^\beta_n) - \theta_nt^\alpha_n):\theta_nD_{\tau,n}(Z_n) \right\} \\
&=~ \lim_{n \to \infty}\epsilon_n^d \sum_{\xi \in \epsilon_n\mathcal{L}} \sum_{(\rho\alpha\beta)}\sum_{(\tau\gamma\delta)} V_{,(\rho\alpha\beta)(\tau\gamma\delta)}:\left\{(T_{\rho,n}(\theta_nt^\beta_n) - \theta_nt^\alpha_n):\big(T_{\tau,n}(\theta_n)D_{\tau,n}(Z_n) + D_{\tau,n}(\theta_n)(Z_n)\big) \right\} \\
&=~ \lim_{n \to \infty}\epsilon_n^d \sum_{\xi \in \epsilon_n\mathcal{L}} \sum_{(\rho\alpha\beta)}\sum_{(\tau\gamma\delta)} V_{,(\rho\alpha\beta)(\tau\gamma\delta)}:\left\{(T_{\rho,n}(\theta_nt^\beta_n) - \theta_nt^\alpha_n):D_{\tau,n}(\theta_n Z_n)\right\}
\end{split}
\end{equation}

After again using the convergence properties of $\theta_n$, it is then straightforward to show
\begin{equation}\label{soup9}
\begin{split}
&\lim_{n \to \infty} A_4^n \\
&=~ \lim_{n \to \infty}\epsilon_n^d \sum_{\xi \in \epsilon_n\mathcal{L}} \sum_{(\rho\alpha\beta)}\sum_{(\tau\gamma\delta)} V_{,(\rho\alpha\beta)(\tau\gamma\delta)}:(T_{\rho,n}(\theta_nt^\beta_n) - \theta_nt^\alpha_n):(T_{\tau,n}(\theta_nt^\delta_n) - \theta_nt^\gamma_n).
\end{split}
\end{equation}
The equations~\eqref{soupy15},~\eqref{soup25},~\eqref{soup5}, and~\eqref{soup9} applied to~\eqref{base1} give the conclusion of the Lemma.
\end{proof}

\begin{lemma}\label{p1_lemma}
Let $W_n, r_n$, and $\varphi_n$ be defined as in the proof of Theorem~\ref{stab_theorem} with $I_{h,n}$ the $\mathcal{P}_1$ interpolant onto $\mathcal{T}_{h,n}$.  Then
\begin{equation} \label{app_p1_lemma}
   \begin{split}
   & \|\nabla I_{h,n}(\varphi_n W_n) - \nabla (\varphi_n W_n)\|_{L^2(\mathbb{R}^d)} \to~ 0, \\
   & \|I_{h,n}(\varphi_n r^\alpha_n) - (\varphi_n r^\alpha_n)\|_{L^2(\mathbb{R}^d)} \to~ 0.
   \end{split}
\end{equation}
\end{lemma}

\begin{proof}
Let $T$ be any element of $\mathcal{T}_n$ and $I_T$ the linear interpolant onto this triangle.  Then
\begin{equation}\label{appetizer}
\begin{split}
   \| \nabla I_T (\varphi_n W_n) - \nabla (\varphi_n W_n) \|_{L^2(T)}
   & \lesssim h_T \| \nabla^2 (\varphi_n W_n) \|_{L^2(T)} \\
   &  \lesssim h_T \| \nabla^2 \varphi_n \otimes W_n \|_{L^2(T)}
            + h_T \| \nabla \varphi_n \otimes \nabla W_n \|_{L^2(T)} \\
   & \lesssim h_T \| W_n \|_{L^2(T)} + h_T \| \nabla W_n \|_{L^2(T)}.
\end{split}
\end{equation}
We recall that $\nabla \varphi_n$ has support on $B_1(0)$ due to our choice of scaling and that $\int_{B_1(0)} W_n\dx = 0$ due to our choice of equivalence class representative.  Because of this and the estimate~\eqref{appetizer}, it then follows that
\begin{align*}
   \| \nabla I_T (\varphi_n W_n) - \nabla (\varphi_n W_n) \|_{L^2(\mathbb{R}^d)} \lesssim~& h_T \|W_n\|_{L^2(B_1(0))} + h_T \| \nabla W_n \|_{L^2(B_1(0))} \\
	\lesssim~& h_T \| \nabla W_n \|_{L^2(B_1(0))},
\end{align*}
by the Poincar\'{e} inequality.  Since $h_T = \mathcal{O}(\epsilon_n)$ on $L^2(B_1(0))$ due to the full mesh refinement assumption, we obtain the first result.

For the second one, the argument is the similar except that we have $h_T^2$, we need not use the Poincar\'{e} inequality, and we do not immediately have that $\nabla r$ is bounded.  However, we note that
\begin{align*}
   h_T^2 \| \nabla r_n^\alpha \|_{L^2(T)} &\lesssim h_T^2 |T|^{1/2} \| \nabla r_n^\alpha \|_{L^\infty(T)}  \\
   &\lesssim h_T |T|^{1/2} \| r_n^\alpha \|_{L^\infty(T)}
   \lesssim h_T \| r_n^\alpha \|_{L^2(T)},
\end{align*}
which could be used to obtain the result for $r_n^\alpha$.
\end{proof}

\section{Notation}\label{sec:appnotation}
This section summarizes notation used in the manuscript.

\begin{itemize}

\item $\mathcal{L}$ --- a Bravais lattice

\item $\mathcal{M}$ --- a multilattice

\item $\mathcal{S} = \{0, \ldots, S-1\}$ --- the index set of atomic species

\item $\xi$ --- an element of $\mathcal{L}$ or $\epsilon\mathcal{L}$ for $\epsilon > 0$.

\item $\alpha,\beta,\gamma,\delta, \iota, \chi$ --- indexes denoting atomic species

\item $\rho, \tau, \sigma \in \mathcal{L}$ --- vectors between lattice sites

\item $\mathcal{R}$ --- an interaction range whose elements are triples of the form $\triple \in \mathcal{L} \times \mathcal{S} \times \mathcal{S}$

\item $\mathcal{R}_1  := \{ \rho \in \mathcal{L} : \exists \triple \in \mathcal{R}\}$ --- projection of $\mathcal{R}$ onto lattice direction

\item $r_{\rm cut} := \max \{ |\rho| :  \triple \in \mathcal{R} \}$ --- a finite cut-off distance

\item $r_{\rm cell}$ --- the radius of the smallest ball inscribing the unit cell of $\mathcal{L}$

\item $r_{\rm buff} := \max\{ r_{\rm cut}, r_{\rm cell}\}$

\item $\bm{u} = \big(u_\alpha\big)_{\alpha = 0}^{S-1}$ --- vector of displacements of all species of atoms

\item $(U, \bm{p})$ --- displacement/shift description defined by $U = u_0$ and $p_\alpha = u_\alpha - u_0$

\item $\bm{y}^{\rm ref}$ and $\bm{p}^{\rm ref}$ --- the reference deformation and shifts

\item $D_{\triple} \bm{u}(\xi) = u_\beta(\xi + \rho) - u_\alpha(\xi), D_{\triple}(U,\bm{p}) = U(\xi+\rho) - U(\xi) + p_\beta(\xi + \rho) - p_\alpha(\xi)$

\item $D\bm{u}(\xi) = \big(D_{\triple}\bm{u}(\xi)\big)_{\triple \in \mathcal{R}}, D(U,\bm{p})(\xi) = \big(D_{\triple}(U,\bm{p})(\xi)\big)_{\triple \in \mathcal{R}}$

\item $\hat{V}_\xi(D\bm{y}(\xi))$ and $V_\xi(D\bm{u})$ --- site potentials defined on deformations and displacements, respectively

\item $\mathcal{E}^\a(\bm{u})$ and $\mathcal{E}^\a_{\rm hom}(\bm{u})$ --- energy difference functionals for defective and defect free lattice.

\item $\mathcal{T}_{\a}$ --- atomistic scale finite element mesh of triangles in $2D$ and tetrahedra in $3D$

\item $\bar{\zeta}(x), \bar{\zeta}_\xi(x) = \bar{\zeta}(x-\xi)$ --- nodal basis function of $\mathcal{T}_\a$ associated with the origin and $\xi$ respectively

\item $\omeRho(x):= \int_{0}^1 \barZeta(x+t\rho)dt$ --- an auxiliary function

\item $Iu_\alpha, IU, Ip_\alpha$ or $\bar{u}_\alpha, \bar{U}, \bar{p}_\alpha$ --- a piecewise linear interpolant with respect to $\mathcal{T}_\a$

\item $\tilde{I}u_\alpha, \tilde{I}U, \tilde{I}p_\alpha$ or $\tilde{u}_\alpha, \tilde{U}, \tilde{p}_\alpha$ --- a ${\rm C}^{2,1}$ interpolant with respect to $\mathcal{T}_\a$

\item $u^*(x) := (\bar{\zeta} * \bar{u})(x)$ --- quasi-interpolant of $u$ defined through convolution

\item $|\cdot|$ --- meaning depends on context:  $|\cdot|$ is $\ell^2$ norm of a vector, matrix, higher order tensor, or finite difference stencil. $|T|$ is area or volume of element $T$ in a finite element partition, $|\gamma|$ is the order of a multiindex.

\item $\|\cdot\|_{\ell^2(A)}$ --- $\ell^2$ norm over a set $A$.  If $f:A \to \mathbb{R}^d$ is a vector-valued function, $\|f\|_{\ell^2(A)} = (\sum_{\alpha \in A}|f(\alpha)|^2)^{1/2}$.

\item $\|\cdot \|_{\a}$ --- norm on admissible displacements defined by  $\|\bm{u}\|_\a^2 :=~ \sum_{\alpha = 0}^{S-1}\|\nabla Iu_\alpha\|_{L^2(\mathbb{R}^d)}^2 + \sum_{\alpha \neq \beta}\| Iu_\alpha - Iu_\beta\|_{L^2(\mathbb{R}^d)}^2$.

\item $\bm{\mathcal{U}}$ ---space of admissible displacements defined by
\[
\mathcal{U} :=~ \left\{ \bm{u} = (u_\alpha)_{\alpha = 0}^{S-1} : u_\alpha:\mathcal{L} \to \mathbb{R}^n, \|\bm{u}\|_\a < \infty \right\}/\mathbb{R}^n
\]

\item $\bm{\mathcal{U}}_0$ ---space of test displacements defined by
\[
\left\{(U, \bm{p}) : {\rm supp}(\nabla IU), \, p_0 \equiv 0, \,  \mbox{and} \, {\rm supp}(Ip_\alpha) \, \mbox{are compact}\right\}/\mathbb{R}^n
\]

\item $\Omega$ --- a finite polygonal domain

\item $\varphi$ ---the blending function

\item $\Omega_\a :=~ {\rm supp}(1-\varphi) + B_{2r_{\rm buff}}$ --- the atomistic domain

\item $\Omega_{\rm b} :=~ {\rm supp}(\nabla \varphi) + B_{2r_{\rm buff}}$ --- the blending region

\item $\Omega_\c :=~ {\rm supp}(\varphi) \cap \Omega + B_{2r_{\rm buff}}$ --- the continuum region

\item $\Omega_{\rm core} :=~ \Omega\setminus\Omega_\c$ --- the defect core region

\item $\mathcal{T}_h$ --- the (coarse) finite element mesh on $\Omega$

\item $h(x) := \max_{T: x\in T} {\rm Diam}(T)$ --- the mesh size function

\item $R_{\rm t} := \inf_{R} \{R > 0: \Omega_{\rm t} \subset B_{R}(0)\}$ --- an exterior measure of a domain $\Omega_{\rm t}$

\item $r_{\rm t} := \sup_{r}\{r > 0: B_{r}(0) \subset \Omega_{\rm t}\}$ --- an interior measure of a domain $\Omega_{\rm t}$

\item $R_{\rm o} := \inf_{R} \{R > 0: \Omega \subset B_{R}(0)\}$  --- an exterior measurement of $\Omega$

\item $r_{\rm i} := \sup_{r}\{r > 0: B_{r}(0) \subset \Omega\}$ --- an interior measurement of $\Omega$

\item $\Omega_{\rm ext} := \mathbb{R}^d\setminus B_{r_{\rm i}/2}(0)$ --- exterior of $\Omega$

\item $I_h$--- the standard piecewise linear nodal interpolant on $\mathcal{T}_h$

\item $S_{h}$ --- the Scott-Zhang quasi-interpolant on $\mathcal{T}_h$.

\item $W_{\rm CB}(U,\bm{p})$ --- Cauchy--Born strain energy density function

\item $\mathcal{E}^\c(U,\bm{p})$ --- Cauchy--Born energy functional

\item $\mathcal{U}_h :=~ \left\{u \in {\rm C}^0(\Omega) : u|_{T} \in \mathcal{P}_1(T), \quad \forall \, T \in \mathcal{T}_h\right\}$ --- a finite element space

\item $\bm{\mathcal{U}}_h :=~ \mathcal{U}_h / \mathbb{R}^n$ space of admissible finite element displacements

\item $\mathcal{U}_{h,0} :=~ \left\{u \in {\rm C}^0(\mathbb{R}^d): u|_{T} \in \mathcal{P}_1(T), \quad \forall \, T \in \mathcal{T}_h, u = 0 \mbox{ on } \mathbb{R}^d\setminus\Omega \right\}$ --- finite element space satisfying homogeneous boundary conditions

\item $\bm{\mathcal{U}}_{h,0} :=~ \mathcal{U}_{h,0}/ \mathbb{R}^n$ --- finite element quotient space

\item $\bm{\mathcal{P}}_{h,0} :=~ \{0\}   \times  (\mathcal{U}_{h,0})^{S-1}$ --- finite element space for shifts

\item $\|(U, \bm{p})\|_{\rm ml}^2 := \|\nabla U \|_{L^2(\mathbb{R}^d)}^2 + \sum_{\alpha = 0}^{S-1}\|p_\alpha\|^2_{L^2(\mathbb{R}^d)} = \|\nabla U \|_{L^2(\mathbb{R}^d)}^2 + \|\bm{p}\|^2_{L^2(\mathbb{R}^d)}$ --- norm on finite element spaces

\item $\|\bm{p}\|_{L^p} := \sum_{\alpha = 0}^{S-1}\|p_\alpha\|_{L^p}, \|\nabla \bm{p}\|_{L^p} := \sum_{\alpha = 0}^{S-1}\|\nabla p_\alpha\|_{L^p}$

\item $\eta(x)$ --- a smooth bump function or standard mollifying function depending on the context

\item $T_{R}u_\alpha(x) = \eta(x/R)\bigg(Iu_\alpha - \frac{1}{|A_R|}\int\limits_{A_R} Iu_0\, dx\bigg)$ --- a truncation operator

\item $\Pi_{h} u_\alpha := S_h (T_{r_{\rm i}}u_\alpha)$ --- an projection operator from discrete displacements to finite element displacements

\item $\Pi_{h} p_\alpha := \Pi_{h} (u_\alpha - u_0)$ ---- a projection operator on shifts

\item $[{\Scd}(U,\bm{q})(x)]_{\beta}$ and $[{\Scs}(U,\bm{q})(x)]_{\alpha\beta}$--- continuum stress function associated with displacements and shifts

\item $[\Sad(U,\bm{q})(x)]_\beta$ and $[\Sas(U,\bm{q})(x)]_{\alpha\beta}$ ---atomistic stress function associated with displacements and shifts

\item $V_{,\triple\tripleTau}\big( \cdot \big):v:w :=~ w^{\transpose}\big[V_{,\triple\tripleTau}\big( \cdot \big)\big]v \quad \forall v,w \in \mathbb{R}^n$

\item $\mathbb{C} : D(W,\bm{q}): D(Z,\bm{r}) :=~ \sum_{\triple \in \mathcal{R}} \sum_{\tripleTau \in \mathcal{R}} V_{,\triple\tripleTau}:D_{\triple}(W,\bm{q}):D_{\tripleTau}(Z,\bm{r})$

\item $\mathbb{C} : \nabla (W,\bm{q}): \nabla (Z,\bm{r}) :=~ \sum_{\triple \in \mathcal{R}} \sum_{\tripleTau \in \mathcal{R}} V_{,\triple\tripleTau}:(\nabla_\rho W + q_\beta - q_\alpha):(\nabla_\tau Z + r_\beta- r_\alpha)$

\item $\varphi_n$ --- a sequence of blending functions

\item $\psi_n := 1-\varphi_n$

\item $\theta_n := \sqrt{\psi_n}$

\item $B_{r}, B_{r}(x)$ --- Ball of radius $r$ about the origin or ball of radius $r$ about $x$.

\item ${\rm supp}(f)$ --- support of a function $f$.

\item ${\rm Diam}(U)$ --- diameter of the set $U$ measured with the Euclidean norm.

\item $(\mathbb{R}^n)^{\calR}$ --- direct product of vectors with $|\calR|$ terms.

\item $\mathstrut^\transpose$ --- transpose of a matrix.

\item $\otimes$ --- tensor product.

\item $\nabla^j$ --- $j$th derivative of a function defined on $\mathbb{R}^d$.

\item $\partial_\gamma$ --- multiindex notation for derivatives.

\item $L^p(U)$ --- Standard Lebesgue spaces.

\item $W^{k,p}(U)$ --- Standard Sobolev spaces.

\item $W^{k,p}_{\rm loc}(U) = \left\{f:U \to \mathbb{R}^d : f \in W^{k,p}(V) \, \forall V \subset\subset U  \right\}$.

\item $H^k(U) = W^{k,2}(U)$, $H^1_0(U) = \left\{f \in H^k(U) : \mbox{Trace}(f) = 0~\, \mbox{on} \,~ \partial U \right\}$.

\item ${\rm C}^{k}$ --- space of $k$ times continuously differentiable functions

\item $\dashint_{U} f\, dx$ --- average value of $f$ over $U$.

\end{itemize}

%%%%%%%%%%%%%%%%%%%%%%%%%%%%%%%%%%%%%%%%%%%%%%%%%%%%%%%%%%%%%%%%%%%%%%
\newpage
\bibliographystyle{plain}	% (uses file "plain.bst")
\bibliography{myrefs}		% expects file "myrefs.bib"

\end{document}